%% file: tese-exemplo.tex
\documentclass[12pt,twoside,a4paper]{book}

\usepackage[T1]{fontenc}
\usepackage[british,UKenglish,USenglish,english,american]{babel}
\usepackage[pdftex]{graphicx}           
\usepackage{setspace}                   
\usepackage{indentfirst}                
\usepackage{makeidx}                    
\usepackage[nottoc]{tocbibind}          
\usepackage{courier}                    
\usepackage{type1cm}                    
\usepackage{listings}                   
\usepackage{titletoc}
\usepackage[fixlanguage]{babelbib}
\usepackage[font=small,format=plain,labelfont=bf,up,textfont=it,up]{caption}
\usepackage[usenames,svgnames,dvipsnames]{xcolor}
\usepackage[a4paper,top=2.54cm,bottom=2.0cm,left=2.0cm,right=2.54cm]{geometry} 
\usepackage[pdftex,plainpages=false,pdfpagelabels,pagebackref,colorlinks=true,citecolor=DarkGreen,linkcolor=NavyBlue,urlcolor=DarkRed,filecolor=green,bookmarksopen=true]{hyperref} 
\usepackage[all]{hypcap}                
\usepackage[square,sort,nonamebreak,comma]{natbib}  
\fontsize{60}{62}\usefont{OT1}{cmr}{m}{n}{\selectfont}

\usepackage{amsthm}
\usepackage{amsfonts}
\usepackage{amsmath}
\usepackage{amssymb}
\usepackage{mathrsfs}
\usepackage{lmodern}
\usepackage{euscript}
\usepackage{relsize}
\usepackage{url}
\theoremstyle{plain}
\usepackage{mathtools}
\usepackage{thmtools}
\usepackage{enumitem}
  
\newtheorem{teo}{Theorem}[chapter]
\newtheorem{prop}[teo]{Proposition}
\newtheorem{lemma}[teo]{Lemma}
\newtheorem{cor}[teo]{Corollary}

\theoremstyle{definition}
\newtheorem{mydef}[teo]{Definition}

\newtheorem{ex}[teo]{Example}

\theoremstyle{remark}
\newtheorem{fact}[teo]{Fact}
\newtheorem{remark}[teo]{Remark}
\newtheorem{remarks}[teo]{Remarks}
\newtheorem{claim}{Claim}
\newtheorem{assumption}{Assumption}

\renewcommand{\qedsymbol}{$\blacksquare$}
\newcommand{\defeq}{\mathrel{\mathop:}=}
\newcommand{\g}{\mathbb{G}}
\newcommand{\zd}{\mathbb{Z}^d}
\newcommand{\zdp}{\mathbb{Z}_{+}^d}
\newcommand{\ds}{\mathcal{A}^{\zd}}
\newcommand{\gds}{\mathcal{A}^{\g}}
\newcommand{\gr}{\mathfrak{T}}

\newcommand*{\QEDB}{\hfill\ensuremath{\square}}
\DeclareRobustCommand{\rchi}{{\mathpalette\irchi\relax}}
\newcommand{\irchi}[2]{\raisebox{\depth}{$#1\chi$}} 

\usepackage{tikz}
\usetikzlibrary{tikzmark}
\usepackage{fancyhdr}
\renewcommand{\chaptermark}[1]{\markboth{\MakeUppercase{#1}}{}}

\renewcommand{\restriction}{\mathord{\upharpoonright}}
\graphicspath{{./figuras/}}             
\makeindex                              
\fontsize{60}{62}\usefont{OT1}{cmr}{m}{n}{\selectfont}
\cleardoublepage
\normalsize

\begin{document}
\frontmatter 
\fancyhead[RO]{{\footnotesize\rightmark}\hspace{2em}\thepage}
\setcounter{tocdepth}{2}
\fancyhead[LE]{\thepage\hspace{2em}\footnotesize{\leftmark}}
\fancyhead[RE,LO]{}
\fancyhead[RO]{{\footnotesize\rightmark}\hspace{2em}\thepage}

\onehalfspacing  

\thispagestyle{empty}
\begin{center}
    \vspace*{2.3cm}
    \textbf{\Large{Gibbs Measures on
 Subshifts}}\\
    
    \vspace*{1.2cm}
    \Large{Bruno Hideki Fukushima Kimura}
    
    \vskip 2cm
    \Large{
    A Thesis Submitted \\[-0.15cm]
    to \\[-0.15cm]
    Instituto de Matemática e Estatística\\[-0.15cm]
    of the\\[-0.15cm]
    University of São Paulo\\
    in partial fulfillment of the requirements\\[-0.15cm]
    for the\\[-0.15cm]
    Degree of Master of Science\\
    }
    
    \vskip 2cm
    Program: Applied Mathematics\\
    Supervisor: Prof. Dr. Rodrigo Bissacot\\
    
   	\vskip 1cm
    \normalsize{During the development of this work the author received financial support from CNPq}
    
    \vskip 0.5cm
    \normalsize{São Paulo, August 2015}
\end{center}

%
%
%




%
%
%
%
\newpage
\thispagestyle{empty}
    \begin{center}
        \vspace*{2.3 cm}
        \textbf{\Large{Gibbs Measures on
 Subshifts}}\\
        \vspace*{2 cm}
    \end{center}

   \vskip 2cm

    \begin{flushright}
	Esta versão da dissertação contém as correções e alterações sugeridas\\
	pela Comissão Julgadora durante a defesa da versão original do trabalho,\\
	realizada em 28/08/2015. Uma cópia da versão original está disponível no\\
	Instituto de Matemática e Estatística da Universidade de São Paulo.

    \vskip 2cm

    \end{flushright}
    \vskip 4.2cm

    \begin{quote}
    \noindent Comissão Julgadora:
    
    \begin{itemize}
		\item Prof. Dr. Rodrigo Bissacot Proença (orientador) - IME-USP 
		\item Prof. Dr. Anatoli Iambartsev - IME-USP 
		\item Prof. Dr. Leandro Martins Cioletti - UnB 
    \end{itemize}
      
    \end{quote}
\pagebreak

\pagenumbering{roman}     

\chapter*{Acknowledgments}
Thanks.

\chapter*{Resumo}

\noindent KIMURA, B. H. F. \textbf{Medidas de Gibbs em subshifts}. 
2015. 89 f.
Dissertação de Mestrado - Instituto de Matemática e Estatística,
Universidade de São Paulo, São Paulo, 2015.
\\

Nós estudamos as propriedades de medidas de Gibbs para funções com variação $d$-somável 
definidas em um subshift $X$. Baseado no trabalho de Meyerovitch \cite{meyerovitch:13}, provamos que se $X$ é um subshift de tipo finito (STF), então qualquer
medida de equilíbrio é também uma medida de Gibbs. Embora a definição fornecida por Meyerovitch não
faz qualquer menção à esperanças condicionais, mostramos que no caso em que $X$ é um STF, é possível caracterizar estas medidas
em termos de 
noções mais familiares apresentadas na literatura (por exemplo, \cite{cap:76},\cite{georgii:11},\cite{ruelle:04}).
\\

\noindent \textbf{Palavras-chave:} Medidas de Gibbs, medidas de equilíbrio, subshifts.

\chapter*{Abstract}
\noindent KIMURA, B. H. F. \textbf{Gibbs measures on subshifts}. 
2015. 89 p.
Master's Thesis - Instituto de Matemática e Estatística,
University of São Paulo, São Paulo, 2015.
\\

We study the properties of Gibbs measures for functions with $d$-summable variation 
defined on a subshift $X$. Based on Meyerovitch's work \cite{meyerovitch:13}, we prove that if $X$ is a subshift of finite type (SFT), then any
equilibrium measure is also a Gibbs measure. Although the definition provided by Meyerovitch does not
make any mention to conditional expectations, we show that in the case where $X$ is a SFT it is possible to characterize these measures 
in terms of more familiar
notions presented in the literature (e.g. \cite{cap:76},\cite{georgii:11},\cite{ruelle:04}).
\\

\noindent \textbf{Keywords:} Gibbs measures, equilibrium measures, subshifts.

\tableofcontents    

\chapter{List of Abbreviations}
\begin{tabular}{ll}
         a.e.       & almost everywhere\\ 
         DLR        & Dobrushin-Lanford-Ruelle\\ 
         m.p.d.s.   & measure preserving dynamical system\\
         SFT        & subshift of finite type\\ 
         SFTs       & subshifts of finite type\\
         t.d.s.     & topological dynamical system \\ 

\end{tabular}

\chapter{List of Symbols}
\begin{tabular}{ll}
       $a$ & Typical symbol from the alphabet $\mathcal{A}$\\
		$\alpha, \beta$ & Typical $\mu$-partitions of $X$ \\
		$\alpha \approx \beta$ & Equivalence of $\mu$-partitions \\
		$\alpha\vee \beta$ & Common refinement of $\alpha$ and $\beta$ \\
		$\mathcal{A}$ & Finite alphabet\\
        $\gds$ & $\g$-full shift \\
        $\mathcal{A}^{\Lambda}$ & Set of all functions from $\Lambda$ into $\mathcal{A}$\\
        $\mathsf{Aut}(X)$ & Set of all Borel automorphisms of $X$ \\
        $\mathcal{B}$ & $\sigma$-algebra of subsets of $X$ \\
        $B(x,r)$ & Open ball of radius $r$ centered at the point $x$ in $\ds$\\
        $\beta \succeq \alpha$ & Finer $\mu$-partition \\
        $\mathcal{C}$ & Collection of all Borel subsets of $\mathcal{R}$ \\ 
        $d$ & Dimension of the lattice \\
        $\delta_{n}(f)$ & n-th variation of $f$ \\
        $\Delta, \Lambda$ & Typical subsets of $\g$ \\
        $D_{\mu,\mathcal{R}}$ & Radon-Nikodym derivative of $\mu$ with respect to $\mathcal{R}$\\
		$\emptyset$ & Empty set \\
        $\EuScript{F}$ & Sub-$\sigma$-algebra of $\mathcal{B}$ \\
        $\mathcal{F}$ & Collection of patterns\\
		$\mathcal{F}(X)$ & Group of Borel automorphisms that generates $\gr^{0}$\\
        $\g$ & The infinite lattice $\zd$ or $\zdp$ \\
       $h_{\mu}(T)$ & Entropy of a m.p.d.s. \\
       $h_{\mu}(T,\alpha)$ & Dynamical entropy of a m.p.d.s. relative to $\alpha$ \\
       $H_{\mu}(\alpha)$ & Entropy of $\alpha$ \\
       $\mathsf{Homeo}(X)$ & Set of all homeomorphisms of $X$ \\
        $i,j,k$ & Typical elements of $\g$ \\
        $\mathsf{id}_{X}$ & Identity mapping of $X$ \\
        $I_{\alpha}$ & Information function of $\alpha$ \\
        $I_{\alpha|\EuScript{F}}$ & Conditional information function of $\alpha$ given $\EuScript{F}$ \\
        $l,m,n$ & Integers \\
\end{tabular}

\pagebreak
\pagebreak
\begin{tabular}{ll}
        $\Lambda_{n}$ & Open box about the origin of radius $n$ \\
        $\textnormal{Loc}(X)$ & Set of all real-valued local functions on $X$ \\
        $\textnormal{Loc}^\mathbb{Q}(X)$ & Set of all rational-valued local functions on $X$ \\ 
        $\mathcal{M}(T)$ & Set of all $T$-invariant Borel probability measures on $X$\\
        $\mathbb{N}$ & Set of all positive integers \\	
        $\nu_l$ & Left counting measure of $\mu$\\
        $\nu_r$ & Right counting measure of $\mu$\\
        $[\omega]$ & Cylinder with configuration $\omega$\\
        $\phi$ & $\mathcal{R}$-cocycle\\            
        $\phi_f$ & $\gr$-cocycle defined by $f$\\
		$\pi_j$ & Projection of $\gds$ onto the $j$-th coordinate \\
		$\pi_l$ & Left projection of $\mathcal{R}$\\
        $\pi_r$ & Right projection of $\mathcal{R}$\\  
     	 $p(f)$ &  Pressure of $f$\\
     	 $\mathcal{P}(\mathcal{A})$ & Power set of $\mathcal{A}$ \\	
        $\mathbb{Q}$      & Set of all rational numbers \\
        $\mathbb{R}$ & Set of all real numbers \\
        $\overline{\mathbb{R}}$ & Set of all extended real numbers \\
        $\mathcal{R}_G$ & Orbit relation of $G$ \\
        $\mathcal{R}(x)$ & Equivalence class of $x$ \\
        $\mathcal{R}(A)$ & $\mathcal{R}$-saturation of $A$ \\
        $\mathscr{S}_A$ & Collection of all finite subsets of $A$\\
        $SV_d(X)$ & Space of all functions with $d$-summable variation on $X$ \\
        $\sigma^{j}$ & Shift or translation by $j$\\
		$\Sigma_{A_{1},\dots,A_{d}}$ & Matrix subshift\\
        $\gr$ & Gibbs relation \\
        $\gr^0$ & Topological Gibbs relation \\
		$\tau^{\g}$ & Prodiscrete topology on $\gds$\\
        $\theta$ & Flip map \\
        $\mathsf{X}_{\mathcal{F}}$ & Subshift associated to $\mathcal{F}$\\
       $X_\Lambda$ & Set of all restrictions of elements of $X$ to the set $\Lambda$ \\
        $\mathbb{Z}$      & Set of all integers \\
        $\mathbb{Z}_{+}$      & Set of all nonnegative integers \\
        $\zd$ & $d$-dimensional integer lattice \\
        $\zdp$ & $d$-dimensional nonnegative integer lattice \\
        $\|\cdot\|$ & Maximum ``norm'' on $\g$ \\
        $\|\cdot\|_\infty$ & Supremum norm \\
        $\|\cdot\|_{SV_d}$ & Norm of the space $SV_d(X)$ \\
\end{tabular}


\mainmatter

\fancyhead[RE,LO]{\thesection}

\onehalfspacing            

\input cap-introducao        
\input cap-conceitos         
\input cap-conclusoes        

\renewcommand{\chaptermark}[1]{\markboth{\MakeUppercase{\appendixname\ \thechapter}} {\MakeUppercase{#1}} }
\fancyhead[RE,LO]{}
\appendix

\include{ape-conjuntos}      

\backmatter \singlespacing   
\bibliographystyle{alpha-ime}
\bibliography{bibliografia}  

\nocite{meyerovitch:13}
\nocite{aaronson:07}
\nocite{feldman:77}
\nocite{ruelle:04}
\nocite{kechris:95}
\nocite{cohn:13}
\nocite{georgii:11}
\nocite{lind:95}
\nocite{schmidt:02}
\nocite{hunter:07}
\nocite{schmidtlind:02}
\nocite{schpet}
\nocite{schmidt:97}
\nocite{sri:98}
\nocite{lanfordrob:68}
\nocite{quas:00}
\nocite{biss:12}
\nocite{keller:98}
\nocite{israel:79}
\nocite{lanfordruelle:69}
\nocite{dob:68}
\nocite{simon:93}
\nocite{bowen:08}
\nocite{bovier1}
\nocite{bovier2}
\nocite{ny:08}
\nocite{sarig:09}
\nocite{gaans:03}
\nocite{viana}
\nocite{ward}
\nocite{walters}
\nocite{pjfernandez}
\nocite{cap:76}
\nocite{red:00}
\nocite{ruellehaydn:92}
\nocite{ruelle:92}
\nocite{haydn:92}
\nocite{muir:11}
\nocite{muir}
\nocite{pollicott}
\nocite{klenke}
\nocite{martin}
\nocite{Bogo:49}
\nocite{Bogo:69}
\nocite{sinai:72}
\nocite{minlos:67}
\nocite{Fer}
\nocite{Cio:14}

\index{TBP|see{periodicidade região codificante}}
\index{DSP|see{processamento digital de sinais}}
\index{STFT|see{transformada de Fourier de tempo reduzido}}
\index{DFT|see{transformada discreta de Fourier}}
\index{Fourier!transformada|see{transformada de Fourier}}

\printindex   

\end{document}

%% file: cap-introducao.tex
\chapter{Introduction}
\label{cap:intro}

The theory of Gibbs measures is one of the most successful and developed branches of mathematics motivated by ideas from physics. 
A great number of specialists in rigorous statistical mechanics, ergodic theory and symbolic dynamics are or were engaged in topics related to 
some notion of Gibbsianess.

Historically, the first paper with a rigorous treatment on this subject dates back to \cite{Bogo:49} (see also \cite{Bogo:69} for a more up to date exposition) 
by N.N. Bogolyubov and B.I. Khatset. Following the same ideas as presented in this paper R.L. Dobrushin \cite{dob:68} and,  independently, O.E. Lanford with 
D. Ruelle \cite{lanfordruelle:69} introduced the notion of Gibbsianess in the context of statistical mechanics by means of conditional probabilities. Due to its 
physical content and probabilistic interpretation this approach is widely adopted until today both in mathematical physics and probability theory. Such measures are often referred to
 as DLR measures in honor to them. 

On the other hand, these papers together with one by R.A. Minlos \cite{minlos:67} motivated the study of Gibbs measures in (differentiable) dynamical systems
started by Ya. G.  Sinai \cite{sinai:72}. Sinai introduced Markov partitions and symbolic dynamics for Anosov diffeomorphisms, subjects for which R. Bowen
made several contributions, for example, one of the main references for Gibbs measures in symbolic dynamics is the book \cite{bowen:08}. By the 
influence of Ruelle, the notion was also introduced for $\mathbb{Z}^d$-actions on compact metrizable spaces by Capoccaccia in \cite{cap:76}. 
Ruelle also wrote one of the classical books \cite{ruelle:04} towards to Gibbs measures focusing on dynamic aspects, this book (together with Bowen's) are used by the ergodic theory community working in subfield today known as thermodynamic formalism.

So, the notion of Gibbsianess developed by scientists working on this boundary between mathematical physics and dynamical systems, 
together with a several number of papers and books published at the 70's apparently joined the areas. However, the approach adopted by the 
communities to handle with Gibbs measures split in two different ways: while the probabilicists and mathematical physicists follow Dobrushin's ideas 
and the majority think in terms of conditional expectations and thermodynamic limits (see one of the classical modern books c.f. Georgii \cite{georgii:11}), 
the dynamicists follow the approach introduced by Bowen, Ruelle and Capocaccia.

In addition, K. Schmidt and K. Petersen \cite{schpet} studied an abstract notion of Gibbs measures for one-dimensional SFTs (over finite alphabets) whose connection 
with the previous definitions is not obvious, although they used the same name. In 2013, T. Meyerovitch generalized this definition for multidimensional 
subshifts (over finite alphabets), and proved that for SFTs any equilibrium measure for a potential belonging to suitable class of functions (functions with $d$-summable variation) 
is also a Gibbs measure. This result generalizes another one presented in \cite{ruelle:04}, since it expands the class of potential in which this result holds. 

The communities mentioned above know that all these definitions do not always coincide, see \cite{Fer} and \cite{sarig:15} for some examples from the 
probabilistic and dynamic point of view, respectively. Nevertheless, there exist classes of shifts and potentials for which the equivalence of these several notions of Gibbsianess holds. 
One classical reference for positive results, that is, showing the equivalence of some definitions is provided by Keller's book \cite{keller:98}. Inspired by the Capocaccia's definition of 
Gibbs measures he showed that a definition used by the dynamical systems community for the full shift over a finite alphabet in $\mathbb{Z}^d$ coincides with the notion of 
DLR measure for the class of potentials with $d$-summable variation. In particular, for the one-dimensional case, he proved that these measures are 
Gibbs in the Bowen's sense. Recently, S. Muir extended the results obtained by Keller showing that a natural extension of the Capocaccia's and the DLR 
definitions coincide when the configuration space is $\mathbb{N}^{\mathbb{Z}^d}$.

Once in $d=1$ the existence of the Ruelle operator (a standard tool in one-dimensional thermodynamic formalism) is ensured, L. Cioletti and A. O. Lopes \cite{Cio:14} 
showed the equivalence of some notions of Gibbs measures for Walters potentials defined on the full shift over a finite alphabet, such as: DRL measures, 
measures constructed with the Ruelle operator and thermodynamic limits measures. 

The thesis is organized as follows: we dedicate Chapter \ref{cap:shifts} to introduce one of the main objects of study in this text, the so-called shift spaces or subshifts. In statistical mechanics,
we commonly deal with the most simple kind of a shift space, the full shift. In this context, the full shift may be interpreted essentially as the configuration space
of a system of spins arranged at the sites of a countably infinite lattice (generally the $d$-dimensional integer lattice $\zd$), where these these spins are restricted to a finite set. An example of a 
full shift is the set $\{-1,+1\}^{\zd}$ which describes the configuration space used by the most famous model in statistical physics employed to
explain the phenomenon of ferromagnetism, the Ising model (see \cite{georgii:11}).
In addition to its importance in the study of lattice models in
statistical mechanics, the study of shift spaces constitute a beautiful branch of mathematics known as symbolic dynamics. Most of the results presented 
in Chapter \ref{cap:shifts} were based on the masterpiece written by Lind and Marcus \cite{lind:95}.

In order to provide in Chapter \ref{cap:gibbs} a connection between Gibbs and equilibrium measures on subshifts, we devote Chapter \ref{cap:fm} to present a few preliminary results about
the thermodynamic formalism. These elements allow us to define the entropy $h_{\mu}(T)$ of a measure preserving dynamical system $(X,\mathcal{B},\mu,T)$. This quantity
describes the maximum amount of information per unit of time that can be gained about the system (with time evolution determined by $T$) at a measurement process.   
Almost all examples given in Chapter \ref{cap:fm} concerns the entropy of Bernoulli shifts, for further examples see \cite{walters},\cite{pollicott},\cite{viana}.  
Later, we define an equilibrium measure as being an invariant probability measure that is distinguished by means of a variational principle, in the sense that this measure maximizes a certain quantity of the type ``entropy + energy''. 
More precisely, we let the pressure of a potential $f$ be described by

\begin{equation}
p(f) = \sup\limits_{\mu \in \mathcal{M}(T)}\left\{h_{\mu}(T) + \int_{X} f\,d\mu\right\},
\end{equation}
and define an equilibrium measure as a $T$-invariant probability measure $\mu$ which attains the supremum above. For a deeper
study of some aspects of ergodic theory of equilibrium states, we strongly recommend the reader to see \cite{keller:98}. 

The definition of a Gibbs measure for subshifts goes back to Schmidt \cite{schmidt:97}, Petersen and Schmidt \cite{schpet}, Aaronson and Nakada \cite{aaronson:07}, and Meyerovitch \cite{meyerovitch:13}.
Differently from the usual approach, this definition was provided by using more abstract concepts involving conformal measures, without mentioning conditional expectations. 
Thus, we devote Chapter $\ref{cap:res}$ to introduce some basic notions about conformal measures. For further references, see \cite{feldman:77}, \cite{schpet}.

In Chapter \ref{cap:gibbs}, we begin the study of Gibbs measures for a specific class of functions, the so-called functions with $d$-summable variation (\cite{meyerovitch:13}) or
regular local energy functions (\cite{keller:98}, \cite{muir}). Adopting Meyerovitch's approach, we provide the definitions of a Gibbs measure and of a topological
Gibbs measure for such a function $f$. Although we define a Gibbs measure in two different ways, we show that this second one is a relaxed notion which coincides with the first one for subshifts of finite type (SFTs). 
We also show that every equilibrium measure for a function with $d$-summable variation is
a topological Gibbs measure. In particular, if we suppose that we are dealing with a SFT, then every equilibrium measure is also a Gibbs measure.  

The last section of Chapter \ref{cap:gibbs} is completely devoted to connect the notion of a Gibbs measure provided by Meyerovitch with more familiar definitions  
adopted in the literature (e.g. \cite{cap:76}, \cite{georgii:11}, \cite{ny:08}, \cite{sarig:09}, \cite{ruelle:04}). In $1976$, Capocaccia \cite{cap:76} gave a definition of a 
Gibbs measure in the general context of compact metrizable spaces where $\zd$ acts by an expansive group of homeomorphisms. Using the techniques developed in Chapter
\ref{cap:res}, we proved that in the case where we are dealing with subshifts of finite type, Meyerovitch's definition is a particular case of Capocaccia's notion of Gibbs states.
In order to connect our approach with the adopted by Georgii \cite{georgii:11}, for each subshift $X$ and each function $f$ with $d$-summable variation 
we defined a corresponding family $\gamma = (\gamma_{\Lambda})_{\Lambda \in \mathscr{S}}$ of proper probability kernels
satisfying the compatibility relation $\gamma_{\Delta}\gamma_{\Lambda} = \gamma_{\Delta}$ whenever $\Lambda \subseteq \Delta \subseteq \zd$, and proved 
that every Gibbs measure for $f$ satisfies the equation
\begin{equation}\label{dlrdlrdlr}
\mu(A|\mathscr{F}_{\Lambda^c}) = \gamma_{\Lambda}(A|\,\cdot\,){}
\end{equation}
for each Borel set $A$ of $X$ and each $\Lambda$ in $\mathscr{S}$. Conversely, if we suppose that $\mu$ is a probability measure that satisfies (\ref{dlrdlrdlr}),
then $\mu$ is a topological Gibbs measure for $f$. In particular, if $X$ is a SFT, then a probability measure $\mu$ is a Gibbs measure if and only if
$\mu$ satisfies (\ref{dlrdlrdlr}). The set of equations above are often referred to as DLR equations, named for Dobrushin, Lanford and Ruelle.


\chapter{Shift spaces}
\label{cap:shifts}

The aim of this chapter is to introduce the basic properties of shift spaces. These objects are of great importance in the study of classical 
equilibrium statistical mechanics and dynamical systems. 
In this work, we restrict our attention to Gibbs and equilibrium measures on multidimensional shifts over finite alphabets,
but if the reader is interested in results concerning the case of countably infinite alphabets, see Muir \cite{muir} and Sarig \cite{sarig:09}.

In the following, we introduce the definition of a full shift and its topological aspects. In the last section, we will study the concept
of shift spaces (also called subshifts) and provide some examples.

\section{Full shifts}\label{lero}

In this work, we will use $\mathbb{N}$, $\mathbb{Z}_{+}$, $\mathbb{Z}$, $\mathbb{Q}$, $\mathbb{R}$, and $\overline{\mathbb{R}}$ to denote the sets 
of positive integers, of nonnegative integers, of integers, of rational numbers, of real numbers, and of extended real numbers, respectively.
Adopting the terminology of symbolic dynamical systems, a finite set of symbols $\mathcal{A}$ will be referred to as an alphabet.

From now on, let us fix a positive integer $d$ and let $\g$ be the infinite lattice $\zd$ or $\zdp$.
Note that $\g = \zd$ (resp. $\g = \zdp$) is a group (resp. monoid) under the usual operation of addition. 
Let us also define $\|i\| \defeq \max\limits_{1 \leq l \leq d}|i_l|$ for each point $i$ in $\g$, and let $\Lambda_n \defeq \{i \in \g : \|i\| < n\}$ for 
every nonnegative integer $n$. It is easy to check that $\Lambda_{0} = \emptyset$, and for each positive integer $n$ we have
\begin{equation}
\Lambda_{n} = \{-(n-1), \dots,-1, 0, 1, \dots, n-1\}^d
\end{equation}
in the case where $\g = \zd$, and
\begin{equation}
\Lambda_{n} = \{0, 1, \dots, n-1\}^d
\end{equation}
in the case where $\g = \zdp$.

\begin{mydef}
The $\g$-full shift over the alphabet $\mathcal{A}$ is defined by $\gds$, where $\gds$ is the standard mathematical notation for 
the set of all functions from $\g$ into $\mathcal{A}$.
\end{mydef}

\begin{ex}
The full shift $\{0,1\}^{\mathbb{Z}}$ corresponds to the set of all bi-infinite binary sequences.
\end{ex}

\begin{ex}\label{ising1}
The full shift $\{-1, +1\}^{\zd}$ plays an important role in the study of a special model in statistical mechanics, the so-called Ising model.
The Ising model is a mathematical model of ferromagnetism that describes the statistical behavior of a system
consisting of magnetic dipole moments of atomic spins located
at the sites of a crystal lattice. These spins 
may be oriented upwards or downwards (corresponding to the values $+1$ and $-1$, respectively) and are allowed to interact with their neighbors.
See \cite{georgii:11} for further information on the topic discussed in this example. 
\end{ex}

Now, let us fix some notation. As usual, for every element $x$ of $\gds$, we will write $x_i$ instead of $x(i)$ for each point $i$ in $\g$, and let $(x_i)_{i \in \g}$ denote the element $x$. 
Let $\Lambda$ and $\Delta$ be subsets of $\g$ such that  $\Lambda \subseteq \Delta$, then the restriction of configuration $\omega$ in $\mathcal{A}^{\Delta}$ to the subset $\Lambda$ will be denoted by 
$\omega_{\Lambda}$. In this same setting, if we let $\eta \in \mathcal{A}^{\Lambda}$ and $\zeta \in \mathcal{A}^{\Delta \backslash \Lambda}$, then the juxtaposition
$\eta \zeta$ will be defined as the element of $\mathcal{A}^{\Delta}$ such that $(\eta \zeta)_{\Lambda}=\eta$ and $(\eta \zeta)_{\Delta \backslash \Lambda} = \zeta$. 

\begin{mydef}
For each $j$ in $\g$ the map $\sigma^j : \gds \rightarrow \gds$
given by 
\begin{equation}
\sigma^j (x) = \left(x_{i+j}\right)_{i \in \g}
\end{equation}
for every $x = (x_i)_{i \in \g}$, is called the shift or translation by $j$ . 
\end{mydef}

The next properties follows immediately from the definition above.

\begin{fact}\label{eta}
We have $\sigma^{\mathbf{0}} = \mathsf{id}$, where $\mathbf{0} = (0,\dots,0)$ is the zero element of $\g$ and
$\mathsf{id}$ is the identity mapping of $\gds$. \qed
\end{fact}

\begin{fact}\label{etaa}
The identity $\sigma^{i+j} = \sigma^{i} \circ \sigma^{j}$ holds for all $i, j \in \g$.
\end{fact}
\begin{proof}
Indeed, we have $(\sigma^{i} \circ \sigma^{j})(x)_k = \sigma^{i}(\sigma^{j}(x))_k = \sigma^{j}(x)_{k+i} = x_{(k+i)+j} = 
x_{k+(i+j)} = \sigma^{i+j}(x)_k$
for each $k$ in $\g$ and each $x$ in $\gds$. It follows that the equality $(\sigma^{i} \circ \sigma^{j})(x) = \sigma^{i+j}(x)$ holds for every $x$ in $\gds$.
\end{proof}

\begin{fact}\label{etaaa}
In the case where $\g = \zd$, each shift map $\sigma^{j}$ is invertible and its inverse is given by $\left(\sigma^{j}\right)^{-1} =  \sigma^{-j}$.\qed
\end{fact}

\section{The topology of $\gds$}

We devote this section to explore the topological properties of a full shift, but, in order to do so, first we need to specify the topology defined on it.

For each point $j$ in $\g$, let $\pi_{j} : \gds \rightarrow \mathcal{A}$ be the projection of $\gds$ onto the $j$-th coordinate defined by letting
$\pi_j (x) = x_{j}$ for each element $x = (x_{i})_{i \in \g}$.
Naturally, we will consider the set $\mathcal{A}$ endowed with the discrete topology $\tau=\mathcal{P}(\mathcal{A})$ 
and endow the full shift $\gds$ with the initial topology with respect to the family of projections $(\pi_i)_{i \in \g}$. In other words, we will always consider  
the full shift $\gds$ endowed with the product topology\footnote{In this case one also says that $\tau^{\g}$ is the prodiscrete topology.} $\tau^{\g}$. 

Our main objective in this section is to prove that $\gds$ is a compact metrizable space. First, let us define a function $\rho : \gds \times \gds \rightarrow {[}0,+\infty)$ by letting
\begin{equation}\label{metric}
\rho (x,y) = 
\begin{cases}
2^{-n(x,y)}  & \text{if} \; x \neq y ,\; \text{where} \; n(x,y) \defeq \max\left\{n \in \mathbb{Z}_{+} : x_{\Lambda_n} = y_{\Lambda_n}\right\}, \\
0  & \text{if} \; x=y;
\end{cases}
\end{equation}
and show that it is a metric on $\gds$ that generates its topology. Note that $\rho$ is well defined. In fact, if we suppose that $x$ and $y$ are distinct elements of $\gds$, then
there is a point $i$ in $\g$ such that $x_{i} \neq y_{i}$. It follows that the set $\left\{n \in \mathbb{Z}_{+} : x_{\Lambda_n} = y_{\Lambda_n}\right\}$
is nonempty (because $0$ belongs to it) and bounded above, thus it assumes a maximum element.

Now, let us show that $\rho$ is a metric on $\gds$. It is sufficient to prove the triangular inequality, since the
other properties follow immediately from (\ref{metric}). 
For any $x, y,$ and $z$ in $\gds$, if we suppose that $x=z$ or $y=z$, then it is clear that $\rho(x,y) = \rho(x,z) + \rho(z,y)$. Now, in the case where $x \neq z$ and $y \neq z$,
if we let $n = \min\{n(x,z), n(z,y)\}$, we obtain $x_{\Lambda_n}= y_{\Lambda_n}$, thus $\rho(x,y) \leq 2^{-n} < 2^{-n(x,z)} + 2^{-n(z,y)} = \rho(x,z) + \rho(z,y)$.   

\begin{remark}\label{box}
It is easy to prove that for every positive integer $n$ and for every elements $x$ and $y$ of $\gds$, we have
\begin{equation}
\rho(x,y) \leq 2^{-n}\;\, \text{if and only if}\;\, x_{\Lambda_n} = y_{\Lambda_n}.
\end{equation}
\end{remark}

As usual, we will denote the open ball (with respect to $\rho$) centered at the point $x$ in $\gds$ with radius $r > 0$ 
by
\[B(x,r) \defeq \left\{y \in \gds : \rho(x,y) < r\right\}.\] 

The next result shows that $\rho$ is a metric that generates the topology of $\gds$.

\begin{prop}\label{metrizable}
The topological space $(\gds, \tau^{\g})$ is metrizable.
\end{prop}

\begin{proof}
Let us consider the metric $\rho$ introduced in (\ref{metric}) together with its induced topology $\tau_{\rho}$.
Given a point $i$ in $\g$ and a subset $A$ of $\mathcal{A}$, let us show that $\pi_{i}^{-1}(A)$ is an open set with respect to $\rho$. For every element $x$ of $\pi_{i}^{-1}(A)$,
if we let $n = \|i\| + 1$, then for each $y$ in $B(x,2^{-n})$ we have $x_{\Lambda_n} = y_{\Lambda_n}$ (see Remark \ref{box}),
hence $y_i= x_i \in A$.
It follows that $\left\{\pi_{i}^{-1}(A): i \in \g, \, A \subseteq \mathcal{A}\right\} \subseteq \tau_{\rho}$, therefore,
by the definition of product topology, we conclude that $\tau^{\g} \subseteq \tau_{\rho}$.

Conversely, given an arbitrary element $x$ of $\gds$ and a positive number $\epsilon$, let $n$ be a positive integer such that $2^{-n} < \epsilon$ and let
$U = \bigcap\limits_{i \in {\Lambda_{n}}}\pi_{i}^{-1}(\{x_i\})$. Since every point $y$ in $U$ satisfies 
$y_{\Lambda_n}=x_{\Lambda_n}$, it follows that $\rho(x,y) \leq 2^{-n} < \epsilon$.
So, we conclude that $U$ is an element of $\tau^{\g}$ containing $x$ such that $U \subseteq B(x,\epsilon)$. Since the collection $\left\{B(x,\epsilon): x \in \gds, \, \epsilon>0\right\}$ is a basis for
$\tau_{\rho}$, it follows that $\tau_{\rho} \subseteq\tau^{\g}$.
\end{proof}

Using Remark \ref{box} and Proposition \ref{metrizable}, it can be easily verified that a sequence $(x^{(n)})_{n \in \mathbb{N}}$ in $\gds$ converges to a point $x$ if and only if 
for every positive integer $N$ there is another positive integer $n_0$ such that the equality
$x^{(n)}_{\Lambda_N}=x_{\Lambda_N}$	
holds whenever $n \geq n_0$.

\begin{cor}\label{homeo}
The shift map $\sigma^{j}$ is continuous for each $j$ in $\g$. 
\end{cor}

\begin{proof}
Let $(x^{(n)})_{n \in \mathbb{N}}$ be a sequence in $\mathcal{A}^{\g}$ converging to a point $x$. For each positive integer $N$, if we let $M=\max\limits_{k \in \Lambda_{N}}\|k+j\| + 1$,   
then there is another positive integer $n_0$ such that $n \geq n_0$ implies that $x_{\Lambda_M}^{(n)}=x_{\Lambda_M}$. Therefore, the equality $\sigma^{j}(x^{(n)})_{\Lambda_N}=\sigma^{j}(x)_{\Lambda_N}$ holds whenever $n \geq n_0$.
\end{proof}

For some technical proofs, it is convenient to use the fact that $\gds$ is a compact space.
This result can be easily proved by applying Tychonoff's theorem, however, it can also be derived from
the fact that $\gds$ is metrizable space.

\begin{teo}\label{compact}
The full shift $\gds$ is a compact space.
\end{teo}

\begin{proof}[First proof]
Recall that $\mathcal{A}$ is a finite set endowed with the discrete topology $\tau = \mathcal{P}(\mathcal{A})$, thus the space $(\mathcal{A}, \tau)$ is compact. Using Tychonoff's theorem, the result follows.
\end{proof}

\begin{proof}[Second proof]
Let $(x^{(n)})_{n \in \mathbb{N}}$  be an arbitrary sequence in $\gds$. Let us show that we can find some convergent subsequence of $(x^{(n)})_{n \in \mathbb{N}}$.

First, let us define $S_0 = \mathbb{N}$. Note that there exists an element $\omega_1$ of $\mathcal{A}^{\Lambda_1}$ such that $\big\{n \in S_0 : x^{(n)}_{\Lambda_1} = \omega_1\big\}$ is an infinite set. In fact, if 
$\left\{n \in S_0 : x^{(n)}_{\Lambda_1} = \omega\right\}$ were a finite set for every $\omega \in \mathcal{A}^{\Lambda_1}$, then
it would imply that $S_0 = \bigcup\limits_{\omega \in \mathcal{A}^{\Lambda_1}}\big\{n\in S_0: x^{(n)}_{\Lambda_1}=\omega\big\}$ is also a finite set, a
contradiction. Therefore, we let $\omega_1$ be an element of $\mathcal{A}^{\Lambda_1}$ such that the set $S_1 =\big\{n \in S_0 : x^{(n)}_{\Lambda_1}=\omega_1\big\}$ is infinite.

Suppose that we have already defined an element $\omega_N$ of $\mathcal{A}^{\Lambda_N}$ such that 
$S_N = \big\{n \in S_{N-1} : x^{(n)}_{\Lambda_N}=\omega_N\big\}$ is an infinite set. 
Let us show that we can find an element $\omega_{N+1}$ of $\mathcal{A}^{\Lambda_{N+1}}$ such that the set $\big\{n \in S_N : x^{(n)}_{\Lambda_{N+1}}=\omega_{N+1}\big\}$
is infinite. Using an analogous argument as before, if $\big\{n \in S_N : x^{(n)}_{\Lambda_{N+1}}=\omega\big\}$
were a finite set for every $\omega$ in $\mathcal{A}^{\Lambda_{N+1}}$, then $S_N = \bigcup\limits_{\omega \in \mathcal{A}^{\Lambda_{N+1}}}\big\{n \in S_N : x^{(n)}_{\Lambda_{N+1}}=\omega\big\}$
would be a finite set, a contradiction. Therefore, let us define $\omega_{N+1}$ as the element of $\mathcal{A}^{\Lambda_{N+1}}$ such that the set $S_{N+1} = \big\{n \in S_N : x^{(n)}_{\Lambda_{N+1}}=\omega_{N+1}\big\}$ is infinite. 

In this way, we obtain two sequences $(\omega_N)_{N \in \mathbb{N}}$ and $(S_N)_{N \in \mathbb{N}}$ such that $(\omega_{N+1})_{\Lambda_{N}} = \omega_N$ and $S_{N+1} \subseteq S_{N}$ for each
$N$.
Now, let us define $x$ as the element of $\gds$ that satisfies the identity $x_{\Lambda_N} = \omega_N$ for every $N$. It is easy to check that we can 
construct an increasing sequence $n_{1} < n_{2} < \cdots < n_{l} < \cdots$ of positive integers, where each $n_{l}$ belongs to $S_{l}$. 
We claim that $(x^{(n_l)})_{l \in \mathbb{N}}$ is a subsequence of $(x^{(n)})_{n \in \mathbb{N}}$ which converges to $x$. 
Indeed, for each positive integer $N$, the equation
\[x_{\Lambda_{N}}^{(n_l)}= (\omega_{l})_{\Lambda_{N}}= x_{\Lambda_{N}}\]
holds whenever $l$ is an integer satisfying $l \geq N$.
\end{proof}

\section{Subshifts}

In this section, we will study some particular subsets of full shifts called shift spaces. These objects are most commonly referred to as subshifts and play an important
role in the study of dynamical systems. For the one who is interested in studying 
one-dimensional subshifts in symbolic dynamics, we invite the reader to check the book by Lind and Marcus \cite{lind:95}. And, for the 
reader who is interested in how the study of shifts connects with statistical mechanics, we strongly recommend the books by Georgii \cite{georgii:11} and Keller \cite{keller:98} which are
two masterpieces on this subject. 

In the following, we will present the definition and basic properties of a subshift and turn to few examples. 

\begin{mydef}
A subset $X$ of $\gds$ is said to be a subshift if it is topologically closed and invariant under translations (i.e., the inclusion
$\sigma^{j}(X) \subseteq X$ holds for each $j$ in $\g$).
\end{mydef}

\begin{ex}
Clearly, $X = \emptyset$ and $X = \gds$ are subshifts of $\gds$.
\end{ex}

\begin{ex}
Let $\g$ be the set of all nonnegative integers and let $\mathcal{A}$ be the alphabet $\{0,1\}$. If we let $x$ and $y$ be two elements of 
the full shift $\{0,1\}^{\mathbb{Z_{+}}}$ defined by
\begin{equation*}
x_i =
\begin{cases}
0 &\text{if}\; i\; \text{is even,} \\
1 &\text{if}\; i\; \text{is odd;} 
\end{cases} 
\end{equation*}
and
\begin{equation*}
y_i =
\begin{cases}
1 &\text{if}\;i\; \text{is even,} \\
0 &\text{if}\; i\; \text{is odd;}
\end{cases} 
\end{equation*}
one can easily verify that $X = \{x,y\}$ is a subshift of $\{0,1\}^{\mathbb{Z_{+}}}$. 
\end{ex}

If $\Lambda$ is a finite subset of $\g$, then we will sometimes refer to an element of $\mathcal{A}^{\Lambda}$ as pattern on $\Lambda$.
Given an arbitrary collection $\mathcal{F}$ of patterns\footnote{The patterns in this collection are often referred to as forbidden patterns.}, more precisely, 
given a subset $\mathcal{F}$ of $\bigcup\limits_{\substack{\Lambda \subseteq \g \\ \Lambda\; \text{finite}}}\mathcal{A}^{\Lambda}$, let us define a subset $\mathsf{X}_{\mathcal{F}}$ of $\gds$ by 
\begin{equation}\label{subshift}
\mathsf{X}_{\mathcal{F}} \defeq \left\{x \in \gds: \sigma^{j}(x)_{\Lambda} \notin \mathcal{F}\; \text{for all}\; j \in \g\;\text{and for every}\; \Lambda \subseteq \g\; \textnormal{finite}\right\}.
\end{equation}

The next result will provide us an alternative characterization of subshifts, in the sense that every subshift can be written in the form (\ref{subshift}). 
Later, this characterization will allow us to derive the concept of a subshift of finite type.

\begin{teo}[Equivalent definition for subshifts]
A subset $X$ of $\gds$ is a subshift if and only if it can be written in the form $X = \mathsf{X}_{\mathcal{F}}$ for some
collection $\mathcal{F}$ of patterns.
\end{teo}

\begin{proof}
Let $X$ be a subshift of $\gds$. Since $\gds \backslash X$ is an open set, 
then to each point $x$ in $\gds \backslash X$ we associate a pattern 
$\omega(x)$ given by $\omega(x) = x_{\Lambda_{n}}$, where $n$ is a positive integer such that the set $\big\{y \in \gds: x_{\Lambda_n}=y_{\Lambda_n}\big\}$ is included in $\gds \backslash X$.
We also let $\mathcal{F}$ be the collection of patterns defined by $\mathcal{F}=\{\omega(x) : x \in \gds \backslash X\}$. 
We claim that $X = \mathsf{X}_{\mathcal{F}}$.
Indeed, note that the point $x$ does not belong to $X$ if and only if there is a finite subset $\Lambda$ of $\g$ such that $x_\Lambda$ belongs to $\mathcal{F}$. 
This fact together with the translation invariance of $X$ implies that
\begin{eqnarray*}
x \in X &\iff& \sigma^{j}(x) \in X \;\, \text{for all}\; j \in \g \\
&\iff& \sigma^{j}(x)_\Lambda \notin \mathcal{F}\;\,\text{for all}\; j \in \g \; \text{and for every}\; \Lambda \subseteq \g\; \text{finite}.
\end{eqnarray*}
Thus, we conclude that $X = \mathsf{X}_{\mathcal{F}}$.

On the other hand, let $\mathcal{F}$ be a collection of patterns such that $X=\mathsf{X}_{\mathcal{F}}$.
First, let us prove that $X$ is topologically closed. If $(x^{(n)})_{n \in \mathbb{N}}$ is a sequence in $X$ converging to an element $x$ of $\gds$, then
given a point $j$ in $\g$ and a finite  subset $\Lambda$ of $\g$ the continuity of $\sigma^{j}$ implies
on the existence of a positive integer $n$ such that $\sigma^{j}(x)_\Lambda = \sigma^{j}(x^{(n)})_\Lambda$, thus $\sigma^{j}(x)_\Lambda$ does not belongs to $\mathcal{F}$. 
It follows that $x$ belongs to $X$. Now, in order to prove that $X$ is translation invariant, it is sufficient to show that $X \subseteq \sigma^{-j}(X)$ for every $j$ in $\g$. 
Given an element $x$ of $X$, the pattern $\sigma^{i}(\sigma^{j}(x))_\Lambda = \sigma^{i+j}(x)_{\Lambda}$ does not belongs to $\mathcal{F}$ for each point $i$ in $\g$ and each finite subset $\Lambda$ of $\g$, hence
$\sigma^{j}(x)$ belongs to $X$.
\end{proof}

\begin{ex}[Even shift]\label{es}
Let $\g$ be the set of all integers and let $\mathcal{A}$ be the alphabet $\{0,1\}$. For each positive integer $n$, let us define a pattern $\omega^{(n)} : \Lambda_{n+1} \rightarrow \{0,1\}$
by letting
\begin{equation}
\omega^{(n)}_{i} = 
\begin{cases}
0 &\text{if}\; |i| < n, \\
1 &\text{if}\; |i| = n.
\end{cases}
\end{equation}
If we let $\mathcal{F}$ be a collection of patterns given by $\mathcal{F} = \{\omega^{(n)} : n \in \mathbb{N}\}$, 
let us define the even shift as the subshift of $\{0,1\}^{\mathbb{Z}}$ given by $\mathsf{X}_{\mathcal{F}}$. 
One can easily verify that the even shift is the set of all bi-infinite binary sequences so that there are
an even number of $0$'s between any two $1$'s.
\end{ex}

\begin{mydef}\label{defsft}
A subshift $X$ of $\gds$ is called a subshift of finite type (or SFT for short) if it can be written in the form $X=\mathsf{X}_\mathcal{F}$ for some finite set $\mathcal{F}$ of patterns.
\end{mydef}

\begin{remark}\label{remarksft}
Observe that if $X$ is a subshift of finite type, then it can be assumed that
all patterns in $\mathcal{F}$ are defined on the same (finite) set $\Lambda$. Indeed, let $\Lambda$ be
a finite subset of $\g$ containing the domain of all patterns in $\mathcal{F}$. Then, let us define  $\widetilde{\mathcal{F}}$ as the 
(finite) set of patterns
\[\widetilde{\mathcal{F}} = \left\{\eta \in \mathcal{A}^\Lambda : \eta\restriction_{\text{dom}\, \omega} \;= \omega\;\,\text{for some}\;\omega \in \mathcal{F}\right\}.\]
It is straightforward to prove that 
$X_{\widetilde{\mathcal{F}}} = X_{\mathcal{F}}$.
\end{remark}

\begin{ex}
The full shift $X = \mathcal{A}^{\g}$ is a SFT, since $X$ can be written in the form $X = \mathsf{X}_{\mathcal{F}}$ by choosing $\mathcal{F} = \emptyset$.
\end{ex}

\begin{ex}[Golden mean shift]
Let $\omega : \{0,1\} \rightarrow \{0,1\}$ given by $\omega_{i} = 1$ for each point $i \in \{0,1\}$,
and let $\mathcal{F} = \{\omega\}$. The golden mean shift is the subshift of $\{0,1\}^{\mathbb{Z}}$ defined by $\mathsf{X}_\mathcal{F}$.
Clearly, the golden mean shift is a SFT that consists of all bi-infinite binary sequences such that there is no two consecutive $1$'s.
\end{ex}

\begin{ex}
The even shift defined in Example $\ref{es}$ is not a SFT. Indeed, let us suppose that exists a finite set $\mathcal{F}$ of patterns such that the
even shift can be written in the form
$\mathsf{X}_{\mathcal{F}}$. Without loss of generality, we can assume that all patterns in $\mathcal{F}$ are defined on
$\Lambda_{n} = \{-(n-1), \dots, 0, \dots, n-1\}$ for some positive integer $n$. It follows that the element $x$ of $\{0,1\}^{\mathbb{Z}}$ defined by
\begin{equation}
x_i = 
\begin{cases}
0 &\text{if}\; |i| \neq n,\\
1 &\text{if}\; |i| = n;
\end{cases}
\end{equation}
belongs to the even shift, a contradiction.
\end{ex}

Now, we will present an equivalent definition of subshifts of finite type frequently presented in other texts.
Using Remark \ref{remarksft}, it is easy to show that a SFT can also be characterized in terms of a finite number of allowed patterns instead of a finite number of forbidden patterns.   

\begin{prop}\label{csft}
A subset $X$ of $\gds$ is a subshift of finite type if and only if there is a finite subset $\Lambda$ of $\g$ and a set $P$ of patterns on $\Lambda$ such that
\begin{equation}
X  = \left\{x \in \gds :\sigma^{j}(x)_\Lambda \in P \;\, \textnormal{for all}\; j \in \g\right\}. 
\end{equation}
\end{prop}

\begin{proof}
Due to Remark \ref{remarksft}, $X$ is a subshift of finite type if and only if there exists a finite subset $\Lambda$ of $\g$ and a collection $\mathcal{F}$ of patterns 
on $\Lambda$ such that 
\begin{equation*}
X = \mathsf{X}_{\mathcal{F}} = \left\{x \in \gds : \sigma^{j}(x)_{\Lambda} \notin \mathcal{F}\;\,\textnormal{for all}\; j \in \g\right\}. 
\end{equation*}
Therefore, the result follows.
\end{proof}

\begin{ex}[Matrix subshift]
Let $A_1, \dots, A_d \in {\{0,1\}}^{\mathcal{A}\times\mathcal{A}}$ be matrices of $0$'s and $1$'s indexed by $\mathcal{A}\times\mathcal{A}$. In the literature, these
matrices are often referred to as transition matrices. 
If we define
\begin{equation}
\Sigma_{A_1, \dots,A_d} \defeq \left\{x \in \gds : A_n(x_i,x_{i+\textbf{e}_n}) = 1\,\, \textnormal{for all}\;i \in \g\; \textnormal{and for each}\; n \in \{1, \dots, d\} \right\}
\end{equation}
where each $\textbf{e}_n$ is the element of $\g$ defined by $\textbf{e}_n = (0,\dots,\tikzmark{b}1,\dots, 0)$, then $\Sigma_{A_1, \dots,A_d}$ is a SFT called matrix subshift.

\begin{tikzpicture}[remember picture,overlay]
\draw[<-] 
  ([shift={(2pt,-2pt)}]pic cs:b) |- ([shift={(14pt,-10pt)}]pic cs:b) 
  node[anchor=west] {$\scriptstyle n\text{-th position}$}; 
\end{tikzpicture}

Let us show that $\Sigma_{A_1, \dots,A_d}$ is indeed a SFT. If we let $\Lambda = \{\textbf{0}, \textbf{e}_1, \dots, \textbf{e}_d\}$ and define
\[P = \left\{\omega \in \mathcal{A}^\Lambda : A_n\left(\omega_{\textbf{0}}, \omega_{\textbf{e}_n}\right) = 1\;\, \textnormal{for all}\;n \in \{1, \dots, d\}\right\},\]
then
\begin{eqnarray*}
\Sigma_{A_1, \dots,A_d} &=& \left\{x \in \gds : A_n(x_i,x_{i+\textbf{e}_n}) = 1\,\, \textnormal{for all}\;i \in \g\; \textnormal{and for each}\; n \in \{1, \dots, d\} \right\} \\
&=& \left\{x \in \gds : A_n\left(\sigma^i(x)_{\textbf{0}},\sigma^i(x)_{\textbf{e}_n}\right) = 1\;\,\textnormal{for all}\; i \in \g\; \textnormal{and for each}\; n \in \{1, \dots, d\} \right\} \\
&=& \left\{x \in \gds : \sigma^i(x)_{\Lambda} \in P\;\, \textnormal{for all}\; i \in \g\right\}.
\end{eqnarray*}
Thus, using Proposition \ref{csft}, we conclude that $\Sigma_{A_1, \dots,A_d}$ is a subshift of finite type.
\end{ex}

\chapter{Thermodynamic formalism}
\label{cap:fm}

In Chapter \ref{cap:gibbs} we will see how the study of Gibbs and equilibrium measures on subshifts are connected among themselves,
but, in order to do so, we will dedicate this chapter to provide the basic ideas of thermodynamic formalism. The reader who is
familiar with this subject can skip this chapter and proceed directly to the next one.

The first two sections are devoted to the study of entropy. We start by introducing and deriving the basic properties concerning the entropy of partitions, and then we use 
these notions to study the entropy of dynamical systems. Finally, in the last section will be introduced the definition of an equilibrium measure. 

\section{Entropy of partitions}
We start this section by introducing the definition of a partition of a probability space $(X,\mathcal{B},\mu)$, called a $\mu$-partition.

\begin{mydef}
Let $(X,\mathcal{B},\mu)$ be a probability space. A $\mu$-partition of $X$ is a countable (finite or countably infinite) collection
$\alpha$ of measurable subsets of $X$ such that
\begin{itemize}
\item[(a)] $\mu\left(\bigcup \alpha\right) = 1$, and

\item[(b)] $\mu(A \cap B) = 0$ whenever $A$ and $B$ are distinct elements of $\alpha$.
\end{itemize}
\end{mydef}

\begin{remark}\label{rmup}
Observe that for $\mu$-almost every $x$ in $X$ there is a unique element $A$ of $\alpha$ which contains $x$. Indeed, since the set $N = \bigcup \{A_1 \cap A_2 : A_1, A_2 \in \alpha\;\text{and}\; A_1 \neq A_2\}$
 is a countable union of $\mu$-null sets, it follows that $\mu(N) = 0$. Therefore, $\left(\bigcup \alpha\right)\backslash N$ is a set with measure 1 whose points satisfy the 
required propety. 
\end{remark}

In the case where $\alpha$ is a countable collection of measurable subsets of $X$ such that its elements are pairwise disjoint and
$\bigcup \alpha = X$, the collection $\alpha$ will be referred to as a partition of $X$. Notice that every partition is also a $\mu$-partition.

The following example will provide us a method for generating a $\mu$-partition from other two.
\begin{ex}\label{cr}
Given two $\mu$-partitions of $X$, say $\alpha$ and $\beta$, we define their common refinement by
\begin{equation}
\alpha \vee \beta \defeq \{A \cap B : A \in \alpha, B \in \beta\}.	
\end{equation}
Let us show that the collection $\alpha \vee \beta$ is a $\mu$-partition of $X$.
It is easy to check that $\alpha \vee \beta$ is a countable collection of measurable subsets of $X$.
The identity $\mu\left(\bigcup \alpha \vee \beta\right) = 1$, follows from the fact that $\bigcup \alpha \vee \beta = 
\left(\bigcup \alpha\right) \cap \left(\bigcup \beta\right)$.
Now, given two distinct elements $U_1$ and $U_2$ in $\alpha \vee \beta$, there are 
$A_1, A_2 \in \alpha$ and $B_1, B_2 \in \beta$ such that $U_n = A_n \cap B_n$ for each $n \in \{1,2\}$. Since $U_1 \neq U_2$, it follows that either
$A_1 \neq A_2$ or $B_1 \neq B_2$. Therefore, we have $\mu(U_1 \cap U_2) = \mu\left((A_1 \cap A_2)\cap(B_1 \cap B_2)\right) = 0$.  
\end{ex}

In order to define the concept of entropy associated to a $\mu$-partition let us introduce a function which quantifies the
amount of information gained by an observer which observes the system through this partition. Given a $\mu$-partition $\alpha$, an information function of $\alpha$ is a
measurable function $I_{\alpha}: X \rightarrow \overline{\mathbb{R}}$ satisfying
\begin{equation}\label{info1}
I_{\alpha}(x) = \sum\limits_{A \in \alpha} - \log\mu(A)\cdot \rchi_A(x) 
\end{equation}
for $\mu$-almost every $x$ in $X$.  
\begin{remark}\label{cjncn}
\begin{enumerate}[label=(\alph*),ref=\alph*]
\item We need to emphasize that the expression on the right-hand side of (\ref{info1}) makes sense for $\mu$-almost every $x$ in $X$. 
Indeed, if 	we let $N_0 = \bigcup \{A \in \alpha : \mu(A)= 0\}$, then $\mu(N_0) = 0$ and the right-hand side of (\ref{info1}) makes sense for every point
$x$ in $X\backslash N_0$.
\item Clearly, any two information functions of $\alpha$ coincide $\mu$-almost everywhere. 
\item\label{item:c} Observe that always exists an information function of $\alpha$, 
for example, let us consider the function $\widetilde{I_{\alpha}}$ defined by
\begin{equation}
\widetilde{I_{\alpha}}(x) = 
\begin{cases}
\sum\limits_{A \in \alpha} - \log\mu(A)\cdot \rchi_A(x)& \text{if}\;\, x \in X\backslash{N_0}, \\
0 &\text{otherwise};
\end{cases}
\end{equation}
where $N_0 = \bigcup \{A \in \alpha : \mu(A)= 0\}$. It is easy to check that $\widetilde{I_{\alpha}}$ is a nonnegative measurable function on $X$, thus $\widetilde{I_{\alpha}}$ is an information function of $\alpha$.
Moreover, given an arbitrary information function $I_{\alpha}$
we have 
\begin{equation}
\int_{X} I_{\alpha} \,d\mu = \int_{X} \widetilde{I_{\alpha}} \,d\mu = \sum\limits_{A \in \alpha}-\mu(A) \cdot \log \mu(A). 
\end{equation}
\end{enumerate}
\end{remark}

In view of Remark \ref{cjncn}(\ref{item:c}), we will define the entropy of a $\mu$-partition as the average information cointained on it.

\begin{mydef}[Entropy of $\alpha$]
Let $(X,\mathcal{B},\mu)$ be a probability space. Given a $\mu$-partition $\alpha$, we define its entropy by
\begin{equation}
H_{\mu}(\alpha) \defeq \sum\limits_{A \in \alpha}-\mu(A) \cdot \log \mu(A).
\end{equation}
Note that $H_{\mu}(\alpha)$ is a nonnegative extended real number.
\end{mydef}

In the following we present more refined notions of information and entropy of $\mu$-partitions in order to describe the gain of information
in the situation where the observer already has previous knowledge about the system.

Before we introduce the conditional information, let us prove the following technical result.

\begin{lemma}\label{lemmadoido}
Let $(X,\mathcal{B},\mu)$ be a probability space, let $\EuScript{F}$ be a sub-$\sigma$-algebra of $\mathcal{B}$ and let $A$ be an element of $\mathcal{B}$. Then, 
we have $\mu(A|\EuScript{F})(x) \defeq \mathbb{E}_{\mu}[\rchi_A|\EuScript{F}](x) >0$ for $\mu$-almost every point $x$ in $A$. 
\end{lemma} 

\begin{proof}
Let us fix some version of $\mu(A|\EuScript{F})$. Since $\mu(A|\EuScript{F})\geq 0$ $\mu$-a.e., if we let 
$F_A$ be the element of $\EuScript{F}$ given by $\{x \in X : \mu(A|\EuScript{F})(x) = 0\}$ and show that that $A\cap F_A$ has measure zero, then
the proof will be complete.

Using the definition of conditional expectation, we have
\[\mu(A \cap F_A) = \int_{F_A}\rchi_A\, d\mu = \int_{F_A}\mu(A|\EuScript{F})\, d\mu = 0,\] 
concluding the proof.
\end{proof}

Let $\alpha$ be a $\mu$-partition of $X$ and let $\EuScript{F}$ be a sub-$\sigma$-algebra of $\mathcal{B}$. A conditional information 
function of $\alpha$ given $\EuScript{F}$ is a measurable function $I_{\alpha|\EuScript{F}} : X \rightarrow \overline{\mathbb{R}}$ satisfying
\begin{equation}\label{condinfo}
I_{\alpha|\EuScript{F}}(x) = \sum\limits_{A \in \alpha}-\log\mu(A|\EuScript{F})(x) \cdot \rchi_A(x) 
\end{equation} 
for $\mu$-almost every $x$ in $X$, for some (therefore, for any) version of each $\mu(A|\EuScript{F})$. 

\begin{remark}\label{xuleta}
\begin{enumerate}[label=(\alph*),ref=\alph*]
\item Note that the right-hand side of (\ref{condinfo}) is well defined for $\mu$-almost every point $x$ in $X$. Indeed,
for each element $A$ of $\alpha$ the set $N_A = \{x \in A : \mu(A|\EuScript{F})(x) \leq 0\}\cup\{x \in X : \mu(A|\EuScript{F})(x) > 1\}$ has measure zero. It follows that
$N_0 = \bigcup\limits_{A \in \alpha}N_A$ also has measure zero. Therefore, the right-hand side is well defined 
on $X\backslash N_0$.\label{danada}

\item On can easily verify that any two conditional information functions are equal $\mu$-almost everywhere.

\item Observe that always exists a conditional information function of $\alpha$ given $\EuScript{F}$, for example, let us fix a version of each $\mu(A|\EuScript{F})$ and consider the function $\widetilde{I_{\alpha|\EuScript{F}}}(x)$ given by
\begin{equation}
\widetilde{I_{\alpha|\EuScript{F}}}(x) =
\begin{cases}
\sum\limits_{A \in \alpha}-\log\mu(A|\EuScript{F})(x) \cdot \rchi_A(x) &\text{if}\;\, x \in X\backslash N_0, \\
0 &\text{otherwise;}
\end{cases}
\end{equation}
where $N_0$ is defined in the same way as we did on Remark \ref{xuleta}(\ref{danada}).
It is easy to check that $\widetilde{I_{\alpha|\EuScript{F}}}$ is a nonnegative measurable function on $X$, thus $\widetilde{I_{\alpha|\EuScript{F}}}$ is a conditional
information function of $\alpha$ given $\EuScript{F}$. 
\end{enumerate}
\end{remark}

\begin{mydef}[Conditional entropy]
Let $(X,\mathcal{B},\mu)$ be a probability space, let $\alpha$ be a $\mu$-partition and let $\EuScript{F}$ be a sub-$\sigma$-algebra of $\mathcal{B}$.
The conditional entropy of $\alpha$ given $\EuScript{F}$ is defined by
\begin{equation}
H_\mu(\alpha|\EuScript{F}) \defeq \int_{X}I_{\alpha|\EuScript{F}}\, d\mu,
\end{equation}
where $I_{\alpha|\EuScript{F}}$ is an arbitrary information function of $\alpha$ given $\EuScript{F}$. Note that $H_\mu(\alpha|\EuScript{F})$ is a 
nonnegative extended real number and its value does not depends on the choice of $I_{\alpha|\EuScript{F}}$.
\end{mydef}

The following example show us that in the case where $\EuScript{F}$ is trivial mod $\mu$ (i.e., the measure of its elements is either $0$ or $1$), then
both notions of entropy coincide.

\begin{ex}\label{trivial}
If $\mu(F) \in \{0,1\}$ for every $F \in \EuScript{F}$, then we have $I_{\alpha|\EuScript{F}} = I_\alpha$ $\mu$-a.e. and $H_{\mu}(\alpha|\EuScript{F}) = H_{\mu}(\alpha)$.
Indeed, given an arbitrary element $A$ of $\alpha$, the equalities
\begin{eqnarray*}
\int_{F}\rchi_A\, d\mu &=& \mu(A \cap F) 
=\mu(A)\mu(F)
= \int_{F}\mu(A) d\mu
\end{eqnarray*}
hold for any $F$ in $\EuScript{F}$. Then, we have $\mu(A|\EuScript{F}) = \mu(A)$ $\mu$-almost everywhere for each $A \in \alpha$. We conclude that
\begin{eqnarray*}
I_{\alpha|\EuScript{F}}(x) &=& \sum\limits_{A \in \alpha}-\log \mu(A)\cdot \rchi_{A}(x) 
= I_{\alpha}(x)
\end{eqnarray*}
for $\mu$-almost every $x$ in $X$, thus $H_{\mu}(\alpha|\EuScript{F}) = H_{\mu}(\alpha)$.
\end{ex}


In the remainder of this section we prove a few elementary properties of entropy of partitions.

\begin{lemma}\label{lemmacond}
Let $\alpha$ be a $\mu$-partition of $X$, let $\EuScript{F}$ be a sub-$\sigma$-algebra of $\mathcal{B}$ and let $B$ be an element of $\mathcal{B}$. Then,
\begin{equation}\label{eqdoida}
\mu(B|\sigma(\alpha)\vee \EuScript{F}) = \sum\limits_{A \in \alpha}\frac{\mu(A\cap B|\EuScript{F})}{\mu(A|\EuScript{F})}\cdot \rchi_A \qquad \mu\text{-a.e.},
\end{equation}
where $\sigma(\alpha)\vee \EuScript{F}$ denotes the smallest $\sigma$-algebra which contains $\sigma(\alpha)\cup \EuScript{F}$.
\end{lemma}

\begin{proof}
We divide this proof into 3 steps.

Step 1. Let us define a $\sigma(\alpha)\vee \EuScript{F}$-measurable set $X_0$ with measure 1 such that for every point $x$ in $X_0$
the right-hand side of (\ref{eqdoida}) is well defined and there is a unique element $A$ of $\alpha$ which contains $x$. 
If we let
\[N_1 = \bigcup\limits_{A \in \alpha}\{x \in X: \mu(A\cap B| \EuScript{F})(x)<0\}\cup\{x\in A: \mu(A|\EuScript{F})(x) \leq 0\},\]
it is easy to check that $N_1$ is a  $\sigma(\alpha)\vee \EuScript{F}$-measurable set such that $\mu(N_1) = 0$. Moreover, if $x$ belongs to $X_1 = X\backslash N_1$, we have
\[\frac{\mu(A\cap B|\EuScript{F})(x)}{\mu(A|\EuScript{F})(x)}\cdot \rchi_A(x) \in {[}0,+\infty)\]
for each $A$ in $\alpha$. Therefore, the sum on the right-hand side of (\ref{eqdoida}) makes sense for every point $x$ in $X_1$.

Recall that we can find a set $X_2$ in $\sigma(\alpha)\vee \EuScript{F}$ with measure 1 such that
for every $x$ in $X_2$ there is a unique element $A$ of $\alpha$ which contains $x$ (see Remark \ref{rmup}). Therefore, our claim follows by letting $X_0 = X_1 \cap X_2$.

Step 2. Now, let us verify that the equation
\begin{equation}\label{eqsafada}
\int_{A'\cap F}\left(\sum\limits_{A \in \alpha}\frac{\mu(A\cap B|\EuScript{F})}{\mu(A|\EuScript{F})}\cdot \rchi_{X_0 \cap A}\right)d\mu 
=\int_{A'\cap F}\rchi_{B}\,d\mu
\end{equation}
holds for every $A' \in \alpha$ and $F \in \EuScript{F}$.

First, observe that the function 
\[\sum\limits_{A \in \alpha}\frac{\mu(A\cap B|\EuScript{F})}{\mu(A|\EuScript{F})}\cdot \rchi_{X_0 \cap A}\]
is $\sigma(\alpha)\vee \EuScript{F}$-measurable. Since
\begin{eqnarray*}
\int_{A'\cap F}\left(\sum\limits_{A \in \alpha}\frac{\mu(A\cap B|\EuScript{F})}{\mu(A|\EuScript{F})}\cdot \rchi_{X_0 \cap A}\right)d\mu 
&=& \int_{F}\frac{\mu(A'\cap B|\EuScript{F})}{\mu(A'|\EuScript{F})}\cdot \rchi_{X_0\cap A'}\,d\mu
\end{eqnarray*}
and
\[\frac{\mu(A'\cap B|\EuScript{F})}{\mu(A'|\EuScript{F})}\cdot \rchi_{X_0\cap A'} = \left(\frac{\mu(A'\cap B|\EuScript{F})}{\mu(A'|\EuScript{F})}\cdot \rchi_{\{x \in X: \mu(A'|\EuScript{F})(x)>0\}}\right)\rchi_{A'}\quad \mu\text{-a.e.},\]
it follows that
\begin{eqnarray*}
\int_{A'\cap F}\left(\sum\limits_{A \in \alpha}\frac{\mu(A\cap B|\EuScript{F})}{\mu(A|\EuScript{F})}\cdot \rchi_{X_0 \cap A}\right)d\mu 
&=&\int_{F}\underbrace{\left(\frac{\mu(A'\cap B|\EuScript{F})}{\mu(A'|\EuScript{F})}\cdot \rchi_{\{x \in X: \mu(A'|\EuScript{F})(x)>0\}}\right)}_{\EuScript{F}\text{-measurable function on}\;X}\rchi_{A'}\,d\mu \\
&=&\int_{F}\left(\frac{\mu(A'\cap B|\EuScript{F})}{\mu(A'|\EuScript{F})}\cdot \rchi_{\{x \in X: \mu(A'|\EuScript{F})(x)>0\}}\right)\mu(A'|\EuScript{F})\,d\mu \\
&=&\int_{F}\mu(A'\cap B|\EuScript{F}) \rchi_{\{x \in X: \mu(A'|\EuScript{F})(x)>0\}}\,d\mu.
\end{eqnarray*}
Using the fact that $0 \leq\mu(A'\cap B|\EuScript{F}) \leq  \mu(A'|\EuScript{F})$ $\mu$-a.e., we have $\mu(A'\cap B|\EuScript{F}) \rchi_{\{x \in X: \mu(A'|\EuScript{F})(x)>0\}} = \mu(A'\cap B|\EuScript{F})$ 
$\mu$-a.e. Therefore, we have
\begin{eqnarray*}
\int_{A'\cap F}\left(\sum\limits_{A \in \alpha}\frac{\mu(A\cap B|\EuScript{F})}{\mu(A|\EuScript{F})}\cdot \rchi_{X_0 \cap A}\right)d\mu
&=& \int_{F}\mu(A'\cap B|\EuScript{F})\,d\mu 
= \int_{F}\rchi_{A'\cap B}\,d\mu 
= \int_{A'\cap F}\rchi_{B}\,d\mu.
\end{eqnarray*}

Step 3. Let us consider the collection
\[\mathscr{C} \defeq \left\{\left(\bigcap \alpha'\right)\cap F : \alpha' \subseteq \alpha \;\text{and}\; F \in \EuScript{F}\right\},\]
where we adopt the usual convention that $\bigcap \emptyset = X$. It is easy to prove that $\mathscr{C}$ is a $\pi$-system on $X$ (i.e., 
a collection of subsets of $X$ closed under finite intersections)
which generates $\sigma(\alpha)\vee \EuScript{F}$. 

Then, if we define two measures $\nu_1$ and $\nu_2$ on $\sigma(\alpha)\vee \EuScript{F}$ by
\begin{equation}
\nu_1(C) = \int_{C}\left(\sum\limits_{A \in \alpha}\frac{\mu(A\cap B|\EuScript{F})}{\mu(A|\EuScript{F})}\cdot \rchi_{X_0 \cap A}\right)d\mu
\end{equation}
and
\begin{equation}
\nu_1(C) = \int_{C}\rchi_{B}\,d\mu
\end{equation}
for each $C \in \sigma(\mathscr{\alpha})\vee \EuScript{F}$, using equation (\ref{eqsafada}),
we easily verify that $\nu_1(C) = \nu_2(C)$ holds for all $C \in \mathscr{C}$. Since
$\nu_1$ and $\nu_2$ are finite measures that agree 
on a $\pi$-system which generates $\sigma(\mathscr{\alpha})\vee \EuScript{F}$ and satisfy $\nu_1(X) = \nu_2(X)$, we conclude that $\nu_1 = \nu_2$ (see Corolary $1.6.3$ from \cite{cohn:13}).
\end{proof}

\begin{teo}[Addition rule for information]\label{add} 
Let $\alpha$ and $\beta$ be $\mu$-partitions and let $\EuScript{F}$ be a sub-$\sigma$-algebra of $\mathcal{B}$.
Then, the equality
\begin{equation}\label{infoadd}
I_{\alpha \vee \beta|\EuScript{F}} = I_{\alpha|\EuScript{F}} + I_{\beta|\sigma(\alpha) \vee \EuScript{F}} 
\end{equation}
holds $\mu$-a.e.
\end{teo}
\begin{proof}
Observe that exists a set $X_0$ with measure $1$ such that for every $x$ in $X_0$ there are unique sets $A_0 \in \alpha$ 
and $B_0 \in \beta$ such that $x \in A_0 \cap B_0$. We
can assume that the equalities
\begin{eqnarray}
I_{\alpha \vee \beta|\EuScript{F}}(x) &=& \sum\limits_{C \in \alpha \vee \beta}-\log\mu(C|\EuScript{F})(x)\cdot\rchi_C(x), \\
I_{\alpha|\EuScript{F}}(x) &=& \sum\limits_{A \in \alpha}-\log\mu(A|\EuScript{F})(x)\cdot\rchi_A(x), 
\end{eqnarray} 
and
\begin{eqnarray}
I_{\beta|\sigma(\alpha)\vee \EuScript{F}}(x) &=& \sum\limits_{B \in \beta}-\log\mu(B|\sigma(\alpha)\vee \EuScript{F})(x)\cdot\rchi_B(x)
\end{eqnarray} 
hold for each point $x$ in $X_{0}$, and
\begin{eqnarray}
\mu(B|\sigma(\alpha)\vee \EuScript{F})(x) = \sum\limits_{A \in \alpha}\frac{\mu(A\cap B|\EuScript{F})}{\mu(A|\EuScript{F})}(x)\cdot \rchi_A(x) 
\end{eqnarray} 
also holds for every $x$ in $X_{0}$ for each $B \in \beta$.

Then, for all $x$ in $X_0$ we have the following equalities
\begin{equation*}
I_{\alpha \vee \beta|\EuScript{F}}(x) = -\log\mu(A_0 \cap B_0|\EuScript{F})(x),
\end{equation*}
\begin{equation*}
I_{\alpha|\EuScript{F}}(x) = -\log\mu(A_0|\EuScript{F})(x),
\end{equation*}
and
\begin{eqnarray*}
I_{\beta|\sigma(\alpha) \vee \EuScript{F}}(x) 
&=& -\log\mu(B_0|\sigma(\alpha)\vee \EuScript{F})(x) = - \log\frac{\mu(A_0 \cap B_0|\EuScript{F})}{\mu(A_0|\EuScript{F})}(x),
\end{eqnarray*}
where $A_{0}$ and $B_{0}$ are the unique elements of $\alpha$ and $\beta$, respectively, such that $x$ belongs to $A_{0}\cap B_{0}$. 
We conclude that the identity $I_{\alpha \vee \beta|\EuScript{F}}(x) =  I_{\alpha|\EuScript{F}}(x) + I_{\beta|\sigma(\alpha) \vee \EuScript{F}}(x)$ holds for
each $x$ in $X_{0}$.
\end{proof}

\begin{cor}\label{peppa}
Under the same hypotheses of Theorem \ref{add}, we have
\begin{enumerate}[label=(\alph*),ref=\alph*]
\item \label{epa} $H_{\mu}(\alpha \vee \beta|\EuScript{F}) = H_{\mu}(\alpha|\EuScript{F}) + H_{\mu}(\beta|\sigma(\alpha) \vee \EuScript{F})$,
\item $I_{\alpha \vee \beta} = I_{\alpha} + I_{\beta|\sigma(\alpha)}$ $\;\mu$-a.e.,
\item $H_{\mu}(\alpha \vee \beta) = H_{\mu}(\alpha) + H_{\mu}(\beta|\sigma(\alpha))$, and
\item $H_{\mu}(\alpha) \leq H_{\mu}(\beta) + H_{\mu}(\alpha|\sigma(\beta))$.
\end{enumerate}
\end{cor}
\begin{proof}
Part (a) follows by integrating equation (\ref{infoadd}). We obtain parts (b) and (c) by letting $\EuScript{F} =  \{\emptyset, X\}$ on
equation (\ref{infoadd}) and applying Example \ref{trivial}. Finally, part (d) follows from the fact that
$H_{\mu}(\alpha) \leq H_{\mu}(\alpha) + H_{\mu}(\beta|\sigma(\alpha)) = H_{\mu}(\alpha \vee \beta) = H_{\mu}(\beta \vee \alpha) = H_{\mu}(\beta) + H_{\mu}(\alpha|\sigma(\beta))$.
\end{proof}

Let $\alpha$ and $\beta$ be $\mu$-partitions. If each element $A$ of $\alpha$ is a union of
elements of $\beta$ (mod $\mu$) we will say that $\beta$ is finer than $\alpha$ and denote this fact by $\beta \succeq \alpha$. 
We will use $\alpha \approx \beta$ to denote the case where the conditions $\alpha \succeq \beta$ and $\beta \succeq \alpha$ hold.

\begin{prop}\label{xavasca}
Given two $\mu$-partitions, say $\alpha$ and $\beta$, the conditions
\begin{enumerate}[label=(\alph*),ref=\alph*]
\item $H_{\mu}(\alpha|\sigma(\beta)) = 0$,
\item $\beta \succeq \alpha$, and
\item $\alpha \vee \beta \approx \beta$
\end{enumerate}
are equivalent.
\end{prop}

\begin{proof}
First, let us find the explicit value of $H_{\mu}(\alpha|\sigma(\beta))$. If we apply Lemma \ref{lemmacond} in the particular case where $\EuScript{F} = \{\emptyset,X\}$,
it is easy to check that
\[I_{\alpha|\sigma(\beta)} = \sum\limits_{A \in \alpha}\sum\limits_{B \in \beta} -\log\frac{\mu(A\cap B)}{\mu(B)} \cdot \chi_{A\cap B} \quad\text{holds}\;\mu\text{-almost everywhere},\]
and then,
\[H_{\mu}(\alpha|\sigma(\beta)) = \sum\limits_{A \in \alpha}\sum\limits_{B \in \beta} -\log\frac{\mu(A\cap B)}{\mu(B)} \cdot \mu(A\cap B).\]
Therefore, we have $H_{\mu}(\alpha|\sigma(\beta)) = 0$ if and only if $\mu(A\cap B) = \mu(B)$ holds whenever the sets
$A \in \alpha$ and $B \in \beta$ satisfy $\mu(A \cap B) > 0$.

Let us show that (a) is equivalent to (b). Given an arbitrary element $A$ of $\alpha$, if we let $\widetilde{A} = \bigcup\{B \in \beta : \mu(A \cap B) > 0\}$, then we have
\begin{eqnarray*}
A \Delta \widetilde{A} &=& (A\backslash\widetilde{A}) \cup (\widetilde{A} \backslash A)\\
&=& A\backslash (A \cap \widetilde{A}) \cup (\widetilde{A} \backslash A)\\
&=& \left(A\Big\backslash \bigcup\limits_{\substack{B \in \beta \\ \mu(A \cap B)>0}}A\cap B\right) \cup \left(\bigcup\limits_{\substack{B \in \beta \\ \mu(A \cap B)>0}}B\backslash A\right)  \\
&=& \left(A\Big\backslash \bigcup\limits_{\substack{B \in \beta \\ \mu(A \cap B)>0}}A\cap B\right) \cup \underbrace{\left(\bigcup\limits_{\substack{B \in \beta \\ \mu(A \cap B)>0}}B\backslash(A\cap B)\right)}_{\text{measure zero}}  
\end{eqnarray*} 
and
\begin{eqnarray*}
\mu\left(A\Big\backslash \bigcup\limits_{\substack{B \in \beta \\ \mu(A \cap B)>0}}A\cap B\right) &=& \mu(A) - \mu \left(\bigcup\limits_{\substack{B \in \beta \\ \mu(A \cap B)>0}}A\cap B\right) \\
&=& \mu(A) - \mu \left(\bigcup\limits_{B \in \beta}A\cap B\right) \\
&=&  \mu(A) - \mu \left(A \cap \left(\bigcup \beta\right)\right) \\
&=& 0.
\end{eqnarray*} 
Thus, we conclude that $A = \widetilde{A}$ (mod $\mu$).
On the other hand, given $A \in \alpha$ and $B \in \beta$ satisfying $\mu(A\cap B) > 0$, let $\beta'$ be a collection of subsets of $\beta$ such that
$A = \bigcup \beta'$ (mod $\mu$). Since $\mu(A \cap B) = \mu\left(\left(\bigcup \beta'\right) \cap B\right) = \mu\left(\bigcup\limits_{B' \in \beta'} B' \cap B\right) > 0$,
it follows that $\bigcup\limits_{B' \in \beta'} B' \cap B = B$. Therefore, we have $\mu(A \cap B) = \mu(B)$.{}

Now let us show that (b) is equivalent to (c). 
For every $C \in \alpha \vee \beta$ and $B \in \beta$, the condition $\mu(C \cap B) > 0$ implies that 
$C = A \cap B$ for some element $A$ of $\alpha$. Using the fact that $\beta \succeq \alpha$ we obtain $\mu(C \cap B) = \mu(C) = \mu(B)$, hence
$\alpha \vee \beta \approx \beta$. Conversely, given $A \in \alpha$ and $B \in \beta$, the condition $\mu(A \cap B) > 0$ implies that $A \cap B$ and $B$ are respectively elements of
$\alpha \vee \beta$ and $\beta$ such that $\mu((A \cap B) \cap B) >0$. Thus $\mu(A \cap B) = \mu(B)$.
\end{proof}

\begin{cor}\label{info2}
Let $\alpha$ and $\beta$ be $\mu$-partitions such that $\alpha \approx \beta$ and let $\EuScript{F}$ be a sub-$\sigma$-algebra of $\mathcal{B}$. 
Then, we have $I_{\alpha|\EuScript{F}} = I_{\beta|\EuScript{F}}$ $\mu$-a.e. and $H_{\mu}{(\alpha|\EuScript{F})} = H_{\mu}{(\beta|\EuScript{F})}$.
\end{cor}
\begin{proof}
It is easy to check that the equalities
\begin{equation*}
I_{\alpha|\EuScript{F}} = \sum\limits_{A \in \alpha}\sum\limits_{B \in \beta} -\log\mu(A|\EuScript{F}) \cdot \rchi_{A\cap B}
\end{equation*}
and
\begin{equation*}
I_{\beta|\EuScript{F}} = \sum\limits_{A \in \alpha}\sum\limits_{B \in \beta} -\log\mu(B|\EuScript{F}) \cdot \rchi_{A\cap B} 
\end{equation*}
hold $\mu$-almost everywhere. Suppose that $A$ and $B$ belong to $\alpha$ and $\beta$, respectively. Let us show that $-\log\mu(A|\EuScript{F}) \cdot \rchi_{A\cap B} = -\log\mu(B|\EuScript{F}) \cdot \rchi_{A\cap B}$ $\mu$-a.e.
Note that this result easily follows in the case where $\mu(A \cap B) = 0$.
On the other hand, if $\mu(A \cap B) > 0$, we can use the fact that $H_{\mu}(\alpha |\sigma(\beta)) = H_{\mu}(\beta |\sigma(\alpha)) =0$ to obtain
$\mu(A \cap B) = \mu(B) = \mu(A)$ and conclude that $\mu(A \Delta B) = 0$. It follows that $\mu(A|\EuScript{F}) = \mu(B|\EuScript{F})$ $\mu$-a.e., and then, the equality
$-\log\mu(A|\EuScript{F}) \cdot \rchi_{A\cap B} = -\log\mu(B|\EuScript{F}) \cdot \rchi_{A\cap B}$ holds $\mu$-almost everywhere.
We conclude that $I_{\alpha|\EuScript{F}} = I_{\beta|\EuScript{F}}$ $\mu$-a.e., and by integration, we have $H_{\mu}{(\alpha|\EuScript{F})} = H_{\mu}{(\beta|\EuScript{F})}$. 
\end{proof}

\begin{teo}[Monotonicity of conditional entropy]\label{xuxacapeta}
Let $\alpha$ and $\beta$ be $\mu$-partitions and let $\EuScript{F}, \EuScript{F}_1$, and $\EuScript{F}_2$ be sub-$\sigma$-algebras of $\mathcal{B}$.
Then the following statements hold.
\begin{enumerate}[label=(\alph*),ref=\alph*]
\item \label{xuxaxxxa} If $\EuScript{F}_2 \subseteq \EuScript{F}_1$, then $H_{\mu}(\alpha|\EuScript{F}_1) \leq H_{\mu}(\alpha|\EuScript{F}_2)$. In particular, we have $H_{\mu}(\alpha|\EuScript{F}) \leq H_{\mu}(\alpha)$. 

\item If $\beta \succeq \alpha$, then $H_{\mu}(\beta|\EuScript{F}) \geq H_{\mu}(\alpha|\EuScript{F})$. In particular, we have $H_{\mu}(\beta) \geq H_{\mu}(\alpha)$.

\item\label{xuxac} The inequality $H_{\mu}(\alpha \vee \beta | \EuScript{F}) \leq H_{\mu}(\alpha| \EuScript{F}) + H_{\mu}(\beta | \EuScript{F})$ holds. In particular, we have $H_{\mu}(\alpha\vee\beta) \leq H_{\mu}(\alpha) + H_{\mu}(\beta)$. 
\end{enumerate}
\end{teo}
\begin{proof}
In order to prove part (a), it is convenient to assume without loss that 
$0 \leq \mu(A|\EuScript{F}_i) \leq 1$ on $X$ for each $A \in \alpha$ and $i \in \{1,2\}$. Let us consider the convex function $\Phi : [0,+\infty) \rightarrow \mathbb{R}$ 
given by $\Phi(x) = x \log x$, where we adopt the usual convention that $0 \log 0 = 0$.
Using Jensen's inequality for conditional expectations, we obtain
\begin{equation}\label{mariola}
\Phi \circ \mu(A|\EuScript{F}_2) = \Phi \circ \mathbb{E}_{\mu}\left[\mu(A|\EuScript{F}_1)|\EuScript{F}_2\right] \leq \mathbb{E}_{\mu}\left[\Phi \circ \mu(A|\EuScript{F}_1)|\EuScript{F}_2\right] \quad \mu\text{-a.e.}
\end{equation}
for each $A \in \alpha$. Thus, we have
\begin{eqnarray*}
H_{\mu}(\alpha|\EuScript{F}_1) 
&=& \sum\limits_{A \in \alpha}\int_{X}\underbrace{\left(-\log\mu(A|\EuScript{F}_1)\cdot\rchi_{\{x \in X: \mu(A|\EuScript{F}_1)(x) > 0\}} \right)}_{\EuScript{F}_1\text{-measurable function on}\;X}\cdot \rchi_{A}\,d\mu \\
&=& \sum\limits_{A \in \alpha}\int_{X}\left(-\log\mu(A|\EuScript{F}_1)\cdot\rchi_{\{x \in X: \mu(A|\EuScript{F}_1)(x) > 0\}} \right)\cdot \mu(A|\EuScript{F}_1)\,d\mu \\
&=& \sum\limits_{A \in \alpha}\int_{X} -\Phi \circ \mu(A|\EuScript{F}_1)\,d\mu \\
&=& \sum\limits_{A \in \alpha}\int_{X} \mathbb{E}_{\mu}\left[-\Phi \circ \mu(A|\EuScript{F}_1)|\EuScript{F}_2\right]\,d\mu \\
&\leq& \sum\limits_{A \in \alpha}\int_{X} -\Phi \circ \mu(A|\EuScript{F}_2)\,d\mu \\
&=& H_{\mu}(\alpha|\EuScript{F}_2).
\end{eqnarray*}
Now, let us prove part (b). Using Corolary \ref{peppa}(\ref{epa}), we have
\begin{eqnarray}
H_{\mu}(\alpha|\EuScript{F}) \leq H_{\mu}(\alpha|\EuScript{F}) + H_{\mu}(\beta| \sigma(\alpha) \vee \EuScript{F}) = H_{\mu}(\alpha \vee \beta|\EuScript{F}).
\end{eqnarray}
Thus, the result follows by applying Proposition \ref{xavasca} and Corolary \ref{info2}. 
For part (c), note that if we use Corolary \ref{peppa}(\ref{epa}) and part (a) of this theorem, we obtain
\begin{eqnarray*}
H_{\mu}(\alpha \vee \beta | \EuScript{F}) &=& H_{\mu}(\alpha| \EuScript{F}) + H_{\mu}(\beta |\sigma(\alpha) \vee \EuScript{F}) \\
&\leq& H_{\mu}(\alpha| \EuScript{F}) + H_{\mu}(\beta |\EuScript{F}).
\end{eqnarray*}
\end{proof}

\begin{cor}
Let $\alpha, \beta$, and $\gamma$ be $\mu$-partitions of $X$. If $\beta \succeq \alpha$, then
\begin{equation}
H_{\mu}(\gamma|\sigma(\beta)) \leq H_{\mu}(\gamma|\sigma(\alpha)).	
\end{equation}	
\end{cor}
\begin{proof}
Using Corolary \ref{peppa}(\ref{epa}), Proposition \ref{xavasca}, and Theorem \ref{xuxacapeta}(\ref{xuxaxxxa}), we obtain
\begin{eqnarray*}
H_{\mu}(\gamma|\sigma(\beta)) \leq H_{\mu}(\alpha \vee \gamma|\sigma(\beta)) = \underbrace{H_{\mu}(\alpha|\sigma(\beta))}_{= 0} + H_{\mu}(\gamma|\sigma(\alpha) \vee \sigma(\beta)) \leq{}
H_{\mu}(\gamma|\sigma(\alpha)).  	
\end{eqnarray*}	
\end{proof}

\section{Entropy of dynamical Systems}
In this section we introduce the concept of entropy for a special kind of dynamical system, the so-called measure preserving dynamical systems. 

\begin{mydef}
A measure preserving dynamical system (m.p.d.s.) is a quadruple $(X,\mathcal{B},\mu,T)$, where
\begin{itemize}
\item[(a)] the triple $(X,\mathcal{B},\mu)$ is a probability space, and
\item[(b)] $T$ is a map that associates to each point $i$ in $\g$ a $\mathcal{B}$-measurable function $T^{i}:X \rightarrow X$ such that $T^{i}_{\ast}\mu = \mu$, 
and satisfies the identities
\begin{equation}\label{chinelo}
T^{\textbf{0}} = \mathsf{id}_{X},
\end{equation}
where $\mathsf{id}_{X}$ is the identity mapping of $X$, and
\begin{equation}\label{chinela}
T^{i+j} = T^{i}\circ T^{j}
\end{equation}
for every $i$ and $j$ in $\g$.
In other words, $T$ is a $\mathcal{B}$-measurable action of the group $\g = \zd$ or of the monoid $\g = \zdp$ on $X$ which
preserves the measure $\mu$.
\end{itemize}
\end{mydef} 

\begin{ex}[Bernoulli shifts]\label{bernoulli}
Let $X$ be the $\g$-full shift over the alphabet $\mathcal{A}$, and let $\mathcal{B}$ be the Borel $\sigma$-algebra of $\mathcal{A}^{\g}$.
Given a probability measure $\nu$ on the measurable space $(\mathcal{A},\EuScript{E})$, where $\EuScript{E}$ is the power set of $\mathcal{A}$, 
let us denote the product measure $\nu^{\g}$ by $\mu$ 
(recall that $\mu$ is defined on the product $\sigma$-algebra $\EuScript{E}^{\g}$, which coincides with $\mathcal{B}$). Thus, condition (a) is satisfied.

Now, let $T$ be the map that associates to each point $i$ in $\g$ the translation $\sigma^{i}$ by $i$. It remains to prove that each map $T^{i}$ leaves the measure $\mu$ invariant. 
In order to do that, let us find a $\pi$-system $\mathscr{C}$ on $X$ which generates the $\sigma$-algebra
$\mathcal{B}$ such that $T^{i}_{\ast}\mu(C) = \mu(C)$ holds for each $i$ in $\g$ and each $C$ in $\mathscr{C}$, and finally conclude that the
identity $T^{i}_{\ast}\mu = \mu$ holds for each $i$ in $\g$.
Let $\mathscr{C}_{0} = \{\emptyset\}$. For each positive integer $n$, let us define a collection $\mathscr{C}_{n}$ of cylinder sets by letting $\mathscr{C}_{n} = \{[\omega] : \omega \in \mathcal{A}^{\Lambda_{n}}\}$,
where for each $\omega$ in $\mathcal{A}^{\Lambda_{n}}$ the cylinder $[\omega]$ is defined by $[\omega] = \{x \in \gds : x_{\Lambda_{n}} = \omega\}$. 
It is easy to check that the collection $\mathscr{C} = \bigcup\limits_{n \geq 0}\mathscr{C}_{n}$ satisfies the required properties.
Thus, the quadruple $(X,\mathcal{B},\mu,T)$ is a m.p.d.s.
\end{ex}

In the remainder of this section we will always consider a fixed measure preserving dynamical system $(X,\mathcal{B},\mu,T)$.{}

As previously mentioned, our main objective in this section is to formulate the concept of entropy of a m.p.d.s. Note that we need to define this quantity in such a way 
that it represents the gain of information about the system taking into account 
the fact that a dynamic was introduced on it. In order to do so, we will use the entropy of the partitions given as follows.

\begin{lemma}
Let $\alpha$ be a $\mu$-partition (resp. partition) of $X$. 
\begin{itemize}
\item[(a)] For each point $i$ in $\g$, the collection 
\begin{equation}
T^{-i}\alpha \defeq \left\{T^{-i}(A) : A \in \alpha\right\}, 
\end{equation}
where $T^{-i}(A)$ denotes the preimage of $A$ under $T^{i}$, is a $\mu$-partition (resp. partition) of $X$.	
\item[(b)] Given a nonempty finite subset $\Lambda$ of $\g$, the collection
\begin{equation}
\bigvee\limits_{i \in \Lambda} T^{-i}\alpha \defeq \left\{\bigcap\limits_{i \in \Lambda}A_{i} : A_i \in T^{-i}\alpha\;\,\text{for each}\;i \in \Lambda	\right\}	
\end{equation}
is also a $\mu$-partition (resp. partition) of $X$. We will often denote $\bigvee\limits_{i \in \Lambda} T^{-i}\alpha$ by $\alpha^{\Lambda}$.
\end{itemize}
\end{lemma}

\begin{proof}
It is easy to verify that $T^{-i}\alpha$ is a countable collection of elements of $\mathcal{B}$. Note that
\begin{equation}\label{aa1}
\bigcup T^{-i}\alpha = \bigcup\limits_{A \in \alpha} T^{-i}(A) = T^{-i}\left(\bigcup \alpha\right),
\end{equation}
and for each pair $B_1, B_2$ of distinct elements of $T^{-i}\alpha$ there are distinct sets $A_1$ and $A_2$ in $\alpha$ such that 
\begin{equation}\label{aa2}
B_1 \cap B_2 = T^{-i}(A_1) \cap T^{-i}(A_2) = T^{-i}(A_1 \cap A_2).
\end{equation}
Thus, part (a) follows from equations (\ref{aa1}) and (\ref{aa2}).

It is easy to check that $\alpha^{\Lambda}$ is a countable collection of elements of $\mathcal{B}$. Observe that
\begin{equation}\label{aa3}
\bigcup \alpha^{\Lambda} = \bigcap\limits_{i \in \Lambda}\left(\bigcup T^{-i}\alpha\right),
\end{equation}
and for any two distinct sets $A$ and $B$ in $\alpha^{\Lambda}$, one can find an element $j$ of $\Lambda$ together with
distinct sets $A_{j}$ and $B_j$ in $T^{-j}\alpha$ such that
\begin{equation}\label{aa4}
A \cap B \subseteq A_j \cap B_j.
\end{equation}	
Thus, part (b) follows from equations (\ref{aa3}) and (\ref{aa4}).
\end{proof}

\begin{remark}
It is easy to check that if we let $\Lambda$ and $\Delta$ be nonempty finite subsets of $\g$ such that $\Lambda \subseteq \Delta$, it follows that
$\alpha^{\Delta} \succeq \alpha^{\Lambda}$.	In particular, for each point $i$ in $\Delta$, if we let $\Lambda = \{i\}$, then we have $\alpha^{\Delta} \succeq T^{-i}\alpha$. 
\end{remark}

Now, let us derive a few properties related to the entropy of the partitions defined above.

\begin{lemma}\label{xsatan}
Let $\alpha$ be a $\mu$-partition and let $\EuScript{F}$ be a sub-$\sigma$-algebra of $\mathcal{B}$. 
\begin{enumerate}[label=(\alph*),ref=\alph*]
\item\label{xuxaaa} For each point $i$ in $\g$, if we let $T^{-i}\EuScript{F}$ be the sub-$\sigma$-algebra of $\mathcal{B}$ given by $T^{-i}\EuScript{F} = \left\{T^{-i}(F) : F \in \EuScript{F}\right\}$, then the equality	
\begin{equation}\label{uga}
I_{T^{-i}\alpha|T^{-i}\EuScript{F}} = I_{\alpha|\EuScript{F}}\circ T^{i}
\end{equation}
holds $\mu$-a.e., and
\begin{equation}\label{ugauga}
H_{\mu}(T^{-i}\alpha|T^{-i}\EuScript{F}) = H_{\mu}(\alpha|\EuScript{F}).
\end{equation}
In particular, we have $I_{T^{-i}\alpha} = I_{\alpha}\circ T^{i}$ $\mu$-a.e. and $H_{\mu}(T^{-i}\alpha) = H_{\mu}(\alpha)$.{}

\item\label{pipa} If $\Lambda$ is a nonempty finite subset of $\g$, then
\begin{equation}\label{xuxasatanasa}
H_{\mu}(\alpha^{\Lambda}|\EuScript{F}) \leq \sum\limits_{i \in \Lambda} H_{\mu}(T^{-i}\alpha|\EuScript{F}). 
\end{equation}
In particular, we have
\begin{equation}\label{xuxasatanas}
H_{\mu}(\alpha^{\Lambda}) \leq |\Lambda| \cdot H_{\mu}(\alpha). 
\end{equation}
\end{enumerate}
\end{lemma}

\begin{proof}
Let us prove part (a). Observe that for each $A \in \alpha$ and each $F \in \EuScript{F}$ we have
\begin{equation*}
\int_{F}\mu(A|\EuScript{F})\,d\mu = \int_{F}\mu(A|\EuScript{F})\,d (T^{i}_{\ast}\mu) = \int_{T^{-i}(F)}\mu(A|\EuScript{F})\circ T^{i}\,d\mu,	
\end{equation*}	
on the other hand, we also have
\begin{equation*}
\int_{F}\mu(A|\EuScript{F})\,d\mu = \int_{F}\rchi_{A}\,d\mu = \mu(A\cap F) = \mu(T^{-i}(A)\cap T^{-i}(F)) = \int_{T^{-i}(F)}\rchi_{T^{-i}(A)}\,d\mu.	
\end{equation*}	
Then, for each element $A$ of $\alpha$ the equation
\begin{equation}
\int_{T^{-i}(F)}\rchi_{T^{-i}(A)}\,d\mu = \int_{T^{-i}(F)}\mu(A|\EuScript{F})\circ T^{i}\,d\mu	
\end{equation} 
holds for every $F \in \EuScript{F}$, moreover, $\mu(A|\EuScript{F})\circ T^{i}$ is a measurable function with respect to the $\sigma$-algebra $T^{-i}\EuScript{F}$. 
By the definition of conditional expectation, it follows that $\mu(T^{-i}(A)|T^{-i}\EuScript{F}) = \mu(A|\EuScript{F})\circ T^{i}$ $\mu$-a.e.

Thus, the equalities
\begin{eqnarray*}
I_{T^{-i}\alpha|T^{-i}\EuScript{F}} &=& \sum\limits_{B \in \,T^{-i}\alpha} -\log \mu(B|T^{-i}\EuScript{F}) \cdot \rchi_{B}\\
&=& \sum\limits_{A \in \alpha} -\log \mu(T^{-i}(A)|T^{-i}\EuScript{F}) \cdot \rchi_{T^{-i}(A)} \\	
&=& \sum\limits_{A \in \alpha} -\log \mu(A|\EuScript{F})\circ T^{i} \cdot \rchi_{A}\circ T^{i} \\
&=& I_{\alpha|\EuScript{F}} \circ T^{i}
\end{eqnarray*}
hold $\mu$-almost everywhere. We obtain equation (\ref{ugauga}) by integrating equation (\ref{uga}) and using the fact that $T^{i}$ leaves the probability measure $\mu$ invariant.

Now, let us prove part (b). 
In the case where $\Lambda$ contains exactly one element, say $j$, it follows that $H_{\mu}(\alpha^\Lambda|\EuScript{F}) = H_{\mu}(T^{-j}\alpha|\EuScript{F})$. Let us suppose that
equation (\ref{xuxasatanasa}) holds whenever $\Lambda$ has $n$ elements. Now, if $\Lambda$ contains $n+1$ elements, choose an arbitrary element $j$ of $\Lambda$, and use 
Theorem \ref{xuxacapeta}(\ref{xuxac}) to obtain
\begin{eqnarray*}
H_{\mu}(\alpha^{\Lambda}|\EuScript{F}) &=& H_{\mu}(\alpha^{\Lambda\backslash \{j\}} \vee \alpha^{\{j\}}|\EuScript{F}) \leq H_{\mu}(\alpha^{\Lambda\backslash \{j\}}|\EuScript{F}) + H_{\mu}(\alpha^{\{j\}}|\EuScript{F}) \\
&\leq& \sum\limits_{i \in \Lambda\backslash\{j\}} H_{\mu}(T^{-i}\alpha|\EuScript{F}) + H_{\mu}(T^{-j}\alpha|\EuScript{F}) \\
&=& \sum\limits_{i \in \Lambda} H_{\mu}(T^{-i}\alpha|\EuScript{F}).   	
\end{eqnarray*}
Equation (\ref{xuxasatanas}) can be proved by letting $\EuScript{F} = \{\emptyset,X\}$ and applying the result obtained in part (a). 
\end{proof}

In the following, we will use the results obtained above to introduce the dynamical entropy of the system 
$(X,\mathcal{B},\mu,T)$ relative to a $\mu$-partition. Later, its entropy will be defined as the supremum of the set consisting of all
dynamical entropies relative to finite partitions of $X$.

\begin{teo}[Dynamical entropy relative to a $\mu$-partition]
Let $(X,\mathcal{B},\mu,T)$ be a m.p.d.s. and let $\alpha$ be a $\mu$-partition with $H_{\mu}(\alpha) < +\infty$. Then, we have the equality
\begin{equation}\label{dyentropy}
\inf\limits_{n \in \mathbb{N}}\frac{1}{|\Lambda_{n}|}H_{\mu}(\alpha^{\Lambda_{n}}) = \lim\limits_{n \to \infty}\frac{1}{|\Lambda_{n}|}H_{\mu}(\alpha^{\Lambda_{n}}).	
\end{equation}
The quantity defined above is a nonnegative real number that will be denoted by $h_{\mu}(T,\alpha)$, and is often called the dynamical entropy of the system $(X,\mathcal{B},\mu,T)$ relative to the $\mu$-partition $\alpha$.
\end{teo}

Before entering into the proof of this theorem observe that under the hypotheses presented above, Lemma \ref{xsatan}(\ref{pipa}) implies that for each positive integer $n$, the quantity given by
$\frac{1}{|\Lambda_{n}|}H_{\mu}(\alpha^{\Lambda_{n}})$ is a nonnegative real number, since it is less than $H_{\mu}(\alpha)$. Thus, the left-hand side of equation (\ref{dyentropy}) is also a nonnegative real number.
  
\begin{proof}
For each positive integer $m$, let $l_{m}$ be the side length of the cube $\Lambda_{m}$ (in case $\g = \zdp$ we have $l_{m} = m$, and, in case
$\g = \zd$ we have $l_{m} = 2{m} - 1$). Let us consider two positive integers $m$ and $n$. It is straightforward to show that $\g = \bigcup\limits_{j \in l_{m}\g}(\Lambda_{m}+j)$, where
$l_{m}\g = \{l_{m}\cdot i : i \in \g\}$.
If we let $V_{m,n} = \left\{j \in l_{m}\g : (\Lambda_{m}+j) \cap \Lambda_{n} \neq \emptyset\right\}$, then it follows from the inclusion $V_{m,n} \subseteq \Lambda_{m+n}$ that
$\Lambda_{n} \subseteq \widetilde{\Lambda}_{m} \defeq \bigcup\limits_{j \in V_{m,n}}(\Lambda_{m} + j) \subseteq \Lambda_{2m+n}$. Since $|\widetilde{\Lambda}_{m}| = \sum\limits_{j \in V_{m,n}}|\Lambda_{m} + j| = |V_{m,n}|\cdot|\Lambda_{m}| \leq |\Lambda_{2m+n}|$, we obtain
\begin{eqnarray*}
H_{\mu}(\alpha^{\Lambda_{n}}) &\leq& H_{\mu}(\alpha^{\widetilde{\Lambda}_{m}}) = H_{\mu}\left((\alpha^{{\Lambda}_{m}})^{V_{m,n}}\right) \leq |V_{m,n}| \cdot H_{\mu}(\alpha^{\Lambda_{m}}) \leq \frac{|\Lambda_{2m+n}|}{|\Lambda_{m}|} H_{\mu}(\alpha^{\Lambda_{m}}).
\end{eqnarray*}
Thus, the inequality
\begin{equation*}
\frac{1}{|\Lambda_{n}|}H_{\mu}(\alpha^{\Lambda_{n}}) \leq \frac{|\Lambda_{2m+n}|}{|\Lambda_{n}|}\cdot\frac{1}{|\Lambda_{m}|} H_{\mu}(\alpha^{\Lambda_{m}})	
\end{equation*}
holds for each $m$ and $n$. It follows that
\begin{equation*}
\limsup\limits_{n \to \infty} \frac{1}{|\Lambda_{n}|}H_{\mu}(\alpha^{\Lambda_{n}}) \leq \limsup\limits_{n \to \infty} \frac{|\Lambda_{2m+n}|}{|\Lambda_{n}|}\frac{1}{|\Lambda_{m}|} H_{\mu}(\alpha^{\Lambda_{m}}) =  \frac{1}{|\Lambda_{m}|} H_{\mu}(\alpha^{\Lambda_{m}})  	
\end{equation*}
holds for every positive integer $m$. Therefore, we have 
\begin{equation*}
\limsup\limits_{n \to \infty} \frac{1}{|\Lambda_{n}|}H_{\mu}(\alpha^{\Lambda_{n}}) \leq \inf\limits_{m \in \mathbb{N}} \frac{1}{|\Lambda_{m}|} H_{\mu}(\alpha^{\Lambda_{m}}) \leq  \liminf\limits_{n \to \infty}\frac{1}{|\Lambda_{n}|} H_{\mu}(\alpha^{\Lambda_{n}}), 	
\end{equation*}
and the result follows.
\end{proof}

\begin{ex}[Bernoulli shifts II]\label{xanex}
Suppose that we are in the same setting as in Example \ref{bernoulli}. Let $\alpha$ be the partition of $X$ given by
\begin{equation}
\alpha = \left\{\pi^{-1}_{\mathbf{0}}(\{a\}) : a \in \mathcal{A}\right\}.
\end{equation} 
It is easy to check that for each positive integer $n$, we have $\alpha^{\Lambda_{n}} = \{[\omega] : \omega \in \mathcal{A}^{\Lambda_{n}}\}$.
If we denote by $p(a)$ the value of $\nu(\{a\})$, we obtain
\begin{eqnarray*}
H_{\mu}(\alpha^{\Lambda_{n}}) &=& \sum\limits_{\omega \in \mathcal{A}^{\Lambda_{n}}}-\mu([\omega]) \cdot \log \mu([\omega]) = - \sum\limits_{\omega \in \mathcal{A}^{\Lambda_{n}}} \left(\prod\limits_{i \in \Lambda_{n}}p(\omega_{i})\right) \cdot \log \left(\prod\limits_{i \in \Lambda_{n}}p(\omega_{i})\right)\\	
&=& - \sum\limits_{\omega \in \mathcal{A}^{\Lambda_{n}}} \sum\limits_{j \in \Lambda_{n}} \left(\prod\limits_{i \in \Lambda_{n}}p(\omega_{i})\right) \cdot \log p(\omega_{j}) \\
\end{eqnarray*}

\begin{eqnarray*}
&=& - \sum\limits_{j \in \Lambda_{n}}\sum\limits_{\omega \in \mathcal{A}^{\Lambda_{n}}} \left(\prod\limits_{i \in \Lambda_{n}\backslash \{j\}} p(\omega_{i})\right) p(\omega_{j}) \cdot \log p(\omega_{j}) \\
&=& - \sum\limits_{j \in \Lambda_{n}}\sum\limits_{\omega_{j} \in \mathcal{A}} p(\omega_{j}) \cdot \log p(\omega_{j}) \\
&=& - |\Lambda_{n}| \sum\limits_{a \in \mathcal{A}} p(a) \cdot \log p(a). 
\end{eqnarray*}
Thus, for this particular partition $\alpha$, we have $h_{\mu}(T,\alpha) = - \sum\limits_{a \in \mathcal{A}} p(a) \cdot \log p(a)$.
\end{ex}

\begin{mydef}[Entropy]
The entropy of a m.p.d.s. $(X,\mathcal{B},\mu,T)$ is defined by 
\begin{equation}\label{kse1}
h_{\mu}(T) \defeq \sup\{h_{\mu}(T,\alpha) : \alpha\;\text{is a finite partition of}\;X\}.	
\end{equation} 	
The quantity defined above is also called Kolmogorov-Sinai entropy. 
\end{mydef}

In the case where $\g = \mathbb{Z}_{+}$ or $\g = \mathbb{Z}$, we can interpret the quantity given by (\ref{kse1}) as being the maximum amount of information per unit of time that can be gained by an
observer that looks the system (with time evolution described by $T$) through a finite partition. 
In the following, we will show that makes no difference to the observer if he looks through a finite partition or through a $\mu$-partition with
finite entropy.

\begin{teo}[Entropy via $\mu$-partitions]\label{kse}
Let $(X,\mathcal{B},\mu,T)$ be a m.p.d.s. Then, we have
\begin{equation}\label{entropyeq}
h_{\mu}(T) = \sup\{h_{\mu}(T,\alpha) : \alpha\;\text{is a $\mu$-partition with $H_{\mu}(\alpha)< +\infty$}\}.		
\end{equation}	
\end{teo}

In order to prove this theorem, let us show the following preliminary results.
\begin{lemma}\label{coco}
For any two $\mu$-partitions $\alpha$ and $\beta$ with finite entropy, we have
\begin{equation}\label{eqf}
h_{\mu}(T,\beta) \leq h_{\mu}(T,\alpha) + H_{\mu}(\beta|\sigma(\alpha)).	
\end{equation}	
\end{lemma}
\begin{proof}
Our claim follows by using the properties of the entropy obtained in the previous section and applying Lemma \ref{xsatan}. In fact, 

\begin{eqnarray*}
H_{\mu}(\beta^{\Lambda_{n}}) &\leq& H_{\mu}(\alpha^{\Lambda_{n}} \vee \beta^{\Lambda_{n}}) = H_{\mu}(\alpha^{\Lambda_{n}}) + H_{\mu}(\beta^{\Lambda_{n}}|\sigma(\alpha^{\Lambda_{n}}))\\
&\leq& H_{\mu}(\alpha^{\Lambda_{n}}) + \sum\limits_{i \in \Lambda_{n}}H_{\mu}(T^{-i}\beta|\sigma(\alpha^{\Lambda_{n}})) \\	
&\leq& H_{\mu}(\alpha^{\Lambda_{n}}) + \sum\limits_{i \in \Lambda_{n}}H_{\mu}(T^{-i}\beta|\sigma(T^{-i}\alpha)) \\ 
&=& H_{\mu}(\alpha^{\Lambda_{n}}) + \sum\limits_{i \in \Lambda_{n}}H_{\mu}(T^{-i}\beta|T^{-i}\sigma(\alpha)) \\ 
&\leq& H_{\mu}(\alpha^{\Lambda_{n}}) + |\Lambda_{n}| \cdot H_{\mu}(\beta|\sigma(\alpha))	
\end{eqnarray*}
holds for each positive integer $n$. Thus, if we take the limit as $n$ approaches infinity on the equation
\begin{equation}
\frac{1}{|\Lambda_{n}|} H_{\mu}(\beta^{\Lambda_{n}}) \leq \frac{1}{|\Lambda_{n}|} H_{\mu}(\alpha^{\Lambda_{n}}) + H_{\mu}(\beta|\sigma(\alpha)),	
\end{equation}	
the result follows.
\end{proof}

\begin{lemma}\label{choko}
In the case	where $\alpha$ and $\beta$ are $\mu$-partitions with finite entropy such that $\alpha \succeq \beta$, we have
\begin{equation}
h_{\mu}(T,\beta) \leq h_{\mu}(T,\alpha).	
\end{equation}
\end{lemma}

\begin{proof}
The result easily follows from Proposition \ref{xavasca} and Lemma \ref{coco}.
\end{proof}

\begin{prop}\label{desgraça}
Let $(X,\mathcal{B},\mu,T)$ be a m.p.d.s. and let $\alpha$ be a $\mu$-partition with $H_{\mu}(\alpha) < + \infty$. Then, we have
\begin{equation}
h_{\mu}(T,\alpha) = \sup\left\{h_{\mu}(T,\beta) : \beta \;\text{is a finite partition such that}\; \alpha \succeq \beta\right\}.		
\end{equation}
\end{prop}

\begin{proof}
Observe that due to Lemma \ref{choko}, it is sufficient to prove that for each positive number $\epsilon$ there is a finite partition $\beta$ with
$\alpha \succeq \beta$ satisfying $h_{\mu}(T,\alpha) - \epsilon < h_{\mu}(T,\beta)$. Since $\sum\limits_{A \in \alpha}-\mu(A)\cdot \log \mu(A) < +\infty$, it follows that
for every $\epsilon > 0$ there is a finite subset $\alpha '$ of $\alpha$ such that $\sum\limits_{A \in \alpha\backslash \alpha '}-\mu(A)\cdot \log \mu(A) < \epsilon$. If we let
$N_{0} = \bigcup\{A \cap B : A,B \in \alpha'\;\text{such that}\; A \neq B\}$ and $\widetilde{A} = (X \backslash \bigcup \alpha') \cup N_{0}$, then the collection
$\beta = \{A\backslash N_{0} : A \in \alpha'\}\cup \{\widetilde{A}\}$ is a finite partition of $X$ that satisfies $\alpha \succeq \beta$, and
\begin{eqnarray*}
H_{\mu}(\alpha|\sigma(\beta)) &=& \sum\limits_{A \in \alpha}\sum\limits_{B \in \beta} -\mu(A \cap B)\cdot\log\frac{\mu(A \cap B)}{\mu(B)} 
= \sum\limits_{A \in \alpha} -\mu(A \cap \widetilde{A})\cdot\log\frac{\mu(A \cap \widetilde{A})}{\mu(\widetilde{A})} \\
&=& \sum\limits_{A \in \alpha \backslash \alpha '} -\mu(A \cap \widetilde{A})\cdot\log\frac{\mu(A \cap \widetilde{A})}{\mu(\widetilde{A})} 
\leq \sum\limits_{A \in \alpha \backslash \alpha '} -\mu(A \cap \widetilde{A})\cdot\log\mu(A \cap \widetilde{A}) \\
&=& \sum\limits_{A \in \alpha \backslash \alpha '} -\mu(A)\cdot\log \mu(A) < \epsilon.
\end{eqnarray*} 
Thus, we conclude the proof by using Lemma \ref{coco}.
\end{proof}

\begin{proof}[Proof of Theorem \ref{kse}]
Since every finite partition has finite entropy, it follows that
$h_{\mu}(T) \leq \sup\{h_{\mu}(T,\alpha) : \alpha\;\text{is a $\mu$-partition with $H_{\mu}(\alpha)< +\infty$}\}$. On the other hand, if we let $\alpha$ be a $\mu$-partition
with finite entropy and let $\epsilon$ be a positive number, according to Proposition \ref{desgraça} there is a finite partition $\beta$ of $X$ such that
$h_{\mu}(T,\alpha) - \epsilon < h_{\mu}(T,\beta) \leq h_{\mu}(T)$. It implies that $\sup\{h_{\mu}(T,\alpha) : \alpha\;\text{is a $\mu$-partition with $H_{\mu}(\alpha)< +\infty$}\} \leq h_{\mu}(T) + \epsilon$ 
holds for each positive number $\epsilon$, thus the result follows.
\end{proof}

The first question that naturally arises is: Under which conditions does the supremum that occurs in equation (\ref{entropyeq}) is attained? The answer for this question is provided by
Kolmogorov-Sinai theorem. Before we state this result, let us introduce some nomenclature. A $\mu$-partition $\alpha$ will be called a $\mu$-generator for
$(X,\mathcal{B},\mu,T)$ if the smallest $\sigma$-algebra that contains all the collections $T^{-i}\alpha$ coincides with $\mathcal{B}$.
The theorem mentioned above is very useful to compute the entropy of a m.p.d.s. once a $\mu$-generator (with finite entropy) is known. 

\begin{teo}[Kolmogorov-Sinai]\label{ks}
If $\alpha$ is a $\mu$-generator for $(X,\mathcal{B},\mu,T)$ such that $H_{\mu}(\alpha) < +\infty$, then $h_{\mu}(T) = h_{\mu}(T,\alpha)$. 	
\end{teo}
\begin{proof}
See Keller \cite{keller:98}.	
\end{proof}

\begin{ex}[Bernoulli shifts III]
Suppose that we are in the same setting as in Example \ref{bernoulli}. Let $\alpha$ be the partition
\begin{equation}
\alpha = \left\{\pi^{-1}_{\mathbf{0}}(\{a\}) : a \in \mathcal{A}\right\}.
\end{equation} 
It is easy to check that for each
$i$ in $\g$, we have $T^{-i}\alpha = \{\pi^{-1}_{i}(\{a\}) : a \in \mathcal{A}\}$. Clearly, the partition $\alpha$ defined above
is a $\mu$-generator for $(X,\mathcal{B},\mu,T)$. Thus, according to Example \ref{xanex} and Theorem \ref{ks}, it follows that
\begin{eqnarray}
h_{\mu}(T) = - \sum\limits_{a \in \mathcal{A}}p(a) \log p(a),	
\end{eqnarray}
where $p(a)$ denotes the value of $\nu(\{a\})$ for each $a$.
\end{ex}

\section{Pressure}\label{pressure}

Recall that in the previous section we studied a few properties of the entropy of measure preserving dynamical systems without making any topological assumption. 
In the following, in order to introduce the definition of pressure and of an equilibrium measure, we will always 
suppose that $X$ is a compact metrizable space together with its Borel $\sigma$-algebra and that $T$ acts continuously on $X$. This setting can be precisely formulated 
by introducing the concept of a topological dynamical system.

\begin{mydef}
A topological dynamical system (t.d.s.) is a pair $(X,T)$ consisting of
\begin{itemize}
\item[(a)] a nonempty compact metrizable space $X$, and
\item[(b)] a map $T$ that associates to each point $i$ in $\g$ a continuous function $T^{i} : X \rightarrow X$ such that  
\begin{equation*}
T^{\textbf{0}} = \mathsf{id}_{X},
\end{equation*}
and
\begin{equation*}
T^{i+j} = T^{i}\circ T^{j}
\end{equation*}
holds for every $i$ and $j$ in $\g$. In other words, $T$ is a continuous action of the group $\g = \zd$ or of the monoid $\g = \zdp$ on $X$.
\end{itemize}
\end{mydef}

\begin{ex}[Subshifts as dynamical systems]\label{subdysy}
Let $X$ be a nonempty subshift of $\gds$. Naturally, we will always consider the subshift $X$ endowed with the topology inherited from the full shift $\gds$. Observe that condition (a) above is satisfied, 
since $X$ is a closed subset of the compact metrizable space $\gds$. 
Let us define the shift action on $X$ as the map $T$ that associates to each point $i$ in $\g$ a function $T^{i} : X \rightarrow X$ given by
$T^{i}(x) = \sigma^{i}(x)$, where $\sigma^{i}$ is the shift by $i$. Using Proposition \ref{homeo}, one can easily verify that each map $T^{i}$ is continuous, 
and, according to Facts \ref{eta} and \ref{etaa}, it follows that $T^{\bf{0}} = \mathsf{id}_X$, and
$T^{i+j} = T^{i} \circ T^{j}$ holds for each $i$ and $j$ in $\g$. 
Thus, condition (b) follows. We conclude that the pair $(X,T)$ is a topological dynamical system.
\end{ex}

In the following, we will always let $(X,T)$ be a topological dynamical system, and assume that $X$ is endowed with its Borel $\sigma$-algebra. 
Under these assumptions, we immediately see that $T$ is a Borel measurable action of $\g$ on $X$. We will also let $\mathcal{M}(T)$ denote the set of all $T$-invariant Borel probability measures on $X$, i.e.,
the set $\mathcal{M}(T)$ consists of all Borel probability measures $\mu$ on $X$ such that $T^{i}_{\ast}\mu = \mu$ holds for each $i$ in $\g$.  
It is well known that $\mathcal{M}(T)$ is a nonempty, compact, convex subset of the set of all Borel probability measures on $X$ (see Keller \cite{keller:98}). 

\begin{mydef}
For each real-valued continuous function $f$ on $X$, we define its pressure by
\begin{equation}\label{pressuredef}
p(f) \defeq \sup\limits_{\mu \in \mathcal{M}(T)}\left\{h_{\mu}(T) + \int_{X} f\,d\mu\right\}.	
\end{equation}
In particular, if $f$ is identically zero, the quantity
\begin{equation}
p(0) = \sup\limits_{\mu \in \mathcal{M}(T)}h_{\mu}(T)
\end{equation}	
is called the topological entropy of $T$.
\end{mydef}

An equilibrium measure for a continuous function on $X$ will be defined as being an element of $\mathcal{M}(T)$ such that the supremum in (\ref{pressuredef}) is attained.

\begin{mydef}
A $T$-invariant Borel probability measure $\mu$ on $X$ is said to be an equilibrium measure for a continuous function $f : X \rightarrow \mathbb{R}$ if 
\begin{equation}
p(f) = h_{\mu}(T) + \int_{X} f\,d\mu.	
\end{equation}
In the case where $\mu$ is an equilibrium measure for the identically zero function, we say that it is a measure of maximum entropy for $T$.
\end{mydef}



The t.d.s. defined in Example \ref{subdysy} is an important example of an expansive system (the definition is presented bellow). For such systems the expansivity property 
ensures the existence of an equilibrium measure and the finiteness of the topological entropy (see \cite{keller:98}). 
 
\begin{mydef}
Let $(X,T)$ be a t.d.s. and let $\rho$ be a metric that induces the topology of $X$. We say that $T$ is expansive if 
there is a positive number $\epsilon$ such that for any two distinct elements $x$ and $y$ of $X$ one can find a point $i$ in $\g$ satisfying $\rho(T^{i}x,T^{i}y) \geq \epsilon$.
\end{mydef}

\begin{remark}
Note that the property of expansiveness depends only on the topology of $X$, in the sense that the definition above does not depends on the choice of the metric $\rho$. 	
\end{remark}

\begin{ex}[Expansiveness of shift actions]\label{subexp}
Let us consider the t.d.s. $(X,T)$ defined in Example \ref{subdysy} and the metric $\rho$ defined by equation (\ref{metric}) restricted to $X \times X$. Let us show that the shift action $T$ is expansive. Given two distinct points $x$ and $y$ in $X$ there is an
element $i$ of $\g$ such that $x_{i} \neq y_{i}$. It follows that $T_{i}(x)_{\Lambda_{n}} \neq T_{i}(y)_{\Lambda_{n}}$ for each positive integer $n$, thus $\rho(x,y) = 1$.	
\end{ex}

%% file: cap-conceitos.tex
\chapter{Conformal measures}
\label{cap:res}

Our aim in this chapter is to provide some basic notions of conformal measures. The tools developed in this chapter will be used in Chapter \ref{cap:gibbs} to formulate the definition
of a Gibbs measure. For further references see \cite{feldman:77}, \cite{schpet}, \cite{aaronson:07}.

First, we introduce the concept of a Borel equivalence relation, and turn to some examples. At the end of this chapter,   
we finally introduce and study a few properties of conformal measures that will be essential to provide a precise formulation of the main results in this work. 
Due to its fundamental importance, all the results in this chapter are proved in detail. 

\section{Borel equivalence relations}\label{uia}

Let $X$ be an arbitrary set and $\EuScript{R} \subseteq X \times X$ an equivalence relation. We denote the equivalence class of an element
$x$ of $X$  by $\EuScript{R}(x) \defeq \{y \in X : (x,y) \in \EuScript{R}\}$, and, given a subset $A$ of $X$, we define its
$\EuScript{R}$-saturation by $\EuScript{R}(A) \defeq \bigcup\{\EuScript{R}(x) : x \in A\}$. In the case where $\EuScript{R}(x)$ is a countable set 
for each $x \in X$, then $\EuScript{R}$ is said to be a countable equivalence relation. 

Recall that a topological space $X$ is completely metrizable if there is a metric $\rho$ on $X$ compatible with
its topology such that the pair $(X,\rho)$ is a complete metric space. A completely metrizable separable space is called a
Polish space.
	
\begin{mydef}
Let $\EuScript{R}$ be an equivalence relation on a Polish space $X$. Then, we say that $\EuScript{R}$ is a Borel equivalence relation if 
it is a Borel subset of $X \times X$.
\end{mydef}

In what follows we provide some basic examples of countable Borel equivalence relations. We do this presentation as detailed as possible, 
since these examples play a fundamental role on the development of the following chapters.

\begin{ex}[Orbit equivalence relation]\label{group}
Let $X$ be a Polish space and let
\[\mathsf{Aut(}X\mathsf{)} \defeq \left\{f \in X^X : f\; \textnormal{is invertible, and both}\; f\; \text{and}\;f^{-1} \; \textnormal{are Borel measurable}\right\}\]
be the set of all Borel automorphisms of $X$. Note that $\mathsf{Aut(}X\mathsf{)}$ is a group with respect to the operation of composition of functions. Then, let us
consider a countable group $G \subseteq \mathsf{Aut(}X\mathsf{)}$ and define the orbit equivalence relation by
\[\EuScript{R}_G \defeq \big\{(x,y) \in X \times X:y=g(x)\;\text{for some}\;g\in G\big\}.\] 
Let us prove that $\EuScript{R}_G$ is a countable Borel equivalence relation on $X$.

\begin{itemize}
\item[(a)] For each $x \in X$, we have $(x,x) = (x,\mathsf{id}_X(x)) \in \EuScript{R}_G$, where $\mathsf{id}_{X}$ is the identity mapping of $X$. 

\item[(b)] Given two points $x$ and $y$ in $X$, if the pair $(x,y)$ belongs to $\EuScript{R}_G$, let us consider the element $g \in G$ such that $y=g(x)$. Therefore, we have	
$(y,x) = (y,g^{-1}(y)) \in \EuScript{R}_G$.

\item[(c)] For any $x,y,z \in X$, if each pair $(x,y)$ and $(y,z)$ belongs to $\EuScript{R}_G$, then there are two elements $g_1, g_2 \in G$ such that 
$y=g_1(x)$ and $z=g_2(y)$. It follows that $(x,z) = \left(x,g_2\circ g_1(x)\right) \in \EuScript{R}_G$.
\end{itemize}

Since $\EuScript{R}_G = \bigcup\limits_{g \in G}\mathsf{gr}(g)$, where $\mathsf{gr}(g) = \left\{(x,y) \in X\times X : y=g(x)\right\}$ is the graph of $g$, then
by Theorem $8.3.4$ from \cite{cohn:13} we know that under these conditions each graph $\mathsf{gr}(g)$ is a Borel subset of $X \times X$. It follows that 
$\EuScript{R}_G$ is also a Borel set.
Moreover, $\EuScript{R}_G(x) = \{g(x) : g \in G\}$ is a countable set for each $x \in X$.
\end{ex}

\begin{ex}[Gibbs relation]\label{gr}
Let $(X,T)$ be a topological dynamical system, and suppose that $T$ is an expansive action of the group $\zd$ on $X$ (see Section \ref{pressure}). 
The Gibbs (or homoclinic) relation of $(X,T)$ is defined by
\[\gr(X,T) \defeq \left\{(x,y) \in X \times X : \lim\limits_{\|i\| \to \infty} \rho(T^ix,T^iy) = 0 \right\},\]
where $\rho$ is a metric on $X$ which induces its topology. Note that the definition of $\gr(X,T)$ does not depends on the choice of the metric $\rho$.
If $X$ and $T$ are clear from the context, we will simply denote $\gr(X,T)$ by $\gr$.

Let us show that $\gr$ is a countable Borel equivalence relation on $X$. First, let us verify that
$\gr$ is an equivalence relation.

\begin{itemize}
\item[(a)] For each $x \in X$, we have $\lim\limits_{\|i\| \to \infty} \rho(T^ix,T^ix) = 0$, i.e., the pair $(x,x)$ belongs to $\gr$.  

\item[(b)] Given two elements $x$ and $y$ in $X$, if the pair $(x,y)$ belongs to $\gr$, it follows that
$\lim\limits_{\|i\| \to \infty} \rho(T^iy,T^ix) = \lim\limits_{\|i\| \to \infty} \rho(T^ix,T^iy) = 0 $. Therefore, we have $(y,x) \in \gr$. 

\item[(c)]For every points $x,y$, and $z$ in $X$, if each pair $(x,y)$ and $(y,z)$ belongs to $\gr$, then for any $\epsilon > 0$ 
there is a positive integer $n_{0}$ such that
\[\rho(T^ix,T^iy)< \frac{\epsilon}{2}\]
and
\[\rho(T^iy,T^iz)< \frac{\epsilon}{2}\]
holds whenever $i$ satisfies $\|i\| \geq n_{0}$.
Therefore, it follows that for all $i \in \zd$,
\[\|i\| \geq n_0 \;\,\text{implies that}\;\, \rho(T^ix,T^iz)< \epsilon,\]
i.e., $(x,z) \in \gr$.
\end{itemize}

Now, let us show that $\gr$ is a Borel subset of $X \times X$. Since each map $(x,y) \mapsto \rho(T^{i}x,T^{i}y)$  defined on $X \times X$ is continuous, 
then
\[\left\{(x,y)\in X \times X : \rho(T^{i}x,T^{i}y)<\frac{1}{n}\right\}\]
is an open subset of $X \times X$ for each positive integer $n$. Therefore, the result follows from the fact that
\[\gr = \bigcap\limits_{n \in \mathbb{N}}\bigcup\limits_{N \in \mathbb{N}}\bigcap\limits_{\substack{i \in \zd \\ \|i\| \geq N}}\left\{(x,y)\in X \times X : \rho(T^{i}x,T^{i}y)<\frac{1}{n}\right\}.\]

The following result shows that $\gr$ is a countable equivalence relation.

\begin{lemma}
Each equivalence class $\gr(x)$ is a countable set.	
\end{lemma}

\begin{proof}
Since $T$ is an expansive action, let $\epsilon>0$ be a positive number  
such that for every pair of distinct points $x$ and $y$ in $X$, we have
$\rho(T^{i}x,T^{i}y) \geq \epsilon$ for some $i \in \zd$. It is easy to prove that for each $x \in X$, we have
\[\gr(x) \subseteq \bigcup\limits_{n \in \mathbb{N}} E_n(x),\]
where $E_n(x) = \left\{y \in X: \rho(T^ix,T^iy) < \frac{\epsilon}{2} \;\text{holds whenever}\; i\;\text{satisfies}\; \|i\| \geq n \right\}$.
If we prove that $E_n(x)$ is a finite set for each $n$, then the result follows.

Given $n \in \mathbb{N}$, let us consider the metric $\rho_n$ on $X$ defined
by $\rho_n(x,y) = \max\limits_{\|i\|<n}\rho(T^ix,T^iy)$. Since $\rho$ and $\rho_{n}$ are two equivalent metrics on $X$,
it follows that $(X,\rho_n)$ is a compact metric space. Furthermore,
for any two distinct points $y_1$ and $y_2$ in $E_n(x)$, necessarily we have   
$\rho(T^{i}y_1,T^{i}y_2) \geq \epsilon$ for some $i \in \zd$ with $\|i\|<n$, thus $\rho_n(y_1,y_2) \geq \epsilon$. It means that $E_n(x)$ is an $\epsilon$-separated subset  
of the compact metric space $(X,\rho_n)$, therefore,
we conclude that $E_n(x)$ is finite.
\end{proof}
\end{ex}

\begin{ex}[Gibbs relation for subshifts]\label{grshift}
Let $X$ be a nonempty subshift of $\ds$ and let $T$ be the shift action on $X$ (see Example \ref{subdysy}).
According to Example \ref{subexp}, the topological dynamical system $(X,T)$ is expansive.
Let us show that the Gibbs relation of $(X,T)$ is given by

\begin{equation}\label{greq}
\gr = \big\{(x,y) \in X \times X : x_{\Lambda^c} = y_{\Lambda^c}\;\text{for some}\;\Lambda \subseteq \zd\; \text{finite}\big\}.
\end{equation}

For each $(x,y) \in \gr$ there is a positive integer $n$ such that $\rho(T^ix,T^iy) \leq \frac{1}{2}$ holds whenever $i$ satisfies $\|i\| \geq n$, in other words, we have $x_i = y_i$ for every
$i \in \Lambda_{n}^{c}$.
On the other hand, let $(x,y)$ be a pair in $X\times X$ such that $x_{\Lambda^c} = y_{\Lambda^c}$ for some $\Lambda \subseteq \zd$ finite. Without loss of generality, it can be supposed
that $\Lambda = \Lambda_{m}$ for some $m \in \mathbb{N}$. Therefore, given a positive integer $n$, it follows that  
$(T^ix)_{\Lambda_{n}} = (T^iy)_{\Lambda_{n}}$ holds whenever $i$ satisfies $\|i\| \geq m+n$.
\end{ex}

\section{Radon-Nikodym derivatives}\label{rn}

Until the end of this chapter, we will denote by $X$ a Polish space and by $\EuScript{R}$ a countable Borel equivalence relation on $X$. We will also let $\mathcal{C}$ denote 
the restriction of the Borel $\sigma$-algebra of $X \times X$ to $\EuScript{R}$. Recall that
$\mathcal{C} = \{B\cap \EuScript{R} : B \;\text{is a Borel set of}\; X \times X\} = \{B \subseteq \EuScript{R} : B \;\text{is a Borel set of}\; X \times X\}$.

Let us define the functions $\pi_l, \pi_r : \EuScript{R} \rightarrow X$ by letting $\pi_l(x,y) = x$ and $\pi_r(x,y) = y$.
The maps $\pi_l$ and $\pi_r$ defined above are called the left projection and the right projection of $\EuScript{R}$. It is also useful to consider the flip map $\theta: \EuScript{R} \rightarrow \EuScript{R}$ defined by
$\theta(x,y) = (y,x)$. 
\begin{remark}
We claim that $\theta$ is an isomorphism, thus it sends
sets in $\mathcal{C}$ to sets in $\mathcal{C}$; and both projections $\pi_l$ and $\pi_r$ send sets in $\mathcal{C}$ to Borel sets of $X$.
The proof of the first statement is straightforward.
The second one follows by using Theorem $4.12.3$ from \cite{sri:98} and the fact that $\pi_r = \pi_l \circ \theta$.	
\end{remark}

In the following, we present the necessary mathematical tools that will allow us to introduce the notion of a Radon-Nikodym derivative of a $\sigma$-finite Borel measure on $X$ with respect to
$\EuScript{R}$. Later, we will use this notion to give rise to the concept of a conformal measure. 
Now let us present an important auxiliary result due to Feldman and Moore \cite{feldman:77}. Note that
this result is closely related to Example \ref{group}.

\begin{teo}[Feldman and Moore]\label{fmteo}
Let $\EuScript{R}$ be a countable Borel equivalence relation on a Polish space $X$. Then there exists a countable group
$G \subseteq \mathsf{Aut(}X\mathsf{)}$ such that $\EuScript{R} = \EuScript{R}_G$.
\end{teo}
\begin{proof}
For a modern proof, see Theorem 5.8.13 from \cite{sri:98}.
\end{proof}

Before we follow to the next definition, let us show that for every Borel set $A$ of $X$, its $\EuScript{R}$-saturation $\EuScript{R}(A)$ is also a Borel set of $X$. Indeed,
according to Theorem \ref{fmteo}, there is a countable group $G \subseteq \mathsf{Aut(}X\mathsf{)}$ such that $\EuScript{R} = \EuScript{R}_G$, then
$\EuScript{R}(A) = \bigcup\limits_{\varphi \in G} \varphi^{-1}(A)$.

\begin{mydef}
Let $\mu$ be a $\sigma$-finite Borel measure on $X$. We say that $\mu$ is quasi-invariant under $\EuScript{R}$ (or $\EuScript{R}$ is nonsingular with respect to $\mu$) if the condition
\[\mu(A)=0 \;\,\text{implies}\;\, \mu\left(\EuScript{R}(A)\right) = 0\] 
is satisfied for every Borel set $A$ of X.
\end{mydef}

In the remainder of this section, we assume that $\mu$ is a $\sigma$-finite Borel measure on $X$ quasi-invariant under $\EuScript{R}$. 

Let us prove the following preliminary result.

\begin{prop}\label{xala}
For all $C \in \mathcal{C}$, we have $\mu\left(\pi_l(C)\right) = 0$ if and only if $\mu\left(\pi_r(C)\right) = 0$.
\end{prop} 
\begin{proof}
First, let us show that given an arbitrary subset $A$ of $X$, its $\EuScript{R}$-saturation coincides with the sets $\pi_r\left(\pi_l^{-1}(A)\right)$ and $ \pi_l\left(\pi_r^{-1}(A)\right)$ .
Indeed, for all $y$, 	
\begin{eqnarray*}
y \in \EuScript{R}(A) 
&\iff& y \in \EuScript{R}(x) \;\,\text{for some}\; x \in A \\
&\iff& (x,y) \in \EuScript{R} \;\,\text{for some}\; x \in A \\
&\iff& y=\pi_r(z) \;\,\text{for some}\;z \in \pi_{l}^{-1}(A) \\
&\iff& y \in \pi_r\left(\pi_l^{-1}(A)\right).
\end{eqnarray*}
Then, it follows that $ \EuScript{R}(A) = \pi_r\left(\pi_l^{-1}(A)\right)$. Furthermore, using the identity $\pi_r = \pi_l\, \circ\, \theta$, we have
$\pi_r\left(\pi_l^{-1}(A)\right) = \pi_l\left(\theta\left(\pi_l^{-1}(A)\right)\right) = \pi_l\left(\theta^{-1}\left(\pi_l^{-1}(A)\right)\right) = \pi_l\left(\pi_r^{-1}(A)\right)$.

Let $C$ be a set of $\mathcal{C}$ satisfiying $\mu\left(\pi_l(C)\right) = 0$. By hypothesis, the measure $\mu$ is quasi-invariant
under $\EuScript{R}$, then $\mu\left(\EuScript{R}(\pi_l(C))\right) =0$. Since $C \subseteq \pi_l^{-1}\left(\pi_l(C)\right)$, it follows that
$\pi_r(C) \subseteq \pi_r\left(\pi_l^{-1}(\pi_l(C))\right) = \EuScript{R}(\pi_l(C))$, thus $\mu\left(\pi_r(C)\right) = 0$.
Analogously, one can easily prove the opposite implication.
\end{proof}

The next theorem will provide us two measures on the measurable space $(\EuScript{R},\mathcal{C})$ which will allow us to define the Radon-Nikodym derivative of
$\mu$ with respect to $\EuScript{R}$.

\begin{teo}
The following properties hold.
\begin{itemize}
\item[(a)] For each $C \in \mathcal{C}$, the map $x \mapsto \left|\pi_{l}^{-1}(\{x\})\cap C\right|$ defined on $X$ is Borel measurable, and the formula 
\begin{equation}
\nu_l(C) = \int_X{\left|\pi_{l}^{-1}(\{x\}) \cap C\right|\, d\mu (x)} 
\end{equation}
defines a $\sigma$-finite measure on $\mathcal{C}$. This measure will be referred to as the left counting measure of $\mu$.

\item[(b)] The null sets of $\nu_l$ are exactly the elements of $\left\{C \in \mathcal{C}: \mu\left(\pi_l(C)\right) = 0\right\}$.

\item[(c)] The right counting measure of $\mu$, defined by $\nu_r = \theta_{\ast}\nu_l$, is a $\sigma$-finite measure on $\mathcal{C}$. Moreover,
\begin{equation}
\nu_r(C) = \int_{X}{\left|\pi_{r}^{-1}(\{x\}) \cap C\right|\, d\mu (x)} 
\end{equation}
for every $C \in \mathcal{C}$.

\item[(d)] We have $\nu_l \ll \nu_r$ and $\nu_r \ll \nu_l$.
\end{itemize}
\end{teo}
\begin{proof}
(a) According to Theorem 5.8.11 from \cite{sri:98} (or Theorem 18.10 from \cite{kechris:95}), we can write $\EuScript{R}$ as a countable union of Borel graphs. 
Therefore, there exists a partition $(C_n)_{n \in \mathbb{N}}$ of $\EuScript{R}$ into Borel sets such that each $\pi_l\restriction_{C_n}$ is one-to-one.

For each positive integer $n$, let us define $\nu^{n} : \mathcal{C} \rightarrow [0,+\infty]$ by letting
\[\nu^{n}(C) = \mu\left(\pi_l(C_n \cap C)\right) \quad\text{for all}\;C \in \mathcal{C}.\]
We claim that $\nu^{n}$ is a $\sigma$-finite measure on $\mathcal{C}$. Indeed, the countable additivity of $\nu^{n}$ follows from the fact that $\pi_l$ is one-to-one on $C_n$, moreover,
the condition $\nu^{n}(\emptyset) = 0$ is easily verified.
The assumption of $\sigma$-finiteness of $\mu$ implies that we can write $X$ as $\bigcup\limits_{m \in \mathbb{N}}X_m$, where each $X_{m}$ is a Borel set of $X$ such that $\mu(X_m)<+\infty$. 
Then, if we define
$R_{m} = \pi_l^{-1}(X_m)$ for each $m$, it follows that $\EuScript{R} = \bigcup\limits_{m \in \mathbb{N}}R_{m}$, where each $R_{m}$ belongs to $\mathcal{C}$ and satisfies
\[\nu^{n}(R_{m}) = \mu\left(\pi_l\left(C_n \cap R_{m}\right)\right) \leq \mu\left(\pi_l\left(R_{m}\right)\right) \leq \mu(X_m) < + \infty.\]

\begin{lemma}{\label{vudvndsu}}
Let $C \in \mathcal{C}$ such that $\pi_{l}\restriction_{C}$ is one-to-one. Then, the expression
\[\left|\pi_{l}^{-1}(\{x\})\cap C\right| = \rchi_{\pi_l(C)}(x)\]
holds for every $x \in X$.
\end{lemma}
\begin{proof}[Proof of Lemma \ref{vudvndsu}]\renewcommand{\qedsymbol}{\QEDB}
Since $\pi_{l}\restriction_{C}$ is one-to-one, then, for each element $x$ of $X$ the number $\left|\pi_{l}^{-1}(\{x\})\cap C\right|$ is equal to either $0$ or $1$ . Therefore, for all $x \in X$, we have
\begin{align*}
x \in \pi_l(C) 
&\iff \text{exists}\; z \in C \;\text{such that}\; \pi_l(z) = x \\ 
&\iff \text{exists}\; z \in C \;\text{such that}\; z \in \pi_{l}^{-1}(\{x\}) \\ 
&\iff \pi_{l}^{-1}(\{x\})\cap C \neq \emptyset \\
&\iff \big|\pi_{l}^{-1}(\{x\})\cap C\big| = 1.
\end{align*}
\end{proof}

Now, let us prove that each map $x \mapsto \left|\pi_l^{-1}(\{x\}) \cap C\right|$ is a Borel function on $X$. By using Lemma \ref{vudvndsu}, the identities
\begin{equation}\label{aeeee}
\left|\pi_l^{-1}(\{x\}) \cap C\right| = \left|\,\bigcup\limits_{n \in \mathbb{N}}{\pi_{l}^{-1}(\{x\}) \cap (C_{n}\cap C)}\,\right|  
= \sum\limits_{n = 1}^{\infty}\left|\pi_{l}^{-1}(\{x\}) \cap (C_{n} \cap C)\right| 
= \sum\limits_{n = 1}^{\infty}\rchi_{\pi_l(C_{n} \cap C)}(x)
\end{equation}
hold for every $x \in X$, therefore, the measurability of $x \mapsto \left|\pi_l^{-1}(\{x\}) \cap C\right|$ follows.

Then, we are finally allowed to define $\nu_{l}$ by letting
\[\nu_{l}(C) = \int_X{\left|\pi_{l}^{-1}(\{x\}) \cap C\right|\,d\mu (x)}\]
for each $C \in \mathcal{C}$. The assertion that $\nu_l$ is a measure follows from the expression $\nu_l(C) = \sum\limits_{n = 1}^{\infty} \nu^{n}(C)$,
obtained by integrating equation (\ref{aeeee}) with respect to $\mu$. Since each $\nu^{n}$ is a $\sigma$-finite measure on $\mathcal{C}$, then $\EuScript{R}$ can be expressed as a union of a sequence $\left(R_{m}^{n}\right)_{m \in \mathbb{N}}$ of sets that
belong to $\mathcal{C}$ and  have finite measure under $\nu^{n}$. If we define $R_{m,n} = R_{m}^{n} \cap C_{n}$ for each $m$ and $n$, it follows that
$\EuScript{R} = \bigcup\limits_{m,n \in \mathbb{N}}R_{m,n}$ and each $R_{m,n}$ satisfies $\nu_{l}(R_{m,n}) = \nu^{n}(R_{m,n}) = \nu^{n}(R_{m}^{n}) < +\infty$.
 
(b) For all $C  \in \mathcal{C}$, we have $\nu_l(C) = 0$ if and only if $\left|\pi_{l}^{-1}(\{x\}) \cap C\right| = 0$ for $\mu$-almost every $x \in X$.
Then, the result follows from the identity $\pi_{l}(C) = \left\{x \in X: \left|\pi_{l}^{-1}(\{x\}) \cap C\right|\neq 0\right\}$.

(c) One can easily show that the $\sigma$-finiteness of $\nu_{l}$ implies that $\nu_{r}$ is also a $\sigma$-finite measure.
Moreover, given a set $C \in \mathcal{C}$, we have
\begin{eqnarray*}
\nu_r(C) &=& \int_X{\left|\pi_{l}^{-1}(\{x\}) \cap \theta^{-1}(C)\right|\, d\mu (x)} \\
&=& \int_X{\left|\theta^{-1}\big(\pi_{r}^{-1}(\{x\}) \cap C\big)\right|\, d\mu (x)} \\
&=& \int_X{\left|\pi_{r}^{-1}(\{x\}) \cap C\right|\, d\mu (x)}.
\end{eqnarray*}

(d) Since for each element $C$ of $\mathcal{C}$ we have
\[\nu_{r}(C) = \nu_{l}\left(\theta^{-1}(C)\right)\] 
and
\[\mu\left(\pi_{r}(C)\right) = \mu\left(\pi_{l}\left(\theta(C)\right)\right) = \mu\left(\pi_{l}\left(\theta^{-1}(C)\right)\right),\]
then it follows from item (b) that the null sets of $\nu_{r}$ are exactely the elements of $\{C \in \mathcal{C} : \mu\left(\pi_{r}(C)\right)=0\}$. Therefore, the result follows by using 
Proposition \ref{xala} and item (b).
 \end{proof}

From now on, instead of we say that a property of points of $\EuScript{R}$ holds $\nu_{l}$-a.e. (equivalently, $\nu_{r}$-a.e.), we will simply say that this property holds almost everywhere (or a.e.).

The theorem above states that the left and right counting measures of $\mu$, respectively denoted by $\nu_{l}$ and $\nu_{r}$, are both $\sigma$-finite measures on $\mathcal{C}$ absolutely continuous with respect to each other.
Then, the Radon-Nikodym derivatives $\frac{d\nu_{l}}{d\nu_{r}}$ and $\frac{d\nu_{r}}{d\nu_{l}}$ satisfy the identity $\frac{d\nu_{l}}{d\nu_{r}}\cdot \frac{d\nu_{r}}{d\nu_{l}} = 1$
a.e., in particular, $\frac{d\nu_{l}}{d\nu_{r}}$ and $\frac{d\nu_{r}}{d\nu_{l}}$ are both positive functions a.e. 

\begin{mydef}
Let $\mu$ be a $\sigma$-finite Borel measure on $X$ quasi-invariant under $\EuScript{R}$. Then, the Radon-Nikodym derivative of $\mu$ with respect to $\EuScript{R}$ is the measurable function $D_{\mu,\EuScript{R}}$ on $\EuScript{R}$ defined by 
\begin{equation}\label{defrn}
D_{\mu,\EuScript{R}} = \frac{d\nu_l}{d\nu_r}.
\end{equation}
The function $D_{\mu,\EuScript{R}}$ is unique up to almost everywhere equality.
\end{mydef}

In the following, given a Borel subset $A$ of $X$ we will denote by $\mu_A$ the measure $\mu$ restricted to the $\sigma$-algebra of Borel subsets of $A$.

\begin{prop}\label{boraut}
Let $A, B \subseteq X$ be Borel sets and let $\varphi: A \rightarrow B$ be an isomorphism with $\mathsf{gr}(\varphi) \subseteq \EuScript{R}$.
Then, $\varphi_{\ast}\mu_A$ is absolutely continuous with respect
to $\mu_B$ and the equation
\begin{equation}\label{xena}
\frac{d \varphi_{\ast}\mu_A}{d\mu_B}(y) = D_{\mu,\EuScript{R}}\left(\varphi^{-1}(y),y\right)
\end{equation}
holds for $\mu$-almost every $y \in B$.
\end{prop}

\begin{proof}
This proof is divided into 3 steps.

Step 1. Let us prove that given a Borel subset $B'$ of $B$ we have 
\begin{equation}\label{rellegal}
\varphi_{\ast}\mu_A (B') = \int_{C_{B'}}D_{\mu,\EuScript{R}} \,d\nu_r,
\end{equation}
where $C_{B'} = \mathsf{gr}(\varphi)\cap\pi_{r}^{-1}(B')$. Indeed, since $\varphi$ is a measurable function and $C_{B'} \subseteq \mathsf{gr}(\varphi)$, it follows that $C_{B'}$ belongs to $\mathcal{C}$ and
$\pi_l\restriction_{C_{B'}}$ is one-to-one. Then, by means of the identity $\varphi^{-1}(B') = \pi_{l}\left(C_{B'}\right)$ and Lemma \ref{vudvndsu}, we find
\begin{eqnarray*}	
\varphi_{\ast}\mu_A (B') &=& \mu\left(\varphi^{-1}(B')\right) = \mu\left(\pi_{l}(C_{B'})\right) \\
&=& \nu_{l}(C_{B'}) = \int_{C_{B'}}\frac{d\nu_{l}}{d\nu_{r}} \,d\nu_r.
\end{eqnarray*}

Step 2. Now, we claim that for every $C\in \mathcal{C}$ such that $C \subseteq \mathsf{gr}(\varphi)$, we have 
\begin{equation}\label{relnut}
\nu_r(C) = T_{\ast}\mu_B (C),
\end{equation}
where $T: B \rightarrow \EuScript{R}$ is the function given by $T(x) = (\varphi^{-1}(x),x)$ for all $x \in B$ (we left to the reader to check that $T$ is measurable).
In fact, since $\theta(C) \subseteq \theta(\mathsf{gr}(\varphi)) = \mathsf{gr}(\varphi^{-1})$, it follows that
$\pi_{l}\restriction_{\theta(C)}$ is one-to-one. Using Lemma \ref{vudvndsu}, we find
\begin{equation*}
\nu_r(C) = \nu_l\left(\theta(C)\right) = \mu\left(\pi_l(\theta(C))\right) = \mu\left(\pi_r(C)\right).
\end{equation*} 
It is easy to check that $\pi_r(C) = T^{-1}(C)$, thus (\ref{relnut}) follows.

Step 3. Let us show that for any measurable function $f : \EuScript{R} \rightarrow [-\infty, + \infty]$ such that the integral $\int_{\EuScript{R}}{f \,d\nu_r}$ exists (not necessarily a finite number),
the equation
\begin{equation}\label{ximbalaie}
\int_{C_{B'}}f\,d\nu_r = \int_{B'}{f\circ T \, d\mu_B}
\end{equation}
holds for every Borel subset $B'$ of $B$. 

First, let us consider a Borel subset $B'$ of $B$. In the case where $f = \rchi_{C}$ for some $C \in \mathcal{C}$, using equation (\ref{relnut}) and the identity $T^{-1}(C_{B'}) = B'$, we have
\begin{eqnarray*}
\int_{C_{B'}}\rchi_{C}\,d\nu_r &=& \nu_r(C\cap C_{B'}) = \mu_B\left(T^{-1}(C\cap C_{B'})\right) \\
\end{eqnarray*}
\begin{eqnarray*}
&=& \mu_B\left(T^{-1}(C)\cap B'\right) = \int_{B'}\rchi_{T^{-1}(C)}\,d\mu_B \\
&=& \int_{B'}\rchi_{C}\circ T\,d\mu_B.  
\end{eqnarray*}
The linearity of the integral and the fact that equation (\ref{ximbalaie}) holds in the case where $f$ is a characteristic function of an element of $\mathcal{C}$, we easily prove that this equation is also satisfied in the case where 
$f$ is a measurable simple function. If $f$ is a $[0,+\infty]$-valued measurable function, let $(f_n)_{n \in \mathbb{N}}$ be an increasing
sequence  of $[0,+\infty)$-valued measurable simple functions converging pointwise to $f$ on $\EuScript{R}$, then 
the monotone convergence theorem and the previous case imply that (\ref{ximbalaie}) holds. The general case follows by applying the previous case to the positive and negative parts of $f$.  

In particular, by letting $f = D_{\mu,\EuScript{R}}$ and combining equations (\ref{rellegal}) and (\ref{ximbalaie}), we obtain
\[\varphi_{\ast}\mu_A (B') = \int_{B'}{D_{\mu,\EuScript{R}}\circ T \, d\mu_B}\]
for every Borel subset $B'$ of $B$. 
\end{proof}

\begin{cor}\label{mot}
For $\mu$-almost every $z \in X$, we have 
\begin{equation}
D_{\mu,\EuScript{R}}(x,z) = D_{\mu,\EuScript{R}}(x,y) \cdot D_{\mu,\EuScript{R}}(y,z)
\end{equation} 	
for all $x,y \in \EuScript{R}(z)$.
\end{cor}

\begin{proof}
Let $G \subseteq \mathsf{Aut}(X)$ be a countable group such that $\EuScript{R} = \EuScript{R}_{G}$.
Let us show that for any two elements $\varphi, \psi \in G$, we have
\begin{equation}\label{rnrel}
D_{\mu, \EuScript{R}}\left((\varphi\circ \psi)^{-1}(w),w\right) = D_{\mu, \EuScript{R}}\left((\varphi\circ \psi)^{-1}(w),\varphi^{-1}(w)\right)\cdot D_{\mu, \EuScript{R}}\left(\varphi^{-1}(w),w\right)
\end{equation}
for $\mu$-almost every $w\in X$.
Indeed, equation (\ref{rnrel}) follows from the fact that
\[(\varphi\circ \psi)_{\ast}\mu(B) = \int_{B} D_{\mu, \EuScript{R}}\left((\varphi\circ \psi)^{-1}(w),w\right)\,d\mu(w)\]
and
\begin{eqnarray*}
(\varphi\circ \psi)_{\ast}\mu(B) &=& {\psi}_{\ast}\mu\left(\varphi^{-1}(B)\right) \\
&=& \int_{\varphi^{-1}(B)} D_{\mu, \EuScript{R}}\left(\psi^{-1}(w),w\right)\,d\mu(w)	\\
&=& \int_{B} D_{\mu, \EuScript{R}}\left(\psi^{-1}(\varphi^{-1}(w)),\varphi^{-1}(w)\right)\,d{\varphi}_{\ast}\mu(w) \\
&=& \int_{B} D_{\mu, \EuScript{R}}\left((\varphi\circ \psi)^{-1}(w),\varphi^{-1}(w)\right)\cdot D_{\mu, \EuScript{R}}\left(\varphi^{-1}(w),w\right)\,d\mu(w) \\
\end{eqnarray*}
for every Borel subset $B$ of $X$. 

Let $N$ be a $\mu$-null set such that (\ref{rnrel}) holds on $X\backslash N$, and let $z$ be an element of $X\backslash N$.
For each pair $x, y$ of points in $\EuScript{R}(z)$ we have both $(x,y)$ and $(y,z)$ in $\EuScript{R}$, then there exist two elements $\varphi$ and $\psi$ of $G$ such
that $\varphi(y) = z$ and $\psi(x) = y$. It follows that
\[D_{\mu,\EuScript{R}}(x,z) = D_{\mu,\EuScript{R}}(x,y) \cdot D_{\mu,\EuScript{R}}(y,z).\]
\end{proof}

\begin{remark}
From Corolary \ref{mot} and from the fact that $D_{\mu,\EuScript{R}} > 0$ a.e., one can easily verify that for $\mu$-almost every $z \in X$, we have
\begin{equation}\label{rncocycle}
\log D_{\mu,\EuScript{R}}(x,z) = \log D_{\mu,\EuScript{R}}(x,y) + \log D_{\mu,\EuScript{R}}(y,z)	
\end{equation}
for all $x,y \in \EuScript{R}(z)$.
\end{remark}

\section{Conformal measures}\label{prel}

In order to define the concept of a conformal measure, first we need to present the definition of a cocycle. For further references, see \cite{schmidt:97}, \cite{aaronson:07}.

\begin{mydef}
An $\EuScript{R}$-cocycle (also called an 1-cocycle of $\EuScript{R}$) is a measurable function $\phi : \EuScript{R} \rightarrow \mathbb{R}$ such that
\begin{equation}\label{cocycle}
\phi(x,z) = \phi(x,y) + \phi(y,z){}
\end{equation}
holds for all $x,y, z \in X$ satisfying $(x,y), (y,z) \in \EuScript{R}$.
\end{mydef}

\begin{remark}
It is easy to check that $\phi(x,x) = 0 $ for every $x \in X$. We also have $\phi(x,y) = -\phi(y,x)$ for each pair $(x,y) \in \EuScript{R}$.	
\end{remark}

Finally, equations (\ref{rncocycle}) and (\ref{cocycle}) motivate the following definition.

\begin{mydef}
Let $\phi : \EuScript{R} \rightarrow \mathbb{R}$ be an $\EuScript{R}$-cocycle. A Borel probability measure $\mu$ on $X$ is
called $(\phi, \EuScript{R})$-conformal if $\mu$ is quasi-invariant under $\EuScript{R}$ and the formula
\begin{equation}\label{xeena} 
D_{\mu,\EuScript{R}} = e^{-\phi} 
\end{equation}
holds almost everywhere.
\end{mydef}

The following proposition characterizes a conformal measure in terms of a group $G$ (provided by Theorem \ref{fmteo}) which generates the relation $\EuScript{R}$. 
This result will be usefull in the following sections.

\begin{prop}\label{haaa}
Let $G \subseteq \mathsf{Aut(}X\mathsf{)}$ be a countable group which generates $\EuScript{R}$. Then, a Borel probability measure $\mu$ on $X$ is 
$(\phi, \EuScript{R})$-conformal if and only if for each $\varphi \in G$ the measure $\varphi_{\ast}\mu$ is absolutely continuous with respect to $\mu$ and 
the equation
\begin{equation}\label{xuxa}
\frac{d \varphi_{\ast}\mu}{d \mu}(x) = e^{\phi(x,\varphi^{-1}(x))} 	
\end{equation}
holds for $\mu$-almost every $x \in X$. 
\end{prop}

\begin{proof}
First, suppose that $\mu$ is $(\phi,\EuScript{R})$-conformal. Let us consider an element $\varphi$ of $G$. Proposition \ref{boraut} implies that $\varphi_{\ast}\mu$ is
absolutely continuous with respect to $\mu$ and there is a $\mu$-null subset $N$ of $X$ such that
\[\frac{d \varphi_{\ast}\mu}{d\mu}(x) = D_{\mu,\EuScript{R}}\left(\varphi^{-1}(x),x\right)\]
holds for all $x \in X \backslash N$.
We also let $C_0$ be an element of $\mathcal{C}$ such that $\mu\left(\pi_r(C_0)\right) = 0$ and (\ref{xeena}) holds at each point of $\EuScript{R}\backslash C_0$.
If we define $N_0 = \pi_r(C_0) \cup N$, then $N_0$ is also a $\mu$-null subset of $X$ and for each point $x$ in $X\backslash N_0$ we have 
$(\varphi^{-1}(x),x) \in \EuScript{R}\backslash C_0$ and $x \in X \backslash N$. Therefore, we have
\[\frac{d \varphi_{\ast}\mu}{d\mu}(x) = e^{-\phi(\varphi^{-1}(x),x)} = e^{\phi(x,\varphi^{-1}(x))}\]
for every $x \in X\backslash N_0$.

On the other hand, let us show that $\mu$ is quasi-invariant under $\EuScript{R}$ and satisfies (\ref{xeena}). Given a  Borel subset $A$ of $X$ such that 
$\mu(A) = 0 $, we have $\varphi_{\ast}\mu(A) = 0$ for each $\varphi \in G$. It follows that $\mu\left(\EuScript{R}(A)\right) = \mu\left(\bigcup\limits_{\varphi \in G}\varphi^{-1}(A)\right) = 0$,
thus $\mu$ is quasi-invariant under $\EuScript{R}$.
Our hypotheses and Proposition \ref{boraut} imply that for each $\varphi \in G$ we have
\begin{equation}\label{psico}
D_{\mu,\EuScript{R}}\left(\varphi^{-1}(y),y\right) = e^{-\phi(\varphi^{-1}(y),y)}
\end{equation}
for $\mu$-almost every $y \in X$. Let $N$ be a $\mu$-null set such that (\ref{psico}) is satisfied for every $\varphi \in G$ at each point $y$ of $X\backslash N$.
If we define $C_0 = \pi_{r}^{-1}(N)$, then it follows that $C_0$ is a $\nu_r$-null set and for all $(x,y) \in \EuScript{R}\backslash C_0$ we have
\[D_{\mu,\EuScript{R}}(x,y) = e^{-\phi(x,y)},\]
since there exists an element $\varphi$ of $G$ such that $x= \varphi^{-1}(y)$.
\end{proof}


%% file: cap-conclusoes.tex
\chapter{Gibbs measures for subshifts}
\label{cap:gibbs}

In this chapter we are finally able to begin the study of Gibbs measures on subshifts, introduced in \cite{meyerovitch:13}, \cite{aaronson:07}, \cite{schmidt:97}, \cite{schpet}.

We restrict ourselves to the study of Gibbs measures for a specific class of functions defined on a subshift, the so-called
functions with $d$-summable variation (\cite{meyerovitch:13}) or regular local energy functions (\cite{keller:98}, \cite{muir}).
Then, we devote the first section to show a few properties of these potentials. Later, based on our 
knowledge on conformal measures, we introduce two different definitions of Gibbs measures provided by \cite{meyerovitch:13}, and
show that both definitions coincide for SFTs.  
Differently from the usual approach, these definitions does not involve conditional expectations, due to this fact,
we dedicate the last section to connect them with other definitions frequently presented in the literature (e.g. \cite{cap:76},\cite{georgii:11},\cite{ny:08},\cite{sarig:09}).  

\section{Potentials}\label{pot}

Let us begin by introducing some notation. From now on, we will always let $X$ denote a nonempty subshift of $\ds$ and let $T$ denote the shift action of $\zd$ on $X$.
If $\Delta$ is an arbitrary subset of $\zd$, then we define the set of all $\Delta$-configurations permited on $X$ by $X_\Delta \defeq \{x_\Delta : x \in X\}$.
And, given a finite subset $\Lambda$ of $\zd$ and a pattern $\omega \in \mathcal{A}^{\Lambda}$, we define the cylinder with configuration $\omega$ as the subset of $X$ given by $[\omega] \defeq \{x \in X : x_\Lambda = \omega\}$. It is easy to check that
every cylinder is a clopen (i.e., open and closed) subset of $X$. 

The set of all functions with $d$-summable variation on $X$ is defined as follows. Given an arbitrary real-valued function $f$ defined on $X$ and a positive integer $n$, we define the $n$-th variation of $f$ as the nonnegative extended real number given by   
\begin{equation}
\delta_n (f) \defeq \sup\left\{\left|f(x)-f(y)\right| : x, y \in X \;\text{satisfy}\; x_{\Lambda_{n}} = y_{\Lambda_{n}}\right\}.
\end{equation}
Then, let us define the set of all functions with $d$-summable variation on $X$  by

\begin{equation}
SV_{d}(X) \defeq \left\{f \in \mathbb{R}^{X} : \sum\limits_{n =1}^{\infty} n^{d-1}\delta_{n}(f) < +\infty\right\}.	
\end{equation}

\begin{remark}
For each $f \in SV_{d}(X)$ we have $\lim\limits_{n \to \infty} \delta_{n}(f) = 0$. It follows that every 
function with $d$-summable variation is uniformly continuous. 
\end{remark}

\begin{ex}
Let $X \subseteq \ds$ be a subshift and let $\rho$ be the metric on $\ds$ defined by $(\ref{metric})$. A function $f : X \rightarrow \mathbb{R}$ is said to be H\" older continuous if there are positive numbers $L$ and $h$ such that
\[|f(x) - f(y)| \leq L \cdot \rho(x,y)^{h}\]  	
holds for each $x$ and $y$ in $X$. It is straightforward to show that every H\" older continuous function on $X$ belongs to $SV_{d}(X)$.
\end{ex}

\begin{ex}
Let us consider the full shift $X = \{-1,+1\}^{\mathbb{Z}}$ and the function $f:X \rightarrow \mathbb{R}$ given by
\[f(x) = \sum\limits_{i=1}^{\infty}\frac{x_{0}x_{i}}{i^{2+\epsilon}} + \sum\limits_{i=1}^{\infty}\frac{x_{0}x_{-i}}{i^{2+\epsilon}},\]
where $\epsilon$ is a positive real number. It is easy to check that $\delta_{n}(f) = \sum\limits_{i = n}^{\infty}\frac{4}{i^{2+\epsilon}}$ for each $n$. 
Therefore, it follows from
\[\sum\limits_{n = 1}^{\infty} \delta_{n}(f) = \sum\limits_{n = 1}^{\infty}\sum\limits_{i = n}^{\infty} \frac{4}{i^{2+\epsilon}} = \sum\limits_{i = 1}^{\infty}\sum\limits_{n = 1}^{i} \frac{4}{i^{2+\epsilon}} = \sum\limits_{i = 1}^{\infty}\frac{4}{i^{1+\epsilon}} < +\infty\]
that $f$ has summable variation. 
\end{ex}

It is straightforward to check that $SV_{d}(X)$ is a real vector space with respect to the usual operations of addition of functions and multiplication by scalar.
And also, one can define a norm on this space by letting
\begin{equation}
\|f\|_{SV_{d}} = \|f\|_{\infty} + \sum\limits_{n =1}^{\infty} n^{d-1}\delta_{n}(f)\quad \;\text{for each}\; f \in SV_{d}(X).
\end{equation}

The following results will be used only in Section \ref{geq}, so they might be skipped at a first reading. 
\begin{prop}
The pair $\left(SV_d(X),\|\cdot\|_{SV_d}\right)$ is a Banach space.
\end{prop}

\begin{proof}
Let $(f_m)_{m \in \mathbb{N}}$ be a Cauchy sequence in $SV_d(X)$. For each $\epsilon > 0$ there is a positive integer $m_0$ such that
$\|f_m - f_{m'}\|_{SV_d} < \frac{\epsilon}{2}$ holds whenever $m \geq m_{0}$ and $m' \geq m_0$. Since $\|f_m - f_{m'}\|_{\infty} \leq \|f_m - f_{m'}\|_{SV_d}$, it follows
that $(f_m)_{m \in \mathbb{N}}$ is a Cauchy sequence with respect to $\|\cdot\|_{\infty}$, so, it converges to some continuous function $f$ with
respect to this norm. Then, for each $m \geq m_0$  and each $N \geq 1$, if we choose a positive integer $m'$ such that $m' \geq m_{0}$ and $(2N^{d}+1)\cdot\|f-f_{m'}\|_{\infty} < \frac{\epsilon}{2}$, we have  

\begin{eqnarray*}
\|f - f_{m}\|_{\infty}+\sum\limits_{n=1}^{N}n^{d-1}\delta_{n}(f - f_{m}) &\leq& \|f_{m} - f_{m'}\|_{\infty} + \sum\limits_{n=1}^{N}n^{d-1}\delta_{n}(f_m-f_{m'}) \\&&+ \|f - f_{m'}\|_{\infty} + \sum\limits_{n=1}^{N}n^{d-1}\delta_{n}(f-f_{m'})\\
&\leq&  \|f_m - f_{m'}\|_{SV_d} + (2N^{d}+1)\cdot\|f-f_{m'}\|_{\infty} \\
&<& \epsilon.
\end{eqnarray*}
Therefore, we have
\begin{equation}
\|f - f_{m}\|_{\infty}+\sum\limits_{n=1}^{\infty}n^{d-1}\delta_{n}(f - f_{m}) \leq \epsilon	
\end{equation}
whenever $m \geq m_{0}$. One can easily prove that $f \in SV_{d}(X)$, thus the result follows.
\end{proof}

In the following, we introduce a special subset of $SV_{d}(X)$ which will play an important role in the proof of Theorem \ref{mey}. Let us define the set of all local functions on $X$ by

\[\textnormal{Loc}(X) \defeq \bigcup_{\substack{\Lambda \subseteq \zd \\ \Lambda \;\text{finite}}}\big\{f \in \mathbb{R}^X : f(x) \;\text{depends only on}\; x_{\Lambda}\big\},\] 
where ``$f(x)$ depends only on $x_\Lambda$'' means that if $x$ and $y$ are two elements of $X$ satisfying 
$x_\Lambda = y_\Lambda$, then $f(x) = f(y)$. 
We claim that $\textnormal{Loc}(X) \subseteq SV_d(X)$. Indeed, let $f$ be a local function on $X$ such that $f(x)$ depends only on $x_{\Lambda}$, where $\Lambda$ is a finite subset of $\zd$. By letting $n_{0}$ be a positive integer such that
$\Lambda_{n_{0}} \supseteq \Lambda$, we have $\delta_{n}(f) = 0$ for every $n \geq n_0$. Hence
\[\sum\limits_{n=1}^{\infty}n^{d-1}\delta_{n}(f) = \sum\limits_{n=1}^{n_0 - 1}n^{d-1}\delta_{n}(f) < +\infty.\]

\begin{ex}[Continuation of Example \ref{ising1}]\label{ising2}
Let us consider again the full shift $X = \{-1,+1\}^{\zd}$ and define $\|i\|_{1} \defeq \sum\limits_{n = 1}^{d}|i_n|$ for each $i \in \zd$. In statistical mechanics, we often say that two sites $i, j \in \zd$ are nearest neighbours
if $\|i-j\|_{1} = 1$, which we denote by $i \sim j$.
Given two parameters $J, h \in \mathbb{R}$, we immediately see that the function $f^{J,h} : X \rightarrow \mathbb{R}$ defined by
\begin{equation}
f^{J,h}(x) = \frac{J}{2} \sum\limits_{j\, \sim\, \bf{0}}x_{\bf{0}}x_j  + h x_{\bf{0}}
\end{equation}
belongs to $\textnormal{Loc}(X)$. This function describes the interaction energy of the spin located at the origin of the $d$-dimensional integer lattice $\zd$, which interacts (with a coupling constant $J$) with its neighbours and with an external field $h$.
\end{ex}

\begin{prop}
The set $\textnormal{Loc}(X)$ is dense in $SV_d(X)$.
\end{prop}
\begin{proof}
Given an arbitrary function $f$ in $SV_d(X)$, let us define a sequence $(f_m)_{m \in \mathbb{N}}$ in $\textnormal{Loc}(X)$ as follows. We define $f_{m}$ by letting
\[f_m(x) = \sup\limits_{\substack{y \in X \\ y_{\Lambda_m} = x_{\Lambda_m}}}f(y)\]
for each $x$ in $X$.
Observe that the supremum above is a real number, since the set in which it is taken is nonempty and bounded
above by $\|f\|_\infty$. It is easy to check that $f_m(x)$ depends only on $x_{\Lambda_m}$.

First, let us show that 
\begin{equation}\label{ihrapaz}
\lim\limits_{m \to \infty}\|f-f_m\|_\infty = 0.
\end{equation}
Indeed, for each positive integer $m$, we have
\begin{eqnarray*}
\|f-f_m\|_\infty &=& \sup\limits_{x \in X}\left|\,f(x) - \sup\limits_{\substack{y \in X \\ y_{\Lambda_m} = x_{\Lambda_m}}}f(y)\,\right| = \sup\limits_{x \in X}\left|\,\sup\limits_{\substack{y \in X \\ y_{\Lambda_m} = x_{\Lambda_m}}}\left(f(y)-f(x)\right)\,\right| \\
&\leq& \sup\limits_{x \in X}\sup\limits_{\substack{y \in X \\ y_{\Lambda_m} = x_{\Lambda_m}}}\left|f(y)-f(x)\right| \\
&\leq& \delta_m(f).
\end{eqnarray*}
Thus, equation (\ref{ihrapaz}) follows by using the fact that $\lim\limits_{m \to \infty}\delta_{m}(f) = 0$.

Now, let us prove that
\begin{equation}\label{rapazdoceu}
\delta_n(f-f_m) = \delta_n(f) 
\end{equation}
holds for every $n \geq m \geq 1$, and
\begin{equation}\label{ihrapaz2}
\delta_n(f-f_m) \leq 2 \delta_n(f) 
\end{equation}
holds for every $m \geq n \geq 1$. The reader can easily check that equation (\ref{rapazdoceu}) follows from the fact that $f_{m}$ depends only on $\Lambda_{m}$. Now, if we suppose that $m \geq n \geq 1$, 
then inequality (\ref{ihrapaz2}) follows from
\[\delta_n(f-f_m) \leq \delta_n(f) + \delta_n(f_m)\] 
and 
\begin{eqnarray*}
\delta_n(f_m) &=& \sup\limits_{\substack{x,y \in X \\ x_{\Lambda_n} = y_{\Lambda_n}}}\left|f_m(x)-f_m(y)\right| = \sup\limits_{\substack{x,y \in X \\ x_{\Lambda_n} = y_{\Lambda_n}}}\left|\,\sup\limits_{\substack{x' \in X \\ x'_{\Lambda_m} = x_{\Lambda_m}}}\left(f(x')-f_m(y)\right)\,\right| \\
&\leq&\sup\limits_{\substack{x,y \in X \\ x_{\Lambda_n} = y_{\Lambda_n}}}\sup\limits_{\substack{x' \in X \\ x'_{\Lambda_m} = x_{\Lambda_m}}}\left|f(x')-f_m(y)\right| = \sup\limits_{\substack{x,y \in X \\ x_{\Lambda_n} = y_{\Lambda_n}}}\sup\limits_{\substack{x' \in X \\ x'_{\Lambda_m} = x_{\Lambda_m}}}\left|f_m(y)-f(x')\right| \\
&=& \sup\limits_{\substack{x,y \in X \\ x_{\Lambda_n} = y_{\Lambda_n}}}\sup\limits_{\substack{x' \in X \\ x'_{\Lambda_m} = x_{\Lambda_m}}}\left|\,\sup\limits_{\substack{y' \in X \\ y'_{\Lambda_{m}} = y_{\Lambda_{m}}}}\left(f(y')-f(x')\right)\,\right| \\
\end{eqnarray*}
\begin{eqnarray*}
&\leq& \sup\limits_{\substack{x,y \in X \\ x_{\Lambda_n} = y_{\Lambda_n}}}\sup\limits_{\substack{x' \in X \\ x'_{\Lambda_m} = x_{\Lambda_m}}}\sup\limits_{\substack{y' \in X \\ y'_{\Lambda_{m}} = y_{\Lambda_{m}}}}\left|f(y')-f(x')\right| \\
&\leq& \sup\limits_{\substack{x',y' \in X \\ x'_{\Lambda_{n}} = y'_{\Lambda_{n}}}}\big|f(x')-f(y')\big| = \delta_n(f).
\end{eqnarray*}

Finally, let us show that the sequence $(f_{m})_{m \in \mathbb{N}}$ converges to $f$. Since $f$ belongs to $SV_{d}(X)$, it follows that for every $\epsilon>0$ there is a positive integer $n_0$ such that for all $N > n_{0}$ we have $\sum\limits_{n=n_0+1}^{N}n^{d-1}\delta_{n}(f)<\frac{\epsilon}{4}$. 
Using equation (\ref{ihrapaz}), we can choose a positive integer $m_{0}$ such that $m_{0} > n_{0}$ and $(2n_0^d+1)\cdot\|f-f_m\|_{\infty}<\frac{\epsilon}{2}$ holds whenever $m\geq m_{0}$. Then, for all integers $m$ and $N$ satisfying
$N > m \geq m_0$, by using equations (\ref{rapazdoceu}) and (\ref{ihrapaz2}), we find

\begin{eqnarray*}
\|f-f_m\|_{\infty} + \sum\limits_{n=1}^{N}n^{d-1}\delta_{n}(f-f_m) &=& \|f-f_m\|_{\infty} + \sum\limits_{n = 1}^{n_0}n^{d-1}\delta_{n}(f-f_{m}) \\ 
&& + \sum\limits_{n=n_0+1}^{m}n^{d-1}\underbrace{\delta_{n}(f-f_m)}_{\leq\, 2\delta_{n}(f)} + \sum\limits_{n=m+1}^{N}n^{d-1}\underbrace{\delta_{n}(f-f_m)}_{= \,\delta_{n}(f) \,\leq\, 2\delta_{n}(f)} \\
&\leq& (2n_{0}^d+1)\cdot \|f-f_m\|_{\infty} + 2 \sum\limits_{n = n_0 + 1}^{N} n^{d-1}\delta_{n}(f) \\
&<& \epsilon.
\end{eqnarray*}
Therefore, we finally conclude that
\[\|f-f_m\|_{SV_d} = \|f-f_m\|_{\infty} + \sum\limits_{n=1}^{\infty}n^{d-1}\delta_{n}(f-f_m) \leq \epsilon\]
holds whenever $m$ is a positive integer satisfying $m \geq m_0$.
\end{proof}

\begin{cor}
The space $SV_d(X)$ is separable.
\end{cor}
\begin{proof}
Let us define the set
\[\textnormal{Loc}^\mathbb{Q}(X) \defeq \bigcup_{\substack{\Lambda \subseteq \zd \\ \Lambda\;\text{finite}}}\left\{f \in \mathbb{Q}^X : f(x) \;\text{depends only on}\; x_{\Lambda}\right\}.\] 
We claim that $\textnormal{Loc}^\mathbb{Q}(X)$ is a countable subset of $\textnormal{Loc}(X)$. In fact, first note that the set $\{\Lambda \subseteq \zd : \Lambda \;\text{is a finite set}\}$ is countable. And, for each finite subset $\Lambda$ of $\zd$, the set
$\left\{f \in \mathbb{Q}^X : f(x) \;\text{depends only on}\; x_{\Lambda}\right\}$ is also countable, since
there is a natural one-to-one map from this set onto $\mathbb{Q}^{X_\Lambda}$.

Let $f$ be an arbitrary element of $SV_d(X)$. For each $\epsilon >0$ there is a function $g \in \textnormal{Loc}(X	)$ such that $\|f-g\|_{SV_d}<\frac{\epsilon}{2}$.
Since $g$ is a local function, without loss of generality we can suppose that $g(x)$ depends only on $x_{\Lambda_N}$ for some positive integer $N$. It follows that $g$ can be written in the form
\[g = \sum\limits_{\omega \in X_{\Lambda_N}}y(\omega) \cdot\rchi_{[\omega]},\]
where each $y(\omega)$ belongs to $\mathbb{R}$. Thus, for each $\omega$ in $X_{\Lambda_N}$, let us choose a rational number $\widetilde{y}(\omega)$ such that $(2N^{d}+1)\cdot|y(\omega)-\widetilde{y}(\omega)|<\frac{\epsilon}{2}$, and let $\widetilde{g}$ be a function on $X$ defined by
\[\widetilde{g} = \sum\limits_{\omega \in X_{\Lambda_N}}\widetilde{y}(\omega) \cdot\rchi_{[\omega]}.\]
Note that $\widetilde{g}$ belongs to $\textnormal{Loc}^\mathbb{Q}\textnormal{(}X\textnormal{)}$ and 
\begin{eqnarray*}
\|g-\widetilde{g}\|_{SV_d} &=& \|g - \widetilde{g}\|_{\infty} + \sum\limits_{n=1}^{\infty}n^{d-1}\delta_{n}(g-\widetilde{g}) \\
&=& \|g - \widetilde{g}\|_{\infty} + \sum\limits_{n=1}^{N}n^{d-1}\delta_{n}(g-\widetilde{g}) \\
& \leq & (2 N^{d}+ 1) \cdot\|g-\widetilde{g}\|_{\infty} < \frac{\epsilon}{2}.
\end{eqnarray*}	
Therefore, we conclude that $\|f-\widetilde{g}\|_{SV_d} \leq \|f-g\|_{SV_d} + \|g-\widetilde{g}\|_{SV_d} < \epsilon$.
\end{proof}

\section{Gibbs measures}\label{lalala}

In this section, we finally begin the study of Gibbs measures for subshifts, introduced by Meyerovitch \cite{meyerovitch:13}, Aaronson and Nakada \cite{aaronson:07}, Schmidt \cite{schmidt:97}, Petersen and Schmidt \cite{schpet}. In Meyerovitch's paper 
was given two different definitions for Gibbs measures, however, we will show
at the end of this section that both definitions coincide for subshifts of finite type. Although these definitions differ from the usual presented in the literature (cf. \cite{georgii:11}, \cite{ruelle:04},\cite{biss:12}, \cite{ny:08}, \cite{sarig:09}) 
since they do not involve conditional expectations, we will show in Section \ref{eq} that
they can be formulated, as usual, in terms of the so-called DLR equations. In addition, these definitions are closely related to another one provided by
Capocaccia \cite{cap:76}, in the sense that all these definitions coincide for subshifts of finite type.

Recall that a subshift $X$ of $\ds$ is a compact metrizable space and the shift action $T$ is an expansive continuous action of $\zd$ on $X$. Then, let $\gr$ be the Gibbs relation of $(X,T)$ given by
\begin{equation}
\gr = \big\{(x,y) \in X \times X : x_{\Lambda^c} = y_{\Lambda^c}\;\text{for some}\;\Lambda \subseteq \zd\; \text{finite}\big\}	
\end{equation}
(see Examples \ref{gr} and \ref{grshift}). Before we proceed to the next definition, we need to prove the following result.

\begin{lemma}\label{eita}
Let $f \in SV_d(X)$ and let $x,y \in X$ such that $x_{\Lambda_{m}^c} = y_{\Lambda_{m}^c}$ for some $m \in \mathbb{N}$.
Then, we have $\lim\limits_{N \to \infty} \sum\limits_{k \in \Lambda_N} |f(T^kx)-f(T^ky)| \leq 2|\Lambda_{m+1}| \cdot \|f\|_{SV_d}$.
\end{lemma}
\begin{proof}
First, let us show that the inequality
\begin{equation}\label{in}
\left|\left\{k \in \zd : \|k\| = n\right\}\right| \leq \frac{2|\Lambda_{m+1}|}{m^{d-1}}\,n^{d-1}
\end{equation}
holds whenever $n$ is a positive integer satisfying $n \geq m$. In fact, under this condition we have
\begin{eqnarray*}
\left|\left\{k \in \zd : \|k\| = n\right\}\right| & = & \left|\left\{k \in \zd : \|k\| \leq n\right\}\right| - \left|\left\{k \in \zd : \|k\| \leq n-1\right\}\right| \\
& = & (2n+1)^d - (2n-1)^d = \sum\limits_{l=0}^{d}\binom{d}{l}(2n)^{d-l}\left(1-(-1)^l\right) \\
& = & n^{d-1} \sum\limits_{l=0}^{d}\binom{d}{l} 2^{d-l} n^{1-l} \left(1-(-1)^l\right) \\
& \leq &  n^{d-1} \sum\limits_{l=0}^{d}\binom{d}{l} 2^{d-l} m^{1-l} \left(1-(-1)^l\right) \\
& \leq & \frac{2n^{d-1}}{m^{d-1}} \sum\limits_{l=0}^{d}\binom{d}{l} {\left(2m\right)}^{d-l} \\
& = & \frac{2|\Lambda_{m + 1}|}{m^{d-1}}\,n^{d-1}. 
\end{eqnarray*}
Now, given an integer $N$ such that $N > m$, it follows from inequality (\ref{in}) that
\begin{eqnarray*}
\sum\limits_{k \in \Lambda_N}|f(T^kx)-f(T^ky)| & = & \sum\limits_{k \in \Lambda_{m}}|f(T^kx)-f(T^ky)| + \sum\limits_{k \in \Lambda_N \backslash \Lambda_{m}}|f(T^kx)-f(T^ky)| \\
& \leq & 2|\Lambda_{m}| \cdot \|f\|_{\infty} + \sum\limits_{n = m}^{N-1} \sum_{\substack{k \in \zd \\ \|k\|=n}} |f(T^kx)-f(T^ky)| \\ 
& \leq & 2|\Lambda_{m+1}| \cdot \|f\|_{\infty} + \sum\limits_{n = m}^{N-1} \sum_{\substack{k \in \zd \\ \|k\|=n}} \delta_{\|k\| - (m-1)}(f) \\ 
& \leq & 2|\Lambda_{m+1}| \cdot \|f\|_{\infty} + \sum\limits_{n = m}^{N-1} \left|\left\{k \in \zd : \|k\| = n\right\}\right| \cdot \delta_{n - (m-1)}(f)\\
& \leq & 2|\Lambda_{m+1}| \cdot \|f\|_{\infty} + 2|\Lambda_{m + 1}|\cdot \sum\limits_{n = m}^{N-1} \left(\frac{n}{m}\right)^{d-1} \delta_{n - (m-1)}(f)\\
& = & 2|\Lambda_{m+1}| \cdot \|f\|_{\infty} + 2|\Lambda_{m + 1}|\cdot \sum\limits_{n = 1}^{N-m} \underbrace{\left(\frac{n+(m-1)}{m}\right)^{d-1}}_{\leq \,n^{d-1}} \delta_{n}(f). 
\end{eqnarray*}

Therefore, the result follows since the inequality
\begin{eqnarray*}
\sum\limits_{k \in \Lambda_N}|f(T^kx)-f(T^ky)| \leq 2|\Lambda_{m+1}| \cdot \|f\|_{SV_d}
\end{eqnarray*}
holds whenever $N > m$.
\end{proof}

As we previously mentioned, we will define a Gibbs measure for a function $f$ with $d$-summable variation on $X$ as a special kind of conformal measure. In order to do so, for each $f$ in $SV_{d}(X)$ let
us introduce a
$\gr$-cocycle associated to it.

\begin{mydef}
Given a function $f$ in $SV_d(X)$, we define the map $\phi_{f} : \gr \rightarrow \mathbb{R}$ by
\begin{equation}
\phi_{f}(x,y) = \sum\limits_{k \in \zd}f(T^k y) - f(T^k x),
\end{equation}
where the sum above is an unordered sum (see Apendix \ref{unsum}).
\end{mydef}

\begin{remark}
\begin{enumerate}[label=(\alph*),ref=\alph*]
\item It is easy to prove that $\phi_{f}$ is well defined. In fact, for each pair $(x,y)$ in $\gr$, without loss of generality we may assume that there is a positive integer $m$ 
such that $x_{\Lambda_{m}^{c}} = y_{\Lambda_{m}^{c}}$. It follows from Lemma \ref{eita} that $\lim\limits_{N \to \infty} \sum\limits_{k \in \Lambda_N} |f(T^k y)-f(T^k x)| \leq 2|\Lambda_{m+1}| \cdot \|f\|_{SV_d} < +\infty$,
and so according to Corolary \ref{svdap} (see Appendix \ref{ape}) the unordered sum $\sum\limits_{k \in \zd}f(T^k y) - f(T^k x)$ converges to a real number.

\item Let us prove that $\phi_{f}$ is in fact a $\gr$-cocycle. Since
\[\phi_{f}(x,y) = \lim\limits_{N \to \infty}\sum\limits_{k \in \Lambda_N}f(T^k y) - f(T^k x)\] 
for every pair $(x,y)$ in $\gr$,
then $\phi_{f}$ is a pointwise limit of a sequence of measurable functions on $\gr$. Thus the measurability of $\phi_{f}$ follows. 
Furthermore, for all pairs $(x,y)$ and $(y,z)$ in $\gr$, we have
\begin{eqnarray*}
\phi_{f}(x,z) &=& \lim\limits_{N \to \infty}\sum\limits_{k \in \Lambda_N}f(T^k z) - f(T^k x) \\
&=& \lim\limits_{N \to \infty}\sum\limits_{k \in \Lambda_N}f(T^k y) - f(T^k x) +\lim\limits_{N \to \infty}\sum\limits_{k \in \Lambda_N}f(T^k z) - f(T^k y) \\
&=& \phi_{f}(x,y) + \phi_{f}(y,z).
\end{eqnarray*}
\end{enumerate}
\end{remark}

\begin{mydef}[Gibbs measure]\label{gibbsm}
A Borel probability measure $\mu$ on $X$ is called a Gibbs measure for a function $f \in SV_d(X)$ 
if it is $(\phi_{f}, \gr)$-conformal.
\end{mydef}

Due to its technical difficulty, it can be very hard to deal with proofs where Gibbs measures are involved unless the generators of $\gr$ are known (see Proposition \ref{haaa}).
Then, let us introduce a subrelation of $\gr$, called topological Gibbs relation, and derive a weaker definition of a Gibbs measure which is easier to 
handle. We will show later that 
both definitions coincide in the case where $X$ is a subshift of finite type.	

Let us introduce the set
\begin{eqnarray*}
\mathcal{F}(X) \qquad \qquad \qquad \qquad \qquad \qquad \qquad \qquad \qquad \qquad \qquad \qquad \qquad \qquad \qquad \qquad \qquad \quad \;\;\;\\
\quad\;\;\;\,\defeq \left\{\varphi \in \mathsf{Homeo(}X{\mathsf{)}}: \text{exists a positive integer}\; n \;\text{such that}\; \varphi(x)_{{\Lambda}_{n}^{c}}=x_{{\Lambda}_{n}^{c}}\; \text{for all}\; x \in X\right\}.	
\end{eqnarray*}
It is straightforward to check that $\mathcal{F}(X)$ is a group of Borel automorphisms of $X$ with respect to the operation of composition of functions. 
The topological Gibbs relation will be defined as the equivalence relation generated by $\mathcal{F}(X)$ (see Example \ref{group}), but in order to do that, we need to show that
$\mathcal{F}(X)$ is countable. 

\begin{lemma}\label{satanasjr}
Given an arbitrary element $\varphi$ of $\mathcal{F}(X)$, there is a positive integer $n$ such that
\begin{itemize}
\item[(a)] the equality $\varphi (x)_{\Lambda_{n}^c}=x_{\Lambda_{n}^c}$ holds for every $x$ in $X$, and 
\item[(b)] for each pair $x,y$ of points in $X$ we have
$x_{\Lambda_{n}}=y_{\Lambda_{n}}$ if and only if $\varphi (x)_{\Lambda_{n}}=\varphi (y)_{\Lambda_{n}}$. 
\end{itemize}	
\end{lemma}

\begin{proof}
Let $m$ be a positive integer such that $\varphi (x)_{\Lambda_{m}^c}= x_{\Lambda_{m}^c}$ holds for every $x$ in $X$. 
Since $\varphi$ and $\varphi^{-1}$ are continuous functions on $X$, by compactness, it follows that both functions are uniformly continuous. Then, 
there is an integer $n \geq m$ such that for all points $x$ and $y$ in $X$ satisfying
$x_{\Lambda_{n}}=y_{\Lambda_{n}}$ we have $\varphi (x)_{\Lambda_{m}}=\varphi (y)_{\Lambda_{m}}$
and $\varphi^{-1} (x)_{\Lambda_{m}}=\varphi^{-1} (y)_{\Lambda_{m}}$.

Note that part (a) follows from the fact that $\Lambda_{n}^{c} \subseteq \Lambda_{m}^{c}$.
On the other hand, for any two elements $x$ and $y$ of $X$ such that
$x_{\Lambda_{n}}=y_{\Lambda_{n}}$ we have  
$\varphi (x)_{\Lambda_{n} \backslash \Lambda_{m}}= x_{\Lambda_{n} \backslash \Lambda_{m}} = y_{\Lambda_{n} \backslash \Lambda_{m}}=\varphi (y)_{\Lambda_{n} \backslash \Lambda_{m}}$.
Thus, the equality $\varphi (x)_{\Lambda_{n}}=\varphi (y)_{\Lambda_{n}}$ holds whenever $x$ and $y$ are of elements of $X$ that satisfy $x_{\Lambda_{n}}=y_{\Lambda_{n}}$.
One can easily prove an analogous result for $\varphi^{-1}$.
Therefore, part (b) follows.
\end{proof}

Let us write 
\[\mathcal{F}(X) = \bigcup\limits_{n \in \mathbb{N}}\mathcal{F}_n(X),\]
where $\mathcal{F}_n(X)$ is the set of all homeomorphisms $\varphi$ of $X$ satisfiying items (a) and (b) from Lemma \ref{satanasjr}.
\begin{remark}\label{fnx}
\begin{enumerate}[label=(\alph*),ref=\alph*]
\item \label{fnfnfn}It is easy to verify that $\mathcal{F}_{n}(X) \subseteq \mathcal{F}_{n+1}(X)$ for every $n \in \mathbb{N}$. 

\item Let us show that each $\mathcal{F}_n(X)$ is a finite set, and finally conclude that $\mathcal{F}(X)$ is countable. Given an element $\varphi$ of $\mathcal{F}_{n}(X)$, let us 
define a function $\widetilde{\varphi} : X_{\Lambda_{n}} \rightarrow X_{\Lambda_{n}}$ as follows. For each $\omega$ in $X_{\Lambda_{n}}$ we let $\widetilde{\varphi}(\omega) = \varphi(x)_{\Lambda_{n}}$, 
where $x$ is an arbitrarily choosen element of $[\omega]$ (note that $\widetilde{\varphi}$ is well defined, since $[\omega] \neq \emptyset$ and the value of $\widetilde{\varphi}(\omega)$ does not depends on the 
choice of the element $x$ of $[\omega]$). One can easily check that the mapping $\varphi \mapsto \widetilde{\varphi}$ establishes a one-to-one correspondence between $\mathcal{F}_{n}(X)$ 
and the set of all functions from $X_{\Lambda_{n}}$ into itself. Therefore, it follows that $\mathcal{F}_n(X)$ is a finite set.
\end{enumerate}
\end{remark}

Now, we define the topological Gibbs relation $\gr^{0}$ as being the equivalence relation generated by $\mathcal{F}(X)$, i.e.,
\begin{equation}
\gr^{0} \defeq \EuScript{R}_{\mathcal{F}(X)} = \left\{(x,y) \in X \times X : y=\varphi(x)\;\text{for some}\;\varphi \in \mathcal{F}(X) \right\}.
\end{equation}
Observe that $\gr^{0}$ is a subset of $\gr$, and given a function $f$ in $SV_{d}(X)$ the restriction of $\phi_{f}$ to $\gr^{0}$ is a $\gr^{0}$-cocycle. 
Now, we are able to introduce the concept of a topological Gibbs measure.
\begin{mydef}[Topological Gibbs measure]\label{topgibbsm}
A Borel probability measure $\mu$ on $X$ is called a topological Gibbs measure for a function $f \in SV_d(X)$ 
if it is $\left(\phi_{f}\restriction_{\gr^{0}}, \gr^0\right)$-conformal.
\end{mydef}

\begin{remark}\label{capeeeta}
It follows from Proposition \ref{haaa} that a Borel probability measure $\mu$ on $X$ is a topological Gibbs measure for a function $f$ in $SV_{d}(X)$ if and only if for each $\varphi$ in $\mathcal{F}(X)$ 
the measure $\varphi_{\ast}\mu$ is absolutely continuous with respect to $\mu$ and 
the equation
\begin{equation}
\frac{d \varphi_{\ast}\mu}{d \mu}(x) = e^{\phi_{f}(x,\varphi^{-1}(x))} 		
\end{equation}	
holds for $\mu$-almost every $x$ in $X$.
\end{remark}

Using Proposition \ref{boraut} and Remark \ref{capeeeta}, one can easily prove that every Gibbs measure for a function $f$ in $SV_{d}(X)$ is also a topological Gibbs measure for $f$. 
The converse is not necessarily true, as we will see in the next section. 

Our next result says that $\gr = \gr^0$ in the case where $X$ is a subshift of finite type.
The main consequence of this result is that, under the same assumption, both notions of Gibbs measures given by Definitions \ref{gibbsm} and \ref{topgibbsm} coincide.

\begin{prop}
If $X$ is a subshift of finite type, then $\gr = \gr^0$.
\end{prop}

\begin{proof}
It is sufficient to prove that $\gr \subseteq \gr^0$. Given an arbitrary element $(x,y)$ of $\gr$ there exists a positive integer $n$ 
such that $x_{\Lambda_{n}^c} = y_{\Lambda_{n}^c}$, and $X = \mathsf{X}_{\mathcal{F}}$ for some collection $\mathcal{F}$ of patterns on $\Lambda_{n}$ (see Remark \ref{remarksft}).

Let $\omega = x_{\Lambda_{3n}}$ and $\eta = y_{\Lambda_{3n}}$. Let us show that for every $z$ in $X$ we have $\omega z_{\Lambda_{3n}^c} \in X$ if and only if $\eta z_{\Lambda_{3n}^c} \in X$.
Indeed, if $\omega z_{\Lambda_{3n}^c}$ belongs to $X$, then $\sigma^{l}\left(\eta z_{\Lambda_{3n}^c}\right)_{\Lambda_{n}} = \big(\sigma^{l}y\big)_{\Lambda_{n}}$ holds whenever
$\|l\| \leq 2n$, and $\sigma^{l}\left(\eta z_{\Lambda_{3n}^c}\right)_{\Lambda_{n}} = \sigma^{l}\left(\omega z_{\Lambda_{3n}^c}\right)_{\Lambda_{n}}$ holds whenever
$\|l\| > 2n$. It follows that $\sigma^{l}\left(\eta z_{\Lambda_{3n}^c}\right)_{\Lambda_{n}} \notin \mathcal{F}$ for each $l \in \zd$, thus $\eta z_{\Lambda_{3n}^c}$ belongs to $X$. The proof of the converse is analogous.

Define $\varphi : X \rightarrow X$ by

\begin{equation}
\varphi(z) = 
\begin{cases}
\omega z_{\Lambda_{3n}^c} &\text{if}\;z\in [\eta], \\
\eta z_{\Lambda_{3n}^c} &\text{if}\;z\in [\omega], \\
z &	\text{otherwise}.
\end{cases}
\end{equation}
It is straightforward to show that $\varphi \circ \varphi = \mathsf{id}_{X}$ and $\varphi$ is continuous. Thus, we conclude that $\varphi$ is an element of $\mathcal{F}(X)$
and $(x,y) = (x,\varphi(x)) \in \gr^{0}$.
\end{proof}

\section{Connection with equilibrium measures}\label{geq}

Recall that in Section \ref{pressure} we introduced the definition of an equilibrium measure in a more general context.
Thus, we devote this section to provide a connection between Gibbs and equilibrium measures for subshifts. Our main aim is to prove the following result.

\begin{teo}[Meyerovitch]\label{mey}
Let $X \subseteq \ds$ be a subshift and let $f$ be a function in $SV_d(X)$. Then, any equilibrium measure $\mu$ for $f$
is a topological Gibbs measure for $f$.
\end{teo}

Before entering into the proof of this theorem let us give some comments and prove a few preliminary results. 

As we commented in the previous section, it is not true that every topological Gibbs measure is also a Gibbs measure. In fact, Meyerovitch \cite{meyerovitch:13} provided an
example of a subshift that admits an equilibrium measure that is not a Gibbs measure. Then, using the theorem above, our assertion follows. 

In view of Theorem \ref{mey}, if we assume that $X$ is a subshift of finite type, we obtain the following corollary.

\begin{cor}\label{corsft}
Let $X \subseteq \ds$ be a SFT and let $f$ be a function in $SV_d(X)$. Then, any equilibrium measure $\mu$ for $f$
is a Gibbs measure for $f$.
\end{cor}

Now, let us turn to the 
preliminary results.




\begin{lemma}\label{ya}
Let $f$ be a function in $SV_d(X)$ and let $\varphi$ be an element of $\mathcal{F}(X)$. Then, the function $F : X \rightarrow \mathbb{R}$ defined by $F(x) = \phi_{f}(x,\varphi(x))$ is continuous. 
\end{lemma}

\begin{proof}
Let $(F_{N})_{N \in \mathbb{N}}$ be a sequence of real-valued functions on $X$, where each $F_N$ is given by $F_N(x) = \sum\limits_{k \in \Lambda_N} f(T^k \varphi(x)) - f(T^k x)$.{}
Note that $(F_{N})_{N \in \mathbb{N}}$ is a sequence of continuous functions that converges pointwise to $F$. If we prove that $(F_N)_{N \in \mathbb{N}}$ is a Cauchy sequence
with respect to the norm $\|\cdot\|_{\infty}$, then we will conclude that $F$ is the limit of this sequence (with respect to $\|\cdot\|_{\infty}$) and the result follows.

Let $m$ be a positive integer such that $\varphi(x)_{\Lambda_{m}^c} = x_{\Lambda_{m}^c}$ holds for all $x$ in $X$. 
Given an arbitrary element $x$ of $X$, let $y = \varphi(x)$. 
Since we are under the same hypotheses of Lemma \ref{eita}, for all integers $M$ and $N$ such that $N>M \geq m$, we can use equation (\ref{in}) in order to obtain 

\begin{eqnarray*}
\left|F_N(x)-F_M(x)\right| & = & \left|\sum\limits_{k \in \Lambda_N\backslash \Lambda_M}f(T^k y) - f(T^k x)\right| \leq \sum\limits_{n=M}^{N-1}\sum_{\substack{k \in \zd \\ \|k\| = n}} \underbrace{\big|f(T^k y) - f(T^k x)\big|}_{\leq\, \delta_{\|k\|-(m-1)}(f)} \\
& \leq & \sum\limits_{n=M}^{N-1}\left|\{k \in \zd : \|k\| = n\}\right| \delta_{n-(m-1)}(f) \\
& \leq & 2\left|\Lambda_{m+1}\right| \sum\limits_{n=M}^{N-1}\left(\frac{n}{m}\right)^{d-1}\,\delta_{n-(m-1)}(f) \\ 
& = & 2\left|\Lambda_{m+1}\right| \sum\limits_{n=M-(m-1)}^{N-m}\left(\frac{n+(m-1)}{m}\right)^{d-1}\,\delta_{n}(f) \\
& \leq & 2\left|\Lambda_{m+1}\right|\sum\limits_{n=M-(m-1)}^{N-m}n^{d-1}\,\delta_{n}(f).\\ 
\end{eqnarray*}
It follows that
\begin{equation}
\|F_{N} - F_{M}\|_{\infty} \leq  2\left|\Lambda_{m+1}\right|\sum\limits_{n=M-(m-1)}^{N-m}n^{d-1}\,\delta_{n}(f)	
\end{equation}
holds whenever $M$ and $N$ satisfy $N > M \geq m$.
Observe that for every positive number $\epsilon$ there is an integer $N_0 \geq m$ such that the conditions $N > M \geq N_0$ imply that $\sum\limits_{n=M-(m-1)}^{N-m}n^{d-1}\,\delta_{n}(f) < \frac{\epsilon}{2|\Lambda_{m+1}|}$.
Thus $\|F_N-F_M\|_{\infty} < \epsilon$ holds whenever $M$ and $N$ satisfy $N > M \geq N_0$.
\end{proof}

The next proposition will play an important role in the proof of Theorem \ref{mey}.

\begin{prop}\label{lim}
Let $(f_n)_{n \in \mathbb{N}}$ be a sequence in $SV_d(X)$ that converges to $f$ in norm. If $(\mu_n)_{n \in \mathbb{N}}$ is
a sequence of Borel probability measures that converges weakly to $\mu$ and each $\mu_n$ is a topological Gibbs measure for $f_{n}$,
then $\mu$ is a topological Gibbs measure for $f$.
\end{prop}

\begin{proof}
According to Remark \ref{capeeeta} it is sufficient to show that
for each $\varphi$ in $\mathcal{F}(X)$ 
the measure $\varphi_{\ast}\mu$ is absolutely continuous with respect to $\mu$ and 
\begin{equation}
\frac{d \varphi_{\ast}\mu}{d \mu}(x) = e^{\phi_{f}(x,\varphi^{-1}(x))} \quad \mu\text{-a.e}.		
\end{equation}	

Let us consider an element $\varphi$ of $\mathcal{F}(X)$. Given a real-valued continuous function $g$ on $X$, we have
\begin{equation*}
\int_{X}g\, d \varphi_{\ast}\mu = \int_{X}g\circ \varphi\, d\mu = \lim\limits_{n \to \infty}\int_{X}g\circ \varphi\, d\mu_n 
 =  \lim\limits_{n \to \infty}\int_{X}g\, d\varphi_{\ast}\mu_n. \\
\end{equation*}
Since each $\mu_n$ is a topological Gibbs measure for $f_{n}$, we obtain 
\begin{equation}\label{carai1}
\int_{X}g\, d \varphi_{\ast}\mu = \lim\limits_{n \to \infty}\int_{X}g(x)\,e^{\phi_{f_n}(x,\varphi^{-1}(x))}d\mu_n (x).   
\end{equation}

For every positive integer $n$, we have

\begin{eqnarray*}
\Big|\int_{X}g(x)\,e^{\phi_{f_n}(x,\varphi^{-1}(x))}d\mu_n (x)&-&\int_{X}g(x)\,e^{\phi_{f}(x,\varphi^{-1}(x))}d\mu (x)\Big|  \leq \\
\leq &&\Big|\int_{X}g(x)\,e^{\phi_{f_n}(x,\varphi^{-1}(x))}d\mu_n (x) - \int_{X}g(x)\,e^{\phi_{f}(x,\varphi^{-1}(x))}d\mu_n (x)\Big| \\
 &+&\Big|\int_{X}g(x)\,e^{\phi_{f}(x,\varphi^{-1}(x))}d\mu_n (x) - \int_{X}g(x)\,e^{\phi_{f}(x,\varphi^{-1}(x))}d\mu (x)\Big| \\
\leq &&\int_{X}\left|g(x)\right|\cdot\left|e^{\phi_{f_n}(x,\varphi^{-1}(x))}- e^{\phi_{f}(x,\varphi^{-1}(x))}\right|d\mu_n (x) \\
 &+&\Big|\int_{X}g(x)\,e^{\phi_{f}(x,\varphi^{-1}(x))}d\mu_n (x) - \int_{X}g(x)\,e^{\phi_{f}(x,\varphi^{-1}(x))}d\mu (x)\Big| \\
\leq && \|g\|_{\infty}\cdot\left\|\,e^{\phi_{f_n}(\,\cdot\,,\varphi^{-1}(\cdot))} - e^{\phi_{f}(\,\cdot\,,\varphi^{-1}(\cdot))}\right\|_{\infty} \\
 &+&\Big|\int_{X}g(x)\,e^{\phi_{f}(x,\varphi^{-1}(x))}d\mu_n (x) - \int_{X}g(x)\,e^{\phi_{f}(x,\varphi^{-1}(x))}d\mu (x)\Big|. 
\end{eqnarray*}
Using Lemma \ref{eita}, the reader can easily verify that there is a positive integer $m$ such that $|\phi_{f}(x,\varphi^{-1}(x)) - \phi_{f_{n}}(x,\varphi^{-1}(x))|  = |\phi_{f-f_{n}}(x,\varphi^{-1}(x)) | \leq 2|\Lambda_{m+1}| \cdot \|f-f_{n}\|_{SV_{d}}$ 
holds for each point $x$ in $X$ and each positive integer $n$. It follows that $e^{\phi_{f_n}(\,\cdot\,,\varphi^{-1}(\cdot))}$ converges uniformly to $e^{\phi_{f}(\,\cdot\,,\varphi^{-1}(\cdot))}$.
Note that Lemma \ref{ya} implies that the function $x \mapsto g(x) \cdot e^{\phi_{f}(x,\varphi^{-1}(x))}$ is continuous, thus $\lim\limits_{n \to \infty}\int_{X}g(x)\,e^{\phi_{f}(x,\varphi^{-1}(x))}d\mu_n (x) = \int_{X}g(x)\,e^{\phi_{f}(x,\varphi^{-1}(x))}d\mu (x)$.
Therefore, we obtain

\begin{equation}\label{carai2}
\lim\limits_{n \to \infty}\int_{X}g(x)\,e^{\phi_{f_n}(x,\varphi^{-1}(x))}d\mu_n (x) = \int_{X}g(x)\,e^{\phi_{f}(x,\varphi^{-1}(x))}d\mu (x).
\end{equation}

Comparing equations (\ref{carai1}) and (\ref{carai2}) we conclude that
\[\int_{X}g\, d \varphi_{\ast}\mu = \int_{X}g(x)\,e^{\phi_{f}(x,\varphi^{-1}(x))}d\mu (x)\]
holds for all real-valued continuous function $g$ on $X$, and the result follows.
\end{proof}

\begin{proof}[Proof of Theorem \ref{mey}]

Let us assume that the theorem is valid for local functions on $X$. As we commented in Section \ref{pressure}, it follows from the expansivity of the shift action $T$ that its topological entropy 
$\sup\limits_{\mu \in \mathcal{M}(T)}h_{\mu}(T)$ is finite. Then, let us consider the pressure function $p : SV_{d}(X) \rightarrow \mathbb{R}$ defined by
\begin{equation}
p(f) = 	\sup\limits_{\mu \in \mathcal{M}(T)}\left\{h_{\mu}(T) + \int_{X}f\,d\mu\right\}.
\end{equation}
It is well known that $p$ is a convex function that satisfies $|p(f_{1}) - p(f_{2})| \leq \|f_{1}-f_{2}\|_{\infty} \leq \|f_{1}-f_{2}\|_{SV_{d}}$ for all
$f_{1}$ and $f_{2}$ in $SV_{d}(X)$ (see \cite{keller:98}). 

A theorem due to Lanford and Robinson \cite{lanfordrob:68} states that given a real-valued continuous convex function
$p$ defined on a separable Banach space $\mathscr{X}$, with a dense subset $\mathscr{X}_{0}$, any linear functional that is tangent to the graph of $p$ at $f \in \mathscr{X}$ belongs to
the weak* closure of the convex hull of the set
\begin{equation}
\left\{\lim\limits_{n \to \infty} \psi_{n} : \psi_{n}\;\text{is tangent to}\;p\;\text{at}\;f_{n},\;\text{where}\;(f_{n})_{n \in \mathbb{N}}\;\text{converges to}\;f\;\text{in norm}\right\}. 	
\end{equation}   
In our context, we have $\mathscr{X} = SV_{d}(X)$ and $\mathscr{X}_{0} = \text{Loc}(X)$. Thus, every equilibrium measure $\mu$ for a function $f$ in $SV_{d}(X)$ is the weak limit
of a sequence $(\mu_{n})_{n\in \mathbb{N}}$ of Borel probability measures such that each $\mu_{n}$ is an equilibrium measure for $f_{n}$, where $(f_{n})_{n \in \mathbb{N}}$ is a sequence of local functions that conveges to $f$ in norm. 
Since we assumed that every equilibrium measure for a local function is a topological Gibbs measure, it follows from Proposition \ref{lim} that $\mu$ is a topological
Gibbs measure for $f$.

From now on, we will concentrate our efforts to show that the theorem is valid for local functions. Let us show that it suffices to prove the result in the case where $f(x)$ depends only on $x_{\mathbf{0}}$. 
Let $f$ be a local function on $X$, let $\Lambda$ be a nonempty finite subset of $\zd$ such that $f(x)$ depends only on $x_{\Lambda}$, and let $Y$ be a subset of the full shift $(X_{\Lambda})^{\zd}$ defined by
\begin{equation}
Y = \left\{\left(T^{i}(x)_{\Lambda}\right)_{i \in \zd} : x \in X\right\}.	
\end{equation} 
It is straightforward to show that $Y$ is a subshift of $(X_{\Lambda})^{\zd}$ and that the map $\Phi : X \rightarrow Y$ defined by letting
\begin{equation}
\Phi(x) = \left(T^{i}(x)_{\Lambda}\right)_{i \in \zd}	
\end{equation}
for each $x$ in $X$, is a homeomorphism. If we let $S$ be the shift action on $Y$, then it is clear that 
\begin{equation}
\Phi \circ T^{j} = S^{j} \circ \Phi	
\end{equation}
holds for every $j$ in $\zd$.

It is well known that for every $T$-invariant Borel probability measure $\mu$ on $X$, the measure $\Phi_{\ast}\mu$ is an $S$-invariant Borel probability measure on $Y$, and we have $h_{\mu}(T) = h_{\Phi_{\ast}\mu}(S)$. 
Thus, the reader can easily check that if $\mu$ is an equilibrium measure for $f$, then $\Phi_{\ast}\mu$ is an equilibrium measure for $f \circ \Phi^{-1}$. Since
$f \circ \Phi^{-1}$ is an element of $\text{Loc}(Y)$ such that $f \circ \Phi^{-1}(y)$ depends ony on $y_{\mathbf{0}}$, it follows from our assumption that $\Phi_{\ast}\mu$ is a
topological Gibbs measure for $f \circ \Phi^{-1}$. Then, for every element $\varphi$ of $\mathcal{F}(X)$ the function $\widetilde{\varphi}$ defined by
$\widetilde{\varphi} = \Phi \circ \varphi \circ \Phi^{-1}$ belongs $\mathcal{F}(Y)$, and satisfies
\begin{eqnarray*}
\varphi_{\ast}\mu (B) &=& \widetilde{\varphi}_{\ast}(\Phi_{\ast}\mu) (\Phi(B)) = \int_{\Phi(B)}	e^{\phi_{f \circ \Phi^{-1}}(y,\widetilde{\varphi}^{-1}(y))} d\Phi_{\ast}\mu(y) \\
&=& \int_{B}e^{\phi_{f \circ \Phi^{-1}}(\Phi(x),\widetilde{\varphi}^{-1}\circ \Phi(x))} d\mu(x) \\
&=& \int_{B}e^{\phi_{f \circ \Phi^{-1}}(\Phi(x),\Phi \circ \varphi^{-1}(x))} d\mu(x)  \\
&=& \int_{B}e^{\phi_{f}(x,\varphi^{-1}(x))} d\mu(x) 
\end{eqnarray*}
for each Borel subset $B$ of $X$. Using Remark \ref{capeeeta}, we conclude that $\mu$ is a Gibbs measure for $f$.

In the case where $f$ is a local function that depends only on $x_{\mathbf{0}}$, see \cite{meyerovitch:13}.
\end{proof}

\section{Characterization of Gibbs measures}\label{eq}
In this section we provide another characterizations of Gibbs measures on subshifts in order to connect both definitions presented in Section \ref{lalala} with
more familiar definitions presented in the literature.	

\subsection{Capocaccia's definition}
Let us present the definition of a Gibbs state, given by Capocaccia \cite{cap:76}, for compact metrizable spaces where $\zd$ acts by an expansive group of homeomorphisms, and relate this notion with the definitions given
in Section \ref{lalala}. 

Let $X$ be a nonempty compact metrizable space, and let $T$ be an expansive continuous action of $\zd$ on $X$. 
Recall that the Gibbs relation of $(X,T)$ is defined by

\[\gr = \left\{(x,y) \in X \times X : \lim\limits_{\|k\| \to \infty} \rho(T^k x,T^k y) = 0 \right\},\]
where $\rho$ is a metric on $X$ which induces its topology, and the relation given above does not depends on the choice of the metric $\rho$.
 
In Capocaccia's terminology, any two points $x$ and $y$ in $X$ are called conjugate if the pair $(x,y)$ belongs to $\gr$. 
And, if $O$ is an open subset of $X$, then a mapping $\varphi:O \rightarrow X$ is said to be conjugating if $\rho(T^kx,T^k\varphi(x))$ tends uniformly to zero as $\|k\|$ approaches infinity.
Note that the notion of a conjugating mapping also does not depends on the choice of the metric $\rho$.

In the remainder of this section we will always assume the following condition.

\begin{assumption}\label{assump} 
Suppose that for every pair of conjugate points $x,y \in X$ there is an open subset $O$ of $X$ containing the point $x$, and a conjugating mapping
$\varphi: O \rightarrow X$ that is continuous at $x$ and satisfies $\varphi(x) = y$.
\end{assumption}

It was proved in \cite{cap:76} that the assumption made above implies that
for every such mapping $\varphi$ there is an open set $\widetilde{O} \subseteq O$
containing $x$ such that $\varphi$ is a homeomorphism of $\widetilde{O}$ onto $\varphi(\widetilde{O})$. 
Also, if $\varphi'$ is a mapping that has the same
properties as $\varphi$, then $\varphi$ and $\varphi'$ agree on some neighborhood of $x$.

\begin{ex}
Note that Assumption \ref{assump} is satisfied in the case where $X$ is a subshift of finite type. Indeed, since we have $\gr = \gr^{0}$, then for each pair $(x,y)$ in $\gr$ there is	
an element $\varphi$ of $\mathcal{F}(X)$ such that $y = \varphi(x)$. It is straightforward to show that $\varphi$ is a conjugating mapping.
\end{ex}

In the same way as we did in Section \ref{rn}, for each Borel subset $O$ of $X$ we will denote the restriction of a Borel measure $\mu$ on $X$ to the $\sigma$-algebra of Borel subsets of $O$ by $\mu_O$.

\begin{mydef}[Capocaccia's definition for a Gibbs state]
We say that a family $\mathscr{I} = (R_{(O,\varphi)})$ is a family of multipliers if
\begin{itemize}
\item[(a)] $\mathscr{I}$ is indexed by all pairs $(O,\varphi)$, where $O$ is an open subset of $X$ and $\varphi$ is a conjugating homeomorphism defined on $O$, and 
$R_{(O,\varphi)}$ is a positive continuous function on $O$,
\item[(b)] if $O' \subseteq O$ and $\varphi' = \varphi\restriction_{O'}$, then $R_{(O',\varphi')} = R_{(O,\varphi)}\restriction_{O'}$, and
\item[(c)] if $O \subseteq O'$ and $\varphi'(O) \subseteq O''$, then
\begin{equation*}
R_{(O,\varphi''\circ \varphi'\restriction_{O})} = R_{(O',\varphi')}\restriction_{O}\cdot R_{(O'',\varphi'')}\circ \varphi'\restriction_{O}.
\end{equation*}
\end{itemize}
A Borel probability measure $\mu$ on $X$ is said to be a Gibbs state for the family of multipliers $\mathscr{I}$ if
\begin{equation}\label{gibbscap}
\varphi_{\ast}\big(R_{(O,\varphi)}\,d\mu_{O}\big) = \mu_{\varphi(O)}\qquad \text{holds for every pair}\;(O,\varphi).
\end{equation}
\end{mydef}

\begin{remarks}
\begin{enumerate}[label=(\alph*),ref=\alph*]
\item Observe that equation (\ref{gibbscap}) makes sense. In fact,
according to Theorem $8.3.7$ from \cite{cohn:13}, the image $\varphi(O)$ of a one-to-one measurable function $\varphi$ from a Borel 
subset $O$ of a Polish space $X$ into another Polish space $Y$, is Borel set of $Y$.  
Thus, for each pair $(O,\varphi)$, since $O$ is a Borel subset of $X$, it follows that $\varphi(O)$ is also a Borel set. 

\item It is easy to verify that $\mu$ is a Gibbs state for the family $\mathscr{I}$ if and only if for each pair $(O,\varphi)$ the equation
\begin{equation}\label{gibbscapp}
\frac{d(\mu_{\varphi(O)}\circ \varphi)}{d\mu_{O}} = R_{(O,\varphi)} \quad \text{holds $\mu$-almost everywhere on}\; O.
\end{equation}
\end{enumerate}
\end{remarks}

\begin{teo}
Let $X$ be a SFT, and let $f$ be a function in $SV_d(X)$. Then, a Borel probability measure $\mu$ on $X$ is a Gibbs measure for
$f$ if and only if $\mu$ is a Gibbs state for the family of multipliers $\mathscr{I}$ defined by  
\[R_{(O,\varphi)}(x) = e^{\phi_f(x,\varphi(x))}\quad \text{at each point $x \in O$,}\]
for all pairs $(O,\varphi)$.
\end{teo}

\begin{proof}
It is straightforward to check that $\mathscr{I}=(R_{(O,\varphi)})$ is a family of multipliers.

Let $\varphi$ be a conjugating homeomorphism defined on $O$. Since 
$\varphi : O \rightarrow \varphi(O)$ is an isomorphism such that $\mathsf{gr}(\varphi) \subseteq \gr$, it follows from 
Proposition \ref{boraut} that
\[\frac{d\mu_{\varphi(O)}\circ \varphi}{d\mu_O}(x)= D_{\mu,\gr}(\varphi(x),x) = e^{\phi_f(x,\varphi(x))}\]
holds for $\mu$-almost every point $x$ in $O$.  

Conversely, since each element $\varphi$ of $\mathcal{F}(X)$ is a conjugating homeomorphism, then the equation
\[\frac{d\mu\circ\varphi}{d\mu} = e^{\phi_{f}(x,\varphi(x))} \quad \text{holds for $\mu$-almost every $x$ in $X$.}\]
It follows from Remark \ref{capeeeta} that $\mu$ is a topological Gibbs measure for $f$, and using the fact that $X$ is a SFT, 
we conclude that $\mu$ is a Gibbs measure for $f$.
\end{proof}

\subsection{DLR equations}

In this section we provide an alternative characterization of Gibbs measures on subshifts of finite type in terms of 
conditional expectations by means of the so-called DLR equations. Due to its probabilistic
interpretation, this approach is widely adopted in many textbooks on statistical mechanics  
(e.g. \cite{georgii:11},\cite{ruelle:04},\cite{ny:08}).  

First, let us introduce some notation. 
We will denote the collection of all nonempty finite subsets of $\zd$
by $\mathscr{S}$.
Recall that for each $j$ in $\zd$ the projection of the full shift $\ds$ onto the $j$-th coordinate is the map $\pi_j : \ds \rightarrow \mathcal{A}$ 
defined by letting $\pi_{j}(x) = x_{j}$ for each $x = (x_i)_{i \in \zd}$.
Given an arbitrary subset $\Delta$ of $\zd$, let $\EuScript{F}_{\Delta}$ denote the smallest $\sigma$-algebra of subsets of $\ds$ 
which contains the collection  
\begin{equation}
\left\{\pi_i^{-1}(A): i \in \Delta, A \subseteq \mathcal{A}\right\}.
\end{equation}

Now, if we let $X$ be a subshift of $\ds$, then we will denote by $\mathscr{F}$ its Borel $\sigma$-algebra and 
by $\mathscr{F}_{\Delta}$ the restriction of the $\sigma$-algebra $\EuScript{F}_{\Delta}$ to $X$. 
The reader can easily check that $\mathscr{F}_{\Delta}$ is a sub-$\sigma$-algebra of $\mathscr{F}$.
In the following, for each real-valued function $f$ defined on $X$ and each positive integer $n$, for notational convenience we will write $f_n$ instead of $\sum\limits_{i \in \Lambda_n}f\circ T^i$.

\begin{lemma}\label{xenaaprincesaguerreira}
Let $X \subseteq \ds$ be a subshift, let $f$ be a function in $SV_d(X)$, and let $\Lambda \in \mathscr{S}$. Then, the limit	
\begin{equation}\label{limlimlim}
\lim\limits_{n \to \infty}\frac{e^{f_n(\omega x_{\Lambda^c})}\bold{1}_{\{\omega x_{\Lambda^c}\in X\}}}{\sum\limits_{\eta\, \in \mathcal{A}^{\Lambda}}e^{f_n(\eta x_{\Lambda^c})}\bold{1}_{\{\eta x_{\Lambda^c}\in X\}}} 
\end{equation}
exists for each $\omega \in \mathcal{A}^{\Lambda}$ and each $x \in X$, moreover, it is a nonnegative real number.
\end{lemma}

\begin{proof}
Note that for each positive integer $n$, we have 
\[\sum\limits_{\eta\, \in \mathcal{A}^{\Lambda}}e^{f_n(\eta x_{\Lambda^c})}\bold{1}_{\{\eta x_{\Lambda^c}\in X\}} > 0.\]
In the case where $\omega x_{\Lambda^c}$ does not belong to $X$, the limit given by equation (\ref{limlimlim}) is equal to $0$. 
Otherwise, if $\omega x_{\Lambda^c}$ belongs to $X$, then
\begin{eqnarray*}
\frac{e^{f_n(\omega x_{\Lambda^c})}\bold{1}_{\{\omega x_{\Lambda^c}\in X\}}}{\sum\limits_{\eta\, \in \mathcal{A}^{\Lambda}}e^{f_n(\eta x_{\Lambda^c})}\bold{1}_{\{\eta x_{\Lambda^c} \in X\}}} 
&=& \frac{e^{f_n(\omega x_{\Lambda^c})}}{\sum\limits_{\eta\, \in \mathcal{A}^{\Lambda}}e^{f_n(\eta x_{\Lambda^c})}\bold{1}_{\{\eta x_{\Lambda^c} \in X\}}} \\ 	
&=& \frac{1}{\sum\limits_{\eta\, \in \mathcal{A}^{\Lambda}}e^{-f_n(\omega x_{\Lambda^c})}\cdot\left(e^{f_n(\eta x_{\Lambda^c})}\bold{1}_{\{\eta x_{\Lambda^c} \in X\}}\right)} \\
&=& \frac{1}{\sum\limits_{\eta\, \in \mathcal{A}^{\Lambda}}\exp{\left(\sum\limits_{i \in \Lambda_n}f\circ T^{i}(\eta x_{\Lambda^c})-f\circ T^{i}(\omega x_{\Lambda^c})\right)}\bold{1}_{\{\eta x_{\Lambda^c} \in X\}}} \\	
\end{eqnarray*}
holds for every positive integer $n$. It is easy to prove that 
\[\lim\limits_{n \to \infty}\sum\limits_{\eta\, \in \mathcal{A}^{\Lambda}}\exp{\left(\sum\limits_{i \in \Lambda_n}f\circ T^{i}(\eta x_{\Lambda^c})-f\circ T^{i}(\omega x_{\Lambda^c})\right)}\bold{1}_{\{\eta x_{\Lambda^c} \in X\}} = \sum\limits_{\eta\, \in \mathcal{A}^{\Lambda}} e^{\phi_f(\omega x_{\Lambda^c},\eta x_{\Lambda^c})}\bold{1}_{\{\eta x_{\Lambda^c} \in X\}} > 0.\]
Therefore, it follows that
\begin{equation}\label{oiadlr}
\lim\limits_{n \to \infty}\frac{e^{f_n(\omega x_{\Lambda^c})}\bold{1}_{\{\omega x_{\Lambda^c}\in X\}}}{\sum\limits_{\eta\, \in \mathcal{A}^{\Lambda}}e^{f_n(\eta x_{\Lambda^c})}\bold{1}_{\{\eta x_{\Lambda^c} \in X\}}} 
= \frac{1}{\sum\limits_{\eta\, \in \mathcal{A}^{\Lambda}} e^{\phi_f(\omega x_{\Lambda^c},\eta x_{\Lambda^c})}\bold{1}_{\{\eta x_{\Lambda^c} \in X\}}} > 0.
\end{equation}
\end{proof}

\begin{mydef}\label{especificacao15}
Let $X$ be a subshift of $\ds$, and let $f$ be a function in $SV_d(X)$. Let us define 
a family $\gamma = (\gamma_{\Lambda})_{\Lambda \in \mathscr{S}}$, where each $\gamma_{\Lambda} : \mathscr{F}\times X\rightarrow {[}0,+\infty)$ is defined by letting
\begin{equation}\label{gamma}
\gamma_{\Lambda}(A|\,x) = \lim\limits_{n \to \infty}\frac{\sum\limits_{\omega \in \mathcal{A}^{\Lambda}}e^{f_n(\omega x_{\Lambda^c})}\bold{1}_{\{\omega x_{\Lambda^c}\in A\}}}{\sum\limits_{\eta\, \in \mathcal{A}^{\Lambda}}e^{f_n(\eta x_{\Lambda^c})}\bold{1}_{\{\eta x_{\Lambda^c}\in X\}}}  
\end{equation}
for each $A \in \mathscr{F}$ and each point $x \in X$.
\end{mydef}
\begin{remark}
 Using Lemma \ref{xenaaprincesaguerreira}, the reader can easily verify that equation (\ref{gamma}) 
is well defined, and the relation
\begin{equation}\label{lalalalallalalala}
\gamma_{\Lambda}(A|x) =  \sum\limits_{\omega \in \mathcal{A}^{\Lambda}} \gamma_{\Lambda}([\omega]|x) \bold{1}_{\{\omega x_{\Lambda^c} \in A\}}	
\end{equation}
holds for each $A \in \mathscr{F}$ and each $x \in X$.
\end{remark}

Now, let us show a few properties satisfied by the family $\gamma$ given in Definition \ref{especificacao15}.

\begin{fact}\label{gammaprop}
Given a nonempty finite subset $\Lambda$ of $\zd$, then
\begin{enumerate}[label=(\alph*),ref=\alph*]
\item $\gamma_{\Lambda}(\cdot\,|x)$ is a Borel probability measure on $X$ for each $x \in X$,

\item $\gamma_{\Lambda}(A|\,\cdot)$ is a $\mathscr{F}_{\Lambda^c}$-measurable function for each $A \in \mathscr{F}$, and

\item $\gamma_{\Lambda}(B|\,\cdot) = \rchi_{B}$ for every $B \in \mathscr{F}_{\Lambda^c}$. 
\end{enumerate} 
\end{fact}
\begin{proof}
For part (a), observe that it follows immediately that $\gamma_{\Lambda}(\emptyset|x) = 0$  and $\gamma_{\Lambda}(X|x) =1$. The countable additivity
of $\gamma_{\Lambda}(\cdot\,|x)$ follows from equation (\ref{lalalalallalalala}). 

For part (b), we use the fact that $\gamma_{\Lambda}(A|\,\cdot)$ is a limit of a sequence of 
$\mathscr{F}_{\Lambda^c}$-measurable functions. 

In order to prove part (c), let us define a collection $\mathscr{C}$ of subsets of $X$ by letting $\mathscr{C} = \left\{[\zeta] : \zeta \in \mathcal{A}^{\Delta}, \Delta \in \mathscr{S} \;\text{such that}\; \Delta \subseteq \Lambda^c \right\}\cup\{\emptyset\}$.
Note that this collection is a $\pi$-system on $X$ which generates $\mathscr{F}_{\Lambda^c}$. For each point $x$ in $X$, if we let $\delta_{x} : \mathscr{F} \rightarrow \mathbb{R}$
be the Dirac measure centered on $x$, then it is easy to check that $\gamma_{\Lambda}(B|x) = \delta_{x}(B)$ for
every $B \in \mathscr{C}$. Therefore, both measures $\gamma_{\Lambda}(\cdot\,|x)$ and $\delta_{x}$ coincide on $\mathscr{F}_{\Lambda^c}$, and so the proof of part (c) is complete.
\end{proof}
In Georgii's terminology from \cite{georgii:11}, the properties presented above imply that 
$\gamma = (\gamma_{\Lambda})_{\Lambda \in \mathscr{S}}$ is a family of proper probability kernels
$\gamma_{\Lambda}$ from $(X, \mathscr{F}_{\Lambda^c})$ to $(X, \mathscr{F})$. 

\begin{fact}\label{compatibility}
The family $\gamma$ satisfies the consistency condition 
\begin{equation}\label{eetalele}
\gamma_{\Delta}\gamma_{\Lambda} = \gamma_{\Delta}
\end{equation}
whenever $\Lambda$ and $\Delta$ are elements of $\mathscr{S}$ satisfying $\Lambda \subseteq \Delta$. It means that the equation
\begin{equation}
\int_{X}\gamma_{\Delta}(dy|x)\gamma_{\Lambda}(A|y)  = \gamma_{\Delta}(A|x)
\end{equation}
holds for each set $A \in \mathscr{F}$ and each point $x \in X$.
\end{fact}

\begin{proof}
First, let us prove that given a Borel measurable function $f : X \rightarrow [0,+\infty]$, the equation
\begin{equation}\label{intgamma}
\int_{X}\gamma_{\Delta}(dy|x)f(y)  = \sum\limits_{\omega \in \mathcal{A}^{\Delta}}\gamma_{\Delta}([\omega]|x)\cdot \left(f(\omega x_{\Delta^c}) \bold{1}_{\{\omega x_{\Delta^c} \in X\}}\right)
\end{equation}
holds for every point $x$ in $X$. Indeed, the equation above is easily verified if $f$ is a characteristic function. Using 
the linearity of the integral, it is straightforward to show that the equation above also holds for simple functions. Now, in the case where
$f$ is a nonnegative extended real-valued function, the result follows by using the fact that 
there is an increasing sequence $(\varphi_{n})_{n \in \mathbb{N}}$ of nonnegative (measurable) simple functions converging pointwise to $f$, and applying
the monotone convergence theorem. 

Therefore,
\begin{eqnarray*}
\int_{X}\gamma_{\Delta}(dy|x)\gamma_{\Lambda}(A|y) = 
\end{eqnarray*}
\begin{eqnarray*}
&=& \sum\limits_{\omega \in \mathcal{A}^{\Delta}}\gamma_{\Delta}([\omega]|x) \cdot \left(\gamma_{\Lambda}(A|\omega x_{\Delta^c})\bold{1}_{\{\omega x_{\Delta^c} \in X\}}\right) \\
&=& \sum\limits_{\omega \in \mathcal{A}^{\Delta}}\lim\limits_{n \to \infty}\frac{e^{f_n(\omega x_{\Delta^c})}\bold{1}_{\{\omega x_{\Delta^c}\in X\}}}{\sum\limits_{\eta\, \in \mathcal{A}^{\Delta}}e^{f_n(\eta x_{\Delta^c})}\bold{1}_{\{\eta x_{\Delta^c}\in X\}}} \cdot \lim\limits_{n \to \infty}\frac{\sum\limits_{\omega' \in \mathcal{A}^{\Lambda}}e^{f_n(\omega'\omega_{\Delta \backslash \Lambda}x_{\Delta^c})}\bold{1}_{\{\omega'\omega_{\Delta \backslash \Lambda}x_{\Delta^c}\in A\}}}{\sum\limits_{\eta'\, \in \mathcal{A}^{\Lambda}}e^{f_n(\eta'\omega_{\Delta \backslash \Lambda}x_{\Delta^c})}\bold{1}_{\{\eta'\omega_{\Delta \backslash \Lambda}x_{\Delta^c}\in X\}}}\bold{1}_{\{\omega x_{\Delta^c} \in X\}} \\
&=& \lim\limits_{n \to \infty}\sum\limits_{\omega \in \mathcal{A}^{\Delta}}\frac{\left(e^{f_n(\omega x_{\Delta^c})}\bold{1}_{\{\omega x_{\Delta^c}\in X\}}\right)\left(\sum\limits_{\omega' \in \mathcal{A}^{\Lambda}}e^{f_n(\omega'\omega_{\Delta \backslash \Lambda}x_{\Delta^c})}\bold{1}_{\{\omega'\omega_{\Delta \backslash \Lambda}x_{\Delta^c}\in A\}}\right)}{\left(\sum\limits_{\eta\, \in \mathcal{A}^{\Delta}}e^{f_n(\eta x_{\Delta^c})}\bold{1}_{\{\eta x_{\Delta^c}\in X\}}\right)\left(\sum\limits_{\eta'\, \in \mathcal{A}^{\Lambda}}e^{f_n(\eta'\omega_{\Delta \backslash \Lambda}x_{\Delta^c})}\bold{1}_{\{\eta'\omega_{\Delta \backslash \Lambda}x_{\Delta^c}\in X\}}\right)}\,\bold{1}_{\{\omega x_{\Delta^c}\in X\}} \\
&=& \lim\limits_{n \to \infty}\sum\limits_{\omega'' \in \mathcal{A}^{\Delta \backslash \Lambda}}\sum\limits_{\omega \in \mathcal{A}^{\Lambda}}\frac{\left(e^{f_n(\omega \omega''x_{\Delta^c})}\bold{1}_{\{\omega \omega''x_{\Delta^c}\in X\}}\right) \left(\sum\limits_{\omega' \in \mathcal{A}^{\Lambda}}e^{f_n(\omega'\omega''x_{\Delta^c})}\bold{1}_{\{\omega'\omega''x_{\Delta^c}\in A\}}\right)}{\left(\sum\limits_{\eta\, \in \mathcal{A}^{\Delta}}e^{f_n(\eta x_{\Delta^c})}\bold{1}_{\{\eta x_{\Delta^c}\in X\}}\right)\left(\sum\limits_{\eta'\, \in \mathcal{A}^{\Lambda}}e^{f_n(\eta'\omega''x_{\Delta^c})}\bold{1}_{\{\eta'\omega''x_{\Delta^c}\in X\}}\right)}\,\bold{1}_{\{\omega \omega''x_{\Delta^c}\in X\}} \\
&=& \lim\limits_{n \to \infty}\sum\limits_{\omega'' \in \mathcal{A}^{\Delta \backslash \Lambda}}\frac{\left(\sum\limits_{\omega' \in \mathcal{A}^{\Lambda}}e^{f_n(\omega'\omega''x_{\Delta^c})}\bold{1}_{\{\omega'\omega''x_{\Delta^c}\in A\}}\right)}{\left(\sum\limits_{\eta\, \in \mathcal{A}^{\Delta}}e^{f_n(\eta x_{\Delta^c})}\bold{1}_{\{\eta x_{\Delta^c}\in X\}}\right)} \\
&=& \lim\limits_{n \to \infty}\frac{\sum\limits_{\omega'' \in \mathcal{A}^{\Delta \backslash \Lambda}}\;\sum\limits_{\omega' \in \mathcal{A}^{\Lambda}}e^{f_n(\omega'\omega''x_{\Delta^c})}\bold{1}_{\{\omega'\omega''x_{\Delta^c}\in A\}}}{\sum\limits_{\eta\, \in \mathcal{A}^{\Delta}}e^{f_n(\eta x_{\Delta^c})}\bold{1}_{\{\eta x_{\Delta^c}\in X\}}} \\
&=& \lim\limits_{n \to \infty}\frac{\sum\limits_{\omega \in \mathcal{A}^{\Delta}}e^{f_n(\omega x_{\Delta^c})}\bold{1}_{\{\omega x_{\Delta^c}\in A\}}}{\sum\limits_{\eta\, \in \mathcal{A}^{\Delta}}e^{f_n(\eta x_{\Delta^c})}\bold{1}_{\{\eta x_{\Delta^c}\in X\}}} \\
&=& \gamma_{\Delta}(A|x)
\end{eqnarray*}
holds for each $A \in \mathscr{F}$ and each $x \in X$.
\end{proof}

In view of \cite{georgii:11} and Facts \ref{gammaprop} and \ref{compatibility}, if we let $X = \ds$ and let $f$ be a function in $SV_{d}(X)$, then the
corresponding family $\gamma = (\gamma_{\Lambda})_{\Lambda \in \mathscr{S}}$ is a
specification with parameter set $\zd$ and state space $(\mathcal{A}, \mathcal{P}(\mathcal{A}))$. 

In the following, our goal is to 
provide a characterization of Gibbs measures on subshifts that involves conditional expectations and relate them with the family $\gamma$. 
It will be done by expressing this relationship in terms of a set of equations that is often referred to as DLR equations (e.g. \cite{muir}, \cite{sarig:09},\cite{ny:08}), named for Dobrushin, Lanford and Ruelle. 

\begin{teo}\label{DLR}
Let $X \subseteq \ds$ be a subshift, let $f$ be a function in $SV_d(X)$, and let $\gamma = (\gamma_{\Lambda})_{\Lambda \in \mathscr{S}}$ be the corresponding family given in Definition \ref{especificacao15}. 
If $\mu$ is a Gibbs measure for $f$, then $\mu$ satisfies
\begin{equation}\label{DLReq}
\mu(A|\mathscr{F}_{\Lambda^c}) = \gamma_{\Lambda}(A|\,\cdot\,) \quad \text{$\mu$-a.e.}
\end{equation}
for each $\Lambda \in \mathscr{S}$ and each $A \in \mathscr{F}$.
\end{teo}

\begin{proof}
Step 1. Let us show that for all $\Lambda \in \mathscr{S}$ and $\omega \in \mathcal{A}^{\Lambda}$, the equation
\begin{equation}
\mu\left([\omega]|{\mathscr{F}_{\Lambda^c}}\right)(x) = \gamma_{\Lambda}([\omega]|x)  
\end{equation}
holds for $\mu$-almost every point $x$ in $X$. For each $\eta$ in $\mathcal{A}^{\Lambda}$,
let us define a map $\varphi : X \rightarrow X$ by letting
\begin{equation}\label{funcaosafada}
\varphi(x) = 
\begin{cases}
\omega x_{\Lambda^c} & \text{if}\;x\in [\eta]\;\text{and}\;\omega x_{\Lambda^c}\in X, \\
\eta x_{\Lambda^c} & \text{if}\;x\in [\omega]\;\text{and}\;\eta x_{\Lambda^c}\in X, \\
x &	\text{otherwise}.
\end{cases}
\end{equation}	

\begin{claim}\label{claim1}
The function $\varphi$ defined in (\ref{funcaosafada}) is an involution.
\end{claim}

\begin{proof}\renewcommand{\qedsymbol}{\QEDB}
If we suppose that $x \in [\eta]$ and $\omega x_{\Lambda^c} \in X$, then we have $\varphi(x) = \omega x_{\Lambda^c}$. Using the fact that $\varphi(x) \in [\omega]$ and
$\eta\varphi(x)_{\Lambda^c} = \eta x_{\Lambda^c} = x \in X$, we obtain $\varphi \circ \varphi (x) = x$.
Now, if $x \in [\omega]$ and $\eta x_{\Lambda^c} \in X$, then $\varphi(x) = \eta x_{\Lambda^c}$. Since $\varphi(x) \in [\eta]$ and
$\omega\varphi(x)_{\Lambda^c} = \omega x_{\Lambda^c} = x \in X$, then it follows that $\varphi \circ \varphi (x) = x$.
Finally, if $x$ does not satisfy both conditions above, we have $\varphi \circ \varphi (x) = \varphi(\varphi(x)) = \varphi(x) = x$. 
\end{proof}

\begin{claim}\label{claim2}
The function $\varphi$ is Borel measurable.
\end{claim}
\begin{proof}\renewcommand{\qedsymbol}{\QEDB}
If we let $X_1 = [\eta] \cap \{x \in X : \omega x_{\Lambda^c} \in X\}$, $X_2 = [\omega] \cap \{x \in X : \eta x_{\Lambda^c} \in X\}$, and $X_3 = X\backslash (X_1 \cup X_2)$,
then it follows that each $X_i$ belongs to $\mathscr{F}$ and each $\varphi\restriction_{X_i}$ is a measurable function on $X_{i}$. Therefore, we conclude that
$\varphi$ is a Borel measurable function on $X$.
\end{proof}
\begin{claim}
The function $\varphi$ is a Borel automorphism of $X$, and $\mathsf{gr}(\varphi) \subseteq \gr$. 
\end{claim}
\begin{proof}
This assertion follows directly from Claims \ref{claim1} and \ref{claim2}, and from the fact that 
the equality $\varphi(x)_{\Lambda^{c}} = x_{\Lambda^{c}}$ holds for every point $x$ in $X$.
\end{proof}

For each $F \in \mathscr{F}_{\Lambda^c}$, we have
\begin{eqnarray*}
\int_{F}\rchi_{[\eta]}(x)\bold{1}_{\{\omega x_{\Lambda^c} \in X\}}\,d\mu(x) &=& \int_{F}\rchi_{[\eta]}(x)\rchi_{\{y \in X : \omega y_{\Lambda^c} \in X\}}(x)\,d\mu(x) \\ 
&=& \mu\big([\eta]\cap \{y \in X : \omega y_{\Lambda^c} \in X\} \cap F\big)\\
&=& \varphi_{\ast}\mu\big(\varphi\big([\eta]\cap \{y \in X : \omega y_{\Lambda^c} \in X\}\big) \cap \varphi(F)\big)\\
&=& \varphi_{\ast}\mu\big([\omega]\cap \{y \in X : \eta y_{\Lambda^c} \in X\} \cap \varphi(F)\big).
\end{eqnarray*}
Observe that the $\sigma$-algebra $\mathscr{F}_{\Lambda^c}$ is generated by the collection $\mathscr{C}$ given by
$\mathscr{C} = \{X \cap \pi_{i}^{-1}(A) : i \in \Lambda^c, A\subseteq \mathcal{A}\}$.
Moreover, since the collection $\{F \in \mathscr{F}_{\Lambda^c} : \varphi(F) = F\}$ is a $\sigma$-algebra of subsets of $X$ which contains
$\mathscr{C}$, then it coincides with $\mathscr{F}_{\Lambda^c}$. It follows that 

\begin{equation}
\int_{F}\rchi_{[\eta]}(x)\bold{1}_{\{\omega x_{\Lambda^c} \in X\}}\,d\mu(x) 
= \varphi_{\ast}\mu\big([\omega]\cap \{y \in X : \eta y_{\Lambda^c} \in X\} \cap F\big).
\end{equation}
In view of Proposition \ref{boraut} and the fact that $\mu$ is $(\phi_{f},\gr)$-conformal, we obtain
\begin{eqnarray*}
\int_{F}\rchi_{[\eta]}(x)\bold{1}_{\{\omega x_{\Lambda^c} \in X\}}\,d\mu(x) 
&=& \int_{[\omega]\cap \{y \in X : \eta y_{\Lambda^c} \in X\} \cap F}e^{\phi_f(x,\varphi^{-1}(x))}d\mu(x) \\
&=& \int_{[\omega]\cap \{y \in X : \eta y_{\Lambda^c} \in X\} \cap F}e^{\phi_f(x,\varphi(x))}d\mu(x)\\
&=& \int_{F}e^{\phi_f(x,\varphi(x))}\bold{1}_{\{\eta x_{\Lambda^c} \in X\}}\rchi_{[\omega]}(x)\,d\mu(x) \\
&=& \int_{F}e^{\phi_f(x,\varphi(x))}\bold{1}_{\{\eta x_{\Lambda^c} \in X\}}\bold{1}_{\{\omega x_{\Lambda^c} \in X\}}\rchi_{[\omega]}(x)\,d\mu(x) \\
&=& \int_{F}\underbrace{\left(e^{\phi_f(\omega x_{\Lambda^c},\varphi(\omega x_{\Lambda^c}))}\bold{1}_{\{\omega x_{\Lambda^c} \in X,\; \eta x_{\Lambda^c} \in X\}}\right)}_{\mathscr{F}_{\Lambda^c}\;\text{-measurable function on}\;X}\cdot\,\rchi_{[\omega]}(x)\,d\mu(x) \\
&=& \int_{F}\left(e^{\phi_f(\omega x_{\Lambda^c},\varphi(\omega x_{\Lambda^c}))}\bold{1}_{\{\omega x_{\Lambda^c} \in X,\; \eta x_{\Lambda^c} \in X\}}\right)\cdot\mu([\omega]|\mathscr{F}_{\Lambda^c})(x)\,d\mu(x) \\
&=& \int_{F}\left(e^{\phi_f(\omega x_{\Lambda^c},\eta x_{\Lambda^c})}\bold{1}_{\{\omega x_{\Lambda^c} \in X,\; \eta x_{\Lambda^c} \in X\}}\right)\cdot\mu([\omega]|\mathscr{F}_{\Lambda^c})(x)\,d\mu(x) 
\end{eqnarray*}
for every $F \in \mathscr{F}_{\Lambda^c}$. It follows that the equation
\begin{equation*}\label{dlreita}
\mu([\eta]|\mathscr{F}_{\Lambda^c})(x)\bold{1}_{\{\omega x_{\Lambda^c} \in X\}} = \left(e^{\phi_f(\omega x_{\Lambda^c},\eta x_{\Lambda^c})}\bold{1}_{\{\omega x_{\Lambda^c} \in X,\; \eta x_{\Lambda^c} \in X\}}\right)\cdot\mu([\omega]|\mathscr{F}_{\Lambda^c})(x)
\end{equation*}
holds for $\mu$-almost every $x$ in $X$.
Thus, if we sum the equation above over all elements $\eta$ of $\mathcal{A}^{\Lambda}$, we conclude that
\begin{equation}\label{dlr1}
\bold{1}_{\{\omega x_{\Lambda^c} \in X\}} = \left(\,\sum\limits_{\eta\, \in \mathcal{A}^{\Lambda}}e^{\phi_f(\omega x_{\Lambda^c},\eta x_{\Lambda^c})}\bold{1}_{\{\omega x_{\Lambda^c} \in X,\; \eta x_{\Lambda^c} \in X\}}\right)\cdot\mu([\omega]|\mathscr{F}_{\Lambda^c})(x)
\end{equation}
holds for $\mu$-almost every point $x$ in $X$. 

\begin{lemma}\label{lemadlr}
The equality
\begin{equation}\label{eqlemmadlr}
\mu([\omega]|\mathscr{F}_{\Lambda^c})(x) =\mu([\omega]|\mathscr{F}_{\Lambda^c})(x) \bold{1}_{\{\omega x_{\Lambda^c} \in X\}}
\end{equation}
holds for $\mu$-almost every $x$ in $X$.
\end{lemma}
\begin{proof}[Proof of Lemma \ref{lemadlr}]\renewcommand{\qedsymbol}{\QEDB}
Given a set $F$ in $\mathscr{F}_{\Lambda^c}$, we have
\begin{eqnarray*} 
\int_{F}\rchi_{[\omega]}(x)\,d\mu &=& \int_{F}\rchi_{[\omega]}(x)\bold{1}_{\{\omega x_{\Lambda^c} \in X\}}\,d\mu(x) \\
&=& \int_{F}\underbrace{\mu([\omega]|\mathscr{F}_{\Lambda^c})(x)\bold{1}_{\{\omega x_{\Lambda^c} \in X\}}}_{\mathscr{F}_{\Lambda^c}\text{-measurable function on}\;X}\,d\mu(x)
\end{eqnarray*}
Thus, the result follows.
\end{proof}

Now, let $N \subseteq X$ be a set of measure zero such that equations (\ref{dlr1}) and (\ref{eqlemmadlr}) holds at each point of $X \backslash N$. For every $x$ in $X \backslash N$, if $\omega x_{\Lambda^c}$ belongs to  $X$, then
\[\underbrace{\left(\sum\limits_{\eta\, \in \mathcal{A}^{\Lambda}}e^{\phi_f(\omega x_{\Lambda^c},\eta x_{\Lambda^c})}\bold{1}_{\{\eta x_{\Lambda^c} \in X\}}\right)}_{>\;0}\cdot\,\mu([\omega]|\mathscr{F}_{\Lambda^c})(x) = 1,\]
and using equation (\ref{oiadlr}), we obtain
\begin{equation}\label{baratabola15}
\mu([\omega]|\mathscr{F}_{\Lambda^c})(x) = \frac{1}{\sum\limits_{\eta\, \in \mathcal{A}^{\Lambda}}e^{\phi_f(\omega x_{\Lambda^c},\eta x_{\Lambda^c})}\bold{1}_{\{\eta x_{\Lambda^c} \in X\}}} = \lim\limits_{n \to \infty}\frac{e^{f_n(\omega x_{\Lambda^c})}\bold{1}_{\{\omega x_{\Lambda^c} \in X\}}}{\sum\limits_{\eta\, \in \mathcal{A}^{\Lambda}}e^{f_n(\eta x_{\Lambda^c})}\bold{1}_{\{\eta x_{\Lambda^c} \in X\}}}. 	
\end{equation}
Otherwise, if $\omega x_{\Lambda^c}$ does not belong to $X$, we have $\mu([\omega]|\mathscr{F}_{\Lambda^c})(x) = 0$.
Therefore, we conclude that 
\begin{equation}
\mu([\omega]|\mathscr{F}_{\Lambda^c})(x) 
= \lim\limits_{n \to \infty}\frac{e^{f_n(\omega x_{\Lambda^c})}\bold{1}_{\{\omega x_{\Lambda^c} \in X\}}}{\sum\limits_{\eta \in \mathcal{A}^{\Lambda}}e^{f_n(\eta x_{\Lambda^c})}\bold{1}_{\{\eta x_{\Lambda^c} \in X\}}} 		
= \gamma_{\Lambda}([\omega]|x)
\end{equation}
holds at each point $x$ in $X \backslash N$.

Step 2. For all $\Lambda \in \mathscr{S}$ and $A \in \mathscr{F}$,  
it is straightforward to prove that we have
\[\mu(A|\mathscr{F}_{\Lambda^c})(x) = \sum\limits_{\omega \in \mathcal{A}^{\Lambda}}\mu([\omega]|\mathscr{F}_{\Lambda^c})(x) \bold{1}_{\{\omega x_{\Lambda^c}\in A\}}\]
for $\mu$-almost every $x$ in $X$.
Then, in view of the previous step and equation (\ref{lalalalallalalala}), we conclude that
\begin{equation}
\mu(A|\mathscr{F}_{\Lambda^c}) (x) = \sum\limits_{\omega \in \mathcal{A}^{\Lambda}}\gamma_{\Lambda}([\omega]|x)\bold{1}_{\{\omega x_{\Lambda^c}\in A\}} = \gamma_{\Lambda}(A|x)
\end{equation}
holds for $\mu$-almost every point $x$ in $X$. 
\end{proof} 

On the other hand, we have the following result.

\begin{teo}\label{DLR1}
Let $X \subseteq \ds$ be a subshift, let $f$ be a function in $SV_d(X)$, and let $\gamma = (\gamma_{\Lambda})_{\Lambda \in \mathscr{S}}$ be the corresponding family given 
in Definition \ref{especificacao15}. If $\mu$ is a Borel probability measure on $X$ that 
satisfies 
\begin{equation}\label{spec}
\mu(A|\mathscr{F}_{\Lambda^c}) = \gamma_{\Lambda}(A|\,\cdot\,) \qquad \mu\text{-a.e.}
\end{equation}
for each $\Lambda \in \mathscr{S}$ and each $A \in \mathscr{F}$, then $\mu$ is a topological Gibbs measure for $f$.
\end{teo}

\begin{proof}
Step 1. Let $N$ be a positive integer, and let $\varphi$ be an arbitrary element of $\mathcal{F}_N(X)$. Let us show that for every $\omega \in \mathcal{A}^{\Lambda_N}$ we have
\begin{equation}\label{recdlr1}
\varphi_{\ast}\mu([\omega]) = \int_{[\omega]}e^{\phi_{f}(x,\varphi^{-1}(x))}\,d\mu(x).
\end{equation}
In the following, let us denote $\Lambda_N$ simply by $\Lambda$ just for convenience. Observe that $\varphi^{-1}([\omega]) = [\zeta]$ for some $\zeta \in \mathcal{A}^{\Lambda}$. 
Furthermore, it is easy to check that for every $x$ in $X$, $\omega x_{\Lambda^{c}}$ belongs to $X$ if and only if $\zeta x_{\Lambda^{c}}$ belongs to $X$.
Then 
\begin{eqnarray*}
\gamma_{\Lambda}([\zeta]|x) 
&=& \lim\limits_{n \to \infty}\frac{e^{f_n(\zeta x_{\Lambda^c})}\bold{1}_{\{\zeta x_{\Lambda^c} \in X\}}}{\sum\limits_{\eta\, \in \mathcal{A}^{\Lambda}}e^{f_n(\eta x_{\Lambda^c})}\bold{1}_{\{\eta x_{\Lambda^c} \in X\}}} \\ 	
&=& \lim\limits_{n \to \infty}\frac{e^{f_n(\zeta x_{\Lambda^c})}\bold{1}_{\{\omega x_{\Lambda^c} \in X\}}}{\sum\limits_{\eta\, \in \mathcal{A}^{\Lambda}}e^{f_n(\eta x_{\Lambda^c})}\bold{1}_{\{\eta x_{\Lambda^c} \in X\}}} \\ 	
&=& \lim\limits_{n \to \infty}\left(e^{f_n(\zeta x_{\Lambda^c})-f_n(\omega x_{\Lambda^c})}\bold{1}_{\{\omega x_{\Lambda^c} \in X\}}\right)\cdot\frac{e^{f_n(\omega x_{\Lambda^c})}\bold{1}_{\{\omega x_{\Lambda^c} \in X\}}}{\sum\limits_{\eta\, \in \mathcal{A}^{\Lambda}}e^{f_n(\eta x_{\Lambda^c})}\bold{1}_{\{\eta x_{\Lambda^c} \in X\}}} \\ 	
&=& \lim\limits_{n \to \infty}\;\exp{\left(\sum\limits_{i \in \Lambda_n}f\circ T^{i}(\zeta x_{\Lambda^c})-f\circ T^{i}(\omega x_{\Lambda^c})\right)}\bold{1}_{\{\omega x_{\Lambda^c} \in X\}} \\
&&\;\quad\quad \times \frac{e^{f_n(\omega x_{\Lambda^c})}\bold{1}_{\{\omega x_{\Lambda^c} \in X\}}}{\sum\limits_{\eta\, \in \mathcal{A}^{\Lambda}}e^{f_n(\eta x_{\Lambda^c})}\bold{1}_{\{\eta x_{\Lambda^c} \in X\}}} \\ 	
&=& \left(e^{\phi_f(\omega x_{\Lambda^c},\zeta x_{\Lambda^c})}\bold{1}_{\{\omega x_{\Lambda^c} \in X\}}\right)\cdot \gamma_{\Lambda}([\omega]|x) \\
&=& \left(e^{\phi_f(\omega x_{\Lambda^c},\varphi^{-1}(\omega x_{\Lambda^c}))}\bold{1}_{\{\omega x_{\Lambda^c} \in X\}}\right)\cdot \gamma_{\Lambda}([\omega]|x)
\end{eqnarray*}
holds at each point $x$ in $X$.

Therefore, we have
\begin{eqnarray*}
\varphi_{\ast}\mu([\omega])
&=& \int_{X}\rchi_{[\zeta]}(x)\,d\mu(x) \\
&=& \int_{X}\mu([\zeta]|\mathscr{F}_{\Lambda^c})(x)\,d\mu(x) \\
&=& \int_{X}\underbrace{\left(e^{\phi_f(\omega x_{\Lambda^c},\varphi^{-1}(\omega x_{\Lambda^c}))}\bold{1}_{\{\omega x_{\Lambda^c} \in X\}}\right)}_{\mathscr{F}_{\Lambda^c}\text{-measurable function on}\; X}\,\cdot\, \mu([\omega]|\mathscr{F}_{\Lambda^c})(x)\,d\mu(x) \\
&=& \int_{X}\left(e^{\phi_f(\omega x_{\Lambda^c},\varphi^{-1}(\omega x_{\Lambda^c}))}\bold{1}_{\{\omega x_{\Lambda^c} \in X\}}\right)\cdot\rchi_{[\omega]}(x)\,d\mu(x) \\
&=& \int_{X}e^{\phi_f(x,\varphi^{-1}(x))}\rchi_{[\omega]}(x)\,d\mu(x), 
\end{eqnarray*}
and equation (\ref{recdlr1}) follows.

Step 2. Let $\varphi$ be an arbitrary element of $\mathcal{F}(X)$. Let us consider a collection $\mathscr{C}$ of subsets of $X$ defined by
\[\mathscr{C} = \left\{[\zeta] : \zeta \in \mathcal{A}^{\Lambda}, \Lambda \in \mathscr{S}\right\}\cup\{\emptyset\}.\]
Observe that $\mathscr{C}$ is a $\pi$-system which generates the Borel $\sigma$-algebra of $X$.
If we show that 
\begin{equation}\label{recdlr2}
\varphi_{\ast}\mu(C) = \int_{C}e^{\phi_{f}(x,\varphi^{-1}(x))}\,d\mu(x)
\end{equation}
holds for every $C \in \mathscr{C}$, then the proof will be complete. 
For each $\Lambda \in \mathscr{S}$ and each $\zeta \in \mathcal{A}^{\Lambda}$,
there is a positive integer $n$ such that $\Lambda \subseteq \Lambda_n$ and $\varphi \in \mathcal{F}_n(X)$ (see Remark \ref{fnx}(\ref{fnfnfn})). Using the identity 
$[\zeta] = \bigcup\limits_{\substack{\omega \in \mathcal{A}^{\Lambda_n} \\ \omega_{\Lambda} = \zeta}}[\omega]$ and the previous step, we obtain
\begin{eqnarray*}
\varphi_{\ast}\mu([\zeta]) &=& \sum\limits_{\substack{\omega \in \mathcal{A}^{\Lambda_{n}} \\ \omega_{\Lambda} = \zeta}}\varphi_{\ast}\mu([\omega]) \\
&=&\sum\limits_{\substack{\omega \in \mathcal{A}^{\Lambda_{n}} \\ \omega_{\Lambda} = \zeta}}\int_{[\omega]}e^{\phi_{f}(x,\varphi^{-1}(x))}\,d\mu(x) \\
&=&\int_{[\zeta]}e^{\phi_{f}(x,\varphi^{-1}(x))}\,d\mu(x).
\end{eqnarray*}
\end{proof}

The next result follows immediately from Theorems \ref{DLR} and \ref{DLR1}, and provide us a characterization for
Gibbs measures on subshifts on finite type in terms of the DLR equations.	

\begin{cor}\label{DLR2}
Let $X \subseteq \ds$ be a subshift of finite type, let $f$ be a function in $SV_d(X)$, and let  $\gamma = (\gamma_{\Lambda})_{\Lambda \in \mathscr{S}}$ be the corresponding family given in Definition \ref{especificacao15}. Then, $\mu$ is a Gibbs measure for $f$ if
and only if $\mu$ is a Borel probability measure on $X$ that satisfies
\begin{equation}\label{spec}
\mu(A|\mathscr{F}_{\Lambda^c}) = \gamma_{\Lambda}(A|\,\cdot\,) \qquad \mu\text{-a.e.}
\end{equation}
for each $\Lambda \in \mathscr{S}$ and each $A \in \mathscr{F}$.
\end{cor}

\section{Gibbs measures in statistical mechanics}

Recall that the set of all nonempty finite subsets of $\zd$ is denoted by $\mathscr{S}$. 
Let $\mathscr{S}_{0}$ be an infinite subset of $\mathscr{S}$, and let $\Psi : \mathscr{S}_{0} \rightarrow \mathbb{R}$ be an arbitrary function. 
We will say that the infinite sum 
\[\sum\limits_{\Lambda \in \mathscr{S}_{0}} \Psi_{\Lambda}\]
exists and is equal to a real number $s$, if the net 
\[\left(\sum\limits_{\substack{\Lambda \in \mathscr{S}_{0} \\ \Lambda \subseteq \Delta}} \Psi_{\Lambda}\right)_{\Delta \in \mathscr{S}}\]
converges to $s$. In this case, we write
\[\sum\limits_{\Lambda \in \mathscr{S}_{0}} \Psi_{\Lambda} = s.\]
	
\begin{remark}\label{posisummmmm}
If $\Psi_{\Lambda} \geq 0$ for each $\Lambda$ in $\mathscr{S}_{0}$, then the sum $\sum\limits_{\Lambda \in \mathscr{S}_{0}} \Psi_{\Lambda}$ converges if and only if $\sup\left\{\sum\limits_{\substack{\Lambda \in \mathscr{S}_{0} \\ \Lambda \subseteq \Delta}}\Psi_{\Lambda} : \Delta \in \mathscr{S}\right\}$ is finite. 	
In either case we have $\sum\limits_{\Lambda \in \mathscr{S}_{0}} \Psi_{\Lambda} = \sup\left\{\sum\limits_{\substack{\Lambda \in \mathscr{S}_{0} \\ \Lambda \subseteq \Delta}}\Psi_{\Lambda} : \Delta \in \mathscr{S}\right\}$.
\end{remark}

If $(\Phi_{\Lambda})_{\Lambda \in \mathscr{S}}$ is a family of real-valued functions defined on $X$ such that $\sum\limits_{\Lambda \in \mathscr{S}_{0}} \Phi_{\Lambda}(x)$ exists for each
$x$ in $X$, then we will denote by $\sum\limits_{\Lambda \in \mathscr{S}_{0}} \Phi_{\Lambda}$ the function which associates to each point $x$ in $X$ the sum $\sum\limits_{\Lambda \in \mathscr{S}_{0}} \Phi_{\Lambda}(x)$.

\begin{lemma}\label{lemmamalucooo}
Let $(\Phi_{\Lambda})_{\Lambda \in \mathscr{S}}$ be a family of real-valued bounded functions defined on $X$, and let $\mathscr{S}_{0}$ be an infinite subset of $\mathscr{S}$. Suppose that $\sum\limits_{\Lambda \in \mathscr{S}_{0}} \|\Phi_{\Lambda}\|_{\infty}$ converges. 
\begin{enumerate}[label=(\alph*),ref=\alph*]
\item The net $\left(\sum\limits_{\substack{\Lambda \in \mathscr{S}_{0} \\ \Lambda \subseteq \Delta}} \Phi_{\Lambda}\right)_{\Delta \in \mathscr{S}}$ converges uniformly to $\sum\limits_{\Lambda \in \mathscr{S}_{0}} \Phi_{\Lambda}$, and

 
\item\label{naointeressa} if $(c_{n})_{n \in \mathbb{N}}$ is a sequence of functions $c_{n} : \mathscr{S}_{0} \rightarrow \mathbb{R}$ which converges pointwise to a function
$c : \mathscr{S}_{0} \rightarrow \mathbb{R}$ and  
\[C \defeq \sup\limits_{n \in \mathbb{N}}\sup\limits_{\Lambda \in \mathscr{S}_{0}}|c_{n}(\Lambda)| < + \infty,\] 
then 
\[\lim\limits_{n \to \infty} \left\|\sum\limits_{\Lambda \in \mathscr{S}_{0}} c_{n}(\Lambda) \Phi_{\Lambda} - \sum\limits_{\Lambda \in \mathscr{S}_{0}} c(\Lambda) \Phi_{\Lambda} \right\|_{\infty}.\]
\end{enumerate}	
\end{lemma}
\begin{proof}
For each positive number $\epsilon$ there is a set $\Delta_{0} \in \mathscr{S}$ such that 
\begin{equation}
\sum\limits_{\substack{\Lambda \in \mathscr{S}_{0},\, \Lambda \cap \Delta_{0}^{c} \neq \emptyset \\ \Lambda \subseteq \Delta}}\|\Phi_{\Lambda}\|_{\infty} = \sum\limits_{\substack{\Lambda \in \mathscr{S}_{0} \\ \Lambda \subseteq \Delta}}\|\Phi_{\Lambda}\|_{\infty}
- \sum\limits_{\substack{\Lambda \in \mathscr{S}_{0} \\ \Lambda \subseteq \Delta_{0}}}\|\Phi_{\Lambda}\|_{\infty} < \frac{\epsilon}{2}	
\end{equation}	
holds whenever $\Delta$ belongs to $\mathscr{S}$ and satisfies $\Delta_{0} \subseteq \Delta$. It follows that for every $\Delta$ and $\Delta'$ in $\mathscr{S}$ such that
$\Delta_{0} \subseteq \Delta$ and $\Delta_{0} \subseteq \Delta'$, we have
\begin{eqnarray*}
\left\|\sum\limits_{\substack{\Lambda \in \mathscr{S}_{0} \\ \Lambda \subseteq \Delta}}\Phi_{\Lambda}- \sum\limits_{\substack{\Lambda \in \mathscr{S}_{0} \\ \Lambda \subseteq \Delta'}}\Phi_{\Lambda}\right\|_{\infty} 	
&=& \left\|\sum\limits_{\substack{\Lambda \in \mathscr{S}_{0},\,\Lambda \cap \Delta_{0}^{c} \neq \emptyset \\ \Lambda \subseteq \Delta}}\Phi_{\Lambda} - \sum\limits_{\substack{\Lambda \in \mathscr{S}_{0},\,\Lambda \cap \Delta_{0}^{c} \neq \emptyset \\ \Lambda \subseteq \Delta'}}\Phi_{\Lambda}\right\|_{\infty}\\
&\leq& \sum\limits_{\substack{\Lambda \in \mathscr{S}_{0},\,\Lambda \cap \Delta_{0}^{c} \neq \emptyset \\ \Lambda \subseteq \Delta}}\|\Phi_{\Lambda}\|_{\infty} + \sum\limits_{\substack{\Lambda \in \mathscr{S}_{0},\,\Lambda \cap \Delta_{0}^{c} \neq \emptyset \\ \Lambda \subseteq \Delta'}}\|\Phi_{\Lambda}\|_{\infty} \\
&<& \epsilon.
\end{eqnarray*}
We conclude that $\left(\sum\limits_{\substack{\Lambda \in \mathscr{S}_{0}\\ \Lambda \subseteq \Delta}}\Phi_{\Lambda}\right)_{\Delta \in \mathscr{S}}$ is a Cauchy net on the space of all real-valued bounded functions 
on $X$, thus part (a) follows.

For part (b), observe that since $|c_{n}(\Lambda)| \leq C$ holds for each $\Lambda$ and each $n$, it follows that $|c(\Lambda)| \leq C$ holds for each $\Lambda$. Thus
$\sum\limits_{\Lambda \in \mathscr{S}_{0}} \|c(\Lambda)\Phi_{\Lambda}\|_{\infty}$ converges, as well as each
sum $\sum\limits_{\Lambda \in \mathscr{S}_{0}} \|c_{n}(\Lambda)\Phi_{\Lambda}\|_{\infty}$. 

For each positive integer $n$ and every $\Delta$ and $\Delta_{0}$ in $\mathscr{S}$ such that $\Delta_{0}\subseteq \Delta$, we have
\begin{eqnarray*}
\Bigg\|\sum\limits_{\Lambda \in \mathscr{S}_{0}} c_{n}(\Lambda)\Phi_{\Lambda}&-&\sum\limits_{\Lambda \in \mathscr{S}_{0}}c(\Lambda)\Phi_{\Lambda}\Bigg\|_{\infty} \\
&\leq& \left\|\sum\limits_{\Lambda \in \mathscr{S}_{0}}c_{n}(\Lambda)\Phi_{\Lambda} - \sum\limits_{\substack{\Lambda \in \mathscr{S}_{0} \\ \Lambda \subseteq \Delta}}c_{n}(\Lambda)\Phi_{\Lambda}\right\|_{\infty} + \left\|\sum\limits_{\substack{\Lambda \in \mathscr{S}_{0} \\ \Lambda \subseteq \Delta}}c_{n}(\Lambda)\Phi_{\Lambda} - \sum\limits_{\substack{\Lambda \in \mathscr{S}_{0}\\ \Lambda \subseteq \Delta_{0}}}c_{n}(\Lambda)\Phi_{\Lambda}\right\|_{\infty} \\	
&&+ \left\|\sum\limits_{\substack{\Lambda \in \mathscr{S}_{0} \\ \Lambda \subseteq \Delta_{0}}}c_{n}(\Lambda)\Phi_{\Lambda} - \sum\limits_{\substack{\Lambda \in \mathscr{S}_{0} \\ \Lambda \subseteq \Delta_{0}}}c(\Lambda)\Phi_{\Lambda}\right\|_{\infty} + \left\|\sum\limits_{\substack{\Lambda \in \mathscr{S}_{0} \\ \Lambda \subseteq \Delta_{0}}}c(\Lambda)\Phi_{\Lambda} - \sum\limits_{\Lambda \in \mathscr{S}_{0}}c(\Lambda)\Phi_{\Lambda}\right\|_{\infty} \\
&\leq& \left\|\sum\limits_{\Lambda \in \mathscr{S}_{0}}c_{n}(\Lambda)\Phi_{\Lambda} - \sum\limits_{\substack{\Lambda \in \mathscr{S}_{0} \\ \Lambda \subseteq \Delta}}c_{n}(\Lambda)\Phi_{\Lambda}\right\|_{\infty} + \left\|\sum\limits_{\substack{\Lambda \in \mathscr{S}_{0},\, \Lambda \cap \Delta_{0}^{c} \neq \emptyset \\ \Lambda \subseteq \Delta}}c_{n}(\Lambda) \Phi_{\Lambda}\right\|_{\infty} \\	
&&+ \left\|\sum\limits_{\substack{\Lambda \in \mathscr{S}_{0} \\ \Lambda \subseteq \Delta_{0}}}(c_{n}(\Lambda) - c(\Lambda))\Phi_{\Lambda}\right\|_{\infty} + \left\|\sum\limits_{\substack{\Lambda \in \mathscr{S}_{0} \\ \Lambda \subseteq \Delta_{0}}}c(\Lambda)\Phi_{\Lambda} - \sum\limits_{\Lambda \in \mathscr{S}_{0}}c(\Lambda)\Phi_{\Lambda}\right\|_{\infty} \\
&\leq& \left\|\sum\limits_{\Lambda \in \mathscr{S}_{0}}c_{n}(\Lambda)\Phi_{\Lambda} - \sum\limits_{\substack{\Lambda \in \mathscr{S}_{0} \\ \Lambda \subseteq \Delta}}c_{n}(\Lambda)\Phi_{\Lambda}\right\|_{\infty} + C \sum\limits_{\substack{\Lambda \in \mathscr{S}_{0},\, \Lambda \cap \Delta_{0}^{c} \neq \emptyset \\ \Lambda \subseteq \Delta}}\|\Phi_{\Lambda}\|_{\infty} \\	
&&+ \max\limits_{\substack{\Lambda \in \mathscr{S}_{0} \\ \Lambda \subseteq \Delta_{0}}}|c_{n}(\Lambda) - c(\Lambda)| \cdot \sum\limits_{\substack{\Lambda \in \mathscr{S}_{0} \\ \Lambda \subseteq \Delta_{0}}} \|\Phi_{\Lambda}\|_{\infty} + \left\|\sum\limits_{\substack{\Lambda \in \mathscr{S}_{0} \\ \Lambda \subseteq \Delta_{0}}}c(\Lambda)\Phi_{\Lambda} - \sum\limits_{\Lambda \in \mathscr{S}_{0}}c(\Lambda)\Phi_{\Lambda}\right\|_{\infty}. 
\end{eqnarray*}
According to part (a), for each positive number $\epsilon$ there is an element $\Delta_{0}$ of $\mathscr{S}$ such that 
\[\left\|\sum\limits_{\substack{\Lambda \in \mathscr{S}_{0} \\ \Lambda \subseteq \Delta_{0}}}c(\Lambda)\Phi_{\Lambda} - \sum\limits_{\Lambda \in \mathscr{S}_{0}}c(\Lambda)\Phi_{\Lambda}\right\|_{\infty} < \frac{\epsilon}{4}\]
and
\[C \sum\limits_{\substack{\Lambda \in \mathscr{S}_{0},\, \Lambda \cap \Delta_{0}^{c} \neq \emptyset \\ \Lambda \subseteq \Delta}}\|\Phi_{\Lambda}\|_{\infty} < \frac{\epsilon}{4}\]
holds for each $\Delta$ in $\mathscr{S}$ satisfying $\Delta_{0} \subseteq \Delta$. And also, we can find a positive integer $n_{0}$ such that 
\[\max\limits_{\substack{\Lambda \in \mathscr{S}_{0} \\ \Lambda \subseteq \Delta_{0}}}|c_{n}(\Lambda) - c(\Lambda)| \cdot \sum\limits_{\substack{\Lambda \in \mathscr{S}_{0} \\ \Lambda \subseteq \Delta_{0}}} \|\Phi_{\Lambda}\|_{\infty} < \frac{\epsilon}{4}\]
holds whenever $n \geq n_{0}$.

We conclude that for every $n \geq n_{0}$, if we let $\Delta$ be an element of $\mathscr{S}$ such that $\Delta_{0} \subseteq \Delta$ and
\[\left\|\sum\limits_{\Lambda \in \mathscr{S}_{0}}c_{n}(\Lambda)\Phi_{\Lambda} - \sum\limits_{\substack{\Lambda \in \mathscr{S}_{0} \\ \Lambda \subseteq \Delta}}c_{n}(\Lambda)\Phi_{\Lambda}\right\|_{\infty} < \frac{\epsilon}{4},\]
we have
\[\left\|\sum\limits_{\Lambda \in \mathscr{S}_{0}} c_{n}(\Lambda)\Phi_{\Lambda}-\sum\limits_{\Lambda \in \mathscr{S}_{0}}c(A)\Phi_{\Lambda}\right\|_{\infty} < \epsilon.\]
\end{proof}

\begin{mydef}
An interaction potential is a family $\Phi = (\Phi_{\Lambda})_{\Lambda \in \mathscr{S}}$ of functions $\Phi_{\Lambda} : X \rightarrow \mathbb{R}$ such that
\begin{enumerate}[label=(\alph*),ref=\alph*]
\item for each $\Lambda \in \mathscr{S}$, the function $\Phi_{\Lambda}$ is $\mathscr{F}_{\Lambda}$-measurable, and	

\item for all $\Lambda \in \mathscr{S}$ and $x \in X$, the sum
\begin{equation}
H^{\Phi}_{\Lambda}(x) \defeq \sum\limits_{\Delta \in \mathscr{S},\,\Delta \cap \Lambda \neq \emptyset}\Phi_{\Delta}(x)	
\end{equation}
converges.
\end{enumerate}
The quantity $H^{\Phi}_{\Lambda}(x)$ is called the energy of $x$ in $\Lambda$ for the interaction potential $\Phi$, and  the Hamiltonian in $\Lambda$ for $\Phi$ is the function 
$H^{\Phi}_{\Lambda}$ which associates to each $x$ in $X$ the energy $H^{\Phi}_{\Lambda}(x)$. 
\end{mydef}
\begin{remark}\label{depdepdep}
Given an arbitrary subset $\Lambda$ of $\zd$ and a $\mathscr{F}_{\Lambda}$-measurable function $f :X \rightarrow \mathbb{R}$,
the equality $f(x) =  f(y)$ holds whenever $x$ and $y$ are elements of $X$ such that $x_{\Lambda} = y_{\Lambda}$. The reader can easily verify that 
it suffices to prove this result for characteristic functions. Observe that 
\[\{B \subseteq X : \text{$\rchi_{B}(x) = \rchi_{B}(y)$ holds whenever $x_{\Lambda} = y_{\Lambda}$}\}\]
is a $\sigma$-algebra of subsets of $X$ which contains the collection
\[\mathscr{C} = \{X \cap \pi_{i}^{-1}(C) : i \in \Lambda, A \subseteq \mathcal{A}\}.\]
Since $\mathscr{F}_{\Lambda}$ is generated by $\mathscr{C}$, the result follows.
\end{remark}

\begin{ex}\label{interactionising}
Let $X$ be the full shift $\{-1,+1\}^{\zd}$. Given two parameters $J$ and $h$ in $\mathbb{R}$, let us consider
\begin{equation}
\Phi^{J,h}_{\Lambda}(x) = 
\begin{cases}
-J x_{i}x_{j} &\text{if}\; \Lambda=\{i,j\}\; \text{and}\; i \sim j,\\
-h x_{i} &\text{if}\; \Lambda=\{i\},\\
0 &\text{otherwise.}	
\end{cases}	
\end{equation}	
The equation above defines an interaction potential $\Phi^{J,h} = (\Phi^{J,h}_{\Lambda})_{\Lambda \in \mathscr{S}}$ is called the Ising potential  
with coupling constant $J$ and external field $h$.
\end{ex}

\begin{mydef}
An interaction potential $\Phi = (\Phi_{\Lambda})_{\Lambda \in \mathscr{S}}$ is said to be 
\begin{enumerate}[label=(\alph*),ref=\alph*]
\item translation invariant if the relation
\begin{equation}
\Phi_{\Lambda} \circ T^{i} = \Phi_{\Lambda + i}	
\end{equation}
holds for each $\Lambda \in \mathscr{S}$ and each $i \in \zd$, and 

\item absolutely summable if each $\Phi_{\Lambda}$ is bounded and satisfies 
\begin{equation}
\sum\limits_{\Lambda \in \mathscr{S},\, i \in \Lambda} \|\Phi_{\Lambda}\|_{\infty} < +\infty
\end{equation}
for each $i \in \zd$.
\end{enumerate}
\end{mydef}

\begin{remark}
\begin{enumerate}[label=(\alph*),ref=\alph*]
\item Observe that if $\Phi$ is absolutely summable, then the sum
$\sum\limits_{\Delta \in \mathscr{S},\,\Delta \cap \Lambda \neq \emptyset}\|\Phi_{\Delta}\|_{\infty}$
converges for each $\Lambda$ in $\mathscr{S}$. In fact, we have
\[\sum\limits_{\substack{\Delta \in \mathscr{S},\,\Delta \cap \Lambda \neq \emptyset \\ \Delta \subseteq \Delta'}}\|\Phi_{\Delta}\|_{\infty} \leq 
\sum_{i \in \Lambda}\sum\limits_{\substack{\Delta \in \mathscr{S},\, i \in \Delta \\ \Delta \subseteq \Delta'}}\|\Phi_{\Delta}\|_{\infty} 
\leq \sum_{i \in \Lambda}\sum\limits_{\Delta \in \mathscr{S},\, i \in \Delta}\|\Phi_{\Delta}\|_{\infty}\]
for each $\Delta'$ in $\mathscr{S}$. Thus, our assertion follows from Remark \ref{posisummmmm}.

\item One can easily verify that the potential $\Phi^{J,h}$ given in Example \ref{interactionising} is translation invariant and absolutely summable.
\end{enumerate}
\end{remark}

In the following, given an absolutely summable potential $\Phi$, we will let $A_{\Phi}$ be a real-valued function defined on $X$ given by
\begin{equation}
A_{\Phi}(x) = -\sum\limits_{\Lambda \in \mathscr{S},\,\mathbf{0} \in \Lambda}\frac{1}{|\Lambda|}\,\Phi_{\Lambda}.	
\end{equation}
Observe that $A_{\Phi}$ is well defined since $\sum\limits_{\Lambda \in \mathscr{S},\,\mathbf{0} \in \Lambda}\frac{1}{|\Lambda|}\|\Phi_{\Lambda}\|_{\infty} \leq
\sum\limits_{\Lambda \in \mathscr{S},\,\mathbf{0} \in \Lambda}\|\Phi_{\Lambda}\|_{\infty} < +\infty$.

\begin{ex}
Let $\Phi^{J,h}$ be the Ising potential defined in Example \ref{interactionising}. Then, the function $A_{\Phi^{J,h}}$ is given by
\begin{equation}
A_{\Phi^{J,h}}(x) = \frac{J}{2} \sum\limits_{j\, \sim\, \bf{0}}x_{\bf{0}}x_j  + h x_{\bf{0}}	
\end{equation}	
for each $x$ in $\{-1,+1\}^{\zd}$. Observe that $A_{\Phi^{J,h}}$ coincides with $f^{J,h}$ (see Example \ref{ising2}).
\end{ex}

\begin{teo}
Let $X \subseteq \mathcal{A}^{\zd}$ be a subshift, and let $\Phi$ be a translation invariant and absolutely summable potential. If we suppose that
the function $f = A_{\Phi}$ belongs to $SV_{d}(X)$, then the corresponding family $\gamma = (\gamma_{\Lambda})_{\Lambda \in \mathscr{S}}$ defined in Definition \ref{especificacao15}
is given by
\begin{equation}
\gamma_{\Lambda}(A|x) = \frac{1}{Z^{\Phi}_{\Lambda}(x)} \int_{\mathcal{A}^{\Lambda}}e^{-H^{\Phi}_{\Lambda}(\zeta x_{\Lambda^{c}})}\bold{1}_{\{\zeta x_{\Lambda^{c}} \in A\}} \lambda^{\Lambda}(d\zeta)	
\end{equation} 	
where $\lambda$ is the uniform measure on $(\mathcal{A},\mathcal{P}(\mathcal{A}))$, and 
\begin{equation}
Z^{\Phi}_{\Lambda}(x) = \int_{\mathcal{A}^{\Lambda}}e^{-H^{\Phi}_{\Lambda}(\zeta x_{\Lambda^{c}})}\bold{1}_{\{\zeta x_{\Lambda^{c}} \in X\}} \lambda^{\Lambda}(d\zeta).		
\end{equation}
\end{teo}
\begin{proof}
Let $\Lambda$ be an element of $\mathscr{S}$, let $A$ be a Borel subset of $X$, and let $x$ be a point in $X$. For each positive integer $n$, we have
\begin{eqnarray*}
f_{n} &=& \sum\limits_{i \in \Lambda_{n}}f \circ T^{i}(x) \\ &=& -\sum\limits_{i \in \Lambda_{n}}\sum\limits_{\Delta \in \mathscr{S},\,\mathbf{0} \in \Delta} \frac{1}{|\Delta|}\,\Phi_{\Delta}\circ T^{i}(x)\\
&=& -\sum\limits_{i \in \Lambda_{n}}\sum\limits_{\Delta \in \mathscr{S},\,\mathbf{0} \in \Delta} \frac{1}{|\Delta + i|}\,\Phi_{\Delta + i}(x) \\
&=& -\sum\limits_{i \in \Lambda_{n}}\sum\limits_{\Delta \in \mathscr{S},\,i \in \Delta} \frac{1}{|\Delta|}\,\Phi_{\Delta}(x) \\
&=& -\sum\limits_{i \in \Lambda_{n}}\sum\limits_{\Delta \in \mathscr{S}} \frac{1}{|\Delta|}\,\Phi_{\Delta}(x) \bold{1}_{\{i \in \Delta\}} \\
&=& -\sum\limits_{\Delta \in \mathscr{S}} \frac{|\Delta \cap \Lambda_{n}|}{|\Delta|}\,\Phi_{\Delta}(x) \\
&=& -\sum\limits_{\Delta \in \mathscr{S},\,\Delta \cap \Lambda \neq \emptyset} \frac{|\Delta \cap \Lambda_{n}|}{|\Delta|}\,\Phi_{\Delta}(x) \; 
-\sum\limits_{\Delta \in \mathscr{S},\,\Delta \cap \Lambda = \emptyset} \frac{|\Delta \cap \Lambda_{n}|}{|\Delta|}\,\Phi_{\Delta}(x).
\end{eqnarray*}	
Using Remark \ref{depdepdep}, we obtain
\begin{eqnarray*}
&&\frac{\sum\limits_{\omega \in \mathcal{A}^{\Lambda}}e^{f_n(\omega x_{\Lambda^c})}\bold{1}_{\{\omega x_{\Lambda^c}\in A\}}}{\sum\limits_{\eta\, \in \mathcal{A}^{\Lambda}}e^{f_n(\eta x_{\Lambda^c})}\bold{1}_{\{\eta x_{\Lambda^c}\in X\}}} \\
&& \qquad\quad= \frac{\sum\limits_{\omega \in \mathcal{A}^{\Lambda}}\exp{\left(-\sum\limits_{\Delta \in \mathscr{S},\,\Delta \cap \Lambda \neq \emptyset} \frac{|\Delta \cap \Lambda_{n}|}{|\Delta|}\,\Phi_{\Delta}(\omega x_{\Lambda^c}) - 
\sum\limits_{\Delta \in \mathscr{S},\,\Delta \cap \Lambda = \emptyset} \frac{|\Delta \cap \Lambda_{n}|}{|\Delta|}\,\Phi_{\Delta}(\omega x_{\Lambda^c})\right)}\bold{1}_{\{\omega x_{\Lambda^c}\in A\}}}
{\sum\limits_{\eta\, \in \mathcal{A}^{\Lambda}}\exp{\left(-\sum\limits_{\Delta \in \mathscr{S},\,\Delta \cap \Lambda \neq \emptyset} \frac{|\Delta \cap \Lambda_{n}|}{|\Delta|}\,\Phi_{\Delta}(\eta x_{\Lambda^c}) -	 
\sum\limits_{\Delta \in \mathscr{S},\,\Delta \cap \Lambda = \emptyset} \frac{|\Delta \cap \Lambda_{n}|}{|\Delta|}\,\Phi_{\Delta}(\eta x_{\Lambda^c})\right)}\bold{1}_{\{\eta x_{\Lambda^c}\in X\}}}\\
&& \qquad\quad= \frac{\sum\limits_{\omega \in \mathcal{A}^{\Lambda}}\exp{\left(-\sum\limits_{\Delta \in \mathscr{S},\,\Delta \cap \Lambda \neq \emptyset} \frac{|\Delta \cap \Lambda_{n}|}{|\Delta|}\,\Phi_{\Delta}(\omega x_{\Lambda^c}) - 
\sum\limits_{\Delta \in \mathscr{S},\,\Delta \cap \Lambda = \emptyset} \frac{|\Delta \cap \Lambda_{n}|}{|\Delta|}\,\Phi_{\Delta}(x)\right)}\bold{1}_{\{\omega x_{\Lambda^c}\in A\}}}
{\sum\limits_{\eta\, \in \mathcal{A}^{\Lambda}}\exp{\left(-\sum\limits_{\Delta \in \mathscr{S},\,\Delta \cap \Lambda \neq \emptyset} \frac{|\Delta \cap \Lambda_{n}|}{|\Delta|}\,\Phi_{\Delta}(\eta x_{\Lambda^c}) -	 
\sum\limits_{\Delta \in \mathscr{S},\,\Delta \cap \Lambda = \emptyset} \frac{|\Delta \cap \Lambda_{n}|}{|\Delta|}\,\Phi_{\Delta}(x)\right)}\bold{1}_{\{\eta x_{\Lambda^c}\in X\}}}\\
&& \qquad\quad= \frac{\sum\limits_{\omega \in \mathcal{A}^{\Lambda}}\exp{\left(-\sum\limits_{\Delta \in \mathscr{S},\,\Delta \cap \Lambda \neq \emptyset} \frac{|\Delta \cap \Lambda_{n}|}{|\Delta|}\,\Phi_{\Delta}(\omega x_{\Lambda^c})\right)}\bold{1}_{\{\omega x_{\Lambda^c}\in A\}}}
{\sum\limits_{\eta\, \in \mathcal{A}^{\Lambda}}\exp{\left(-\sum\limits_{\Delta \in \mathscr{S},\,\Delta \cap \Lambda \neq \emptyset} \frac{|\Delta \cap \Lambda_{n}|}{|\Delta|}\,\Phi_{\Delta}(\eta x_{\Lambda^c})\right)}\bold{1}_{\{\eta x_{\Lambda^c}\in X\}}}, 	 	
\end{eqnarray*} 
and according to Lemma \ref{lemmamalucooo}(\ref{naointeressa}), we conclude that
\begin{eqnarray*}
\gamma_{\Lambda}(A|\,x) &=& \lim\limits_{n \to \infty}\frac{\sum\limits_{\omega \in \mathcal{A}^{\Lambda}}\exp{\left(-\sum\limits_{\Delta \in \mathscr{S},\,\Delta \cap \Lambda \neq \emptyset} \frac{|\Delta \cap \Lambda_{n}|}{|\Delta|}\,\Phi_{\Delta}(\omega x_{\Lambda^c})\right)}\bold{1}_{\{\omega x_{\Lambda^c}\in A\}}}
{\sum\limits_{\eta\, \in \mathcal{A}^{\Lambda}}\exp{\left(-\sum\limits_{\Delta \in \mathscr{S},\,\Delta \cap \Lambda \neq \emptyset} \frac{|\Delta \cap \Lambda_{n}|}{|\Delta|}\,\Phi_{\Delta}(\eta x_{\Lambda^c})\right)}\bold{1}_{\{\eta x_{\Lambda^c}\in X\}}} \\
&=& \frac{\sum\limits_{\omega \in \mathcal{A}^{\Lambda}}e^{-H^{\Phi}_{\Lambda}(\omega x_{\Lambda^{c}})}\bold{1}_{\{\omega x_{\Lambda^c}\in A\}}}{\sum\limits_{\eta \in \mathcal{A}^{\Lambda}}e^{-H^{\Phi}_{\Lambda}(\eta x_{\Lambda^{c}})}\bold{1}_{\{\eta x_{\Lambda^c}\in X\}}} \\
&=& \frac{\sum\limits_{\omega \in \mathcal{A}^{\Lambda}}e^{-H^{\Phi}_{\Lambda}(\omega x_{\Lambda^{c}})}\bold{1}_{\{\omega x_{\Lambda^c}\in A\}} \lambda^{\Lambda}(\{\omega\})}{\sum\limits_{\eta \in \mathcal{A}^{\Lambda}}e^{-H^{\Phi}_{\Lambda}(\eta x_{\Lambda^{c}})}\bold{1}_{\{\eta x_{\Lambda^c}\in X\}}\lambda^{\Lambda}(\{\eta\})}.
\end{eqnarray*}
\end{proof}

%% file: ape-conjuntos.tex




\chapter{Unordered Sums}\label{unsum}
\label{ape}

\section{Nets}

In order to study unordered sums, we need to introduce the idea of a net, also called a Moore-Smith sequence. Let us start by presenting the definition
of a directed set.

\begin{mydef}
A directed set is a set $S$ together with a preorder relation $\preceq$ such that any two elements have an upper bound. In other words, $\preceq$ is a binary relation 
on $S$ such that
\begin{itemize}
\item[(i)] $x \preceq x$ holds for each $x$ in $S$,

\item[(ii)] if $x, y,$ and $z$ belong $S$ and the conditions $x \preceq y$ and $y \preceq z$ are satisfied, then $x \preceq z$, and

\item[(iii)] for each $x$ and $y$ in $S$ there is an element $z$ of $S$ such that $x \preceq z$ and $y \preceq z$.
\end{itemize}
\end{mydef}

The following example will be of great importance in the next section.

\begin{ex}
Let $A$ be an arbitrary set and let $\mathscr{S}_A \defeq \{I \subseteq A: I\;\text{is a finite set}\}$. It is straightforward to check
that $\mathscr{S}_{A}$ is directed by inclusion.
\end{ex}

\begin{mydef}[Net]
A net in a topological space $X$ is a function $f$ from a directed set $S$ into $X$. If $f(\lambda) = x_\lambda$ for each $\lambda \in S$, then we will simply write
$(x_\lambda)_{\lambda \in S}$ instead of $f$.
\end{mydef}

\begin{mydef}[Convergence of a net]
Let $(x_\lambda)_{\lambda \in S}$ be a net in a topological space $X$. We will say that $(x_\lambda)_{\lambda \in S}$ converges to 
a point $x$ in $X$ if for each neighborhood $U$ of $x$ there is an element $\lambda_{0}$ of $S$ such that $x_{\lambda} \in U$ whenever $\lambda$ satisfies $\lambda_{0} \preceq \lambda$. 
\end{mydef}

\section{Unordered Sums}

In this section we introduce the concept of an unordered sum. If the reader is interested in the study of this subject, see \cite{hunter:07}.
The results presented in the following are used on Chapter \ref{cap:gibbs} in order to give a precise 
definition of a Gibbs measure. 

\begin{mydef}
Let $(x_\lambda )_{\lambda  \in A}$ be an arbitrary family of elements of a normed space $X$. We will say that the unordered sum 
$\sum\limits_{\lambda  \in A}x_\lambda $ converges to a point $x$ in $X$ if the net $\Big(\sum\limits_{\lambda  \in I}x_\lambda \Big)_{I \in \mathscr{S}_A}$ converges to
$x$.
\end{mydef} 

The next proposition follows immediately from the definition given above.

\begin{prop}
The unordered sum $\sum\limits_{\lambda \in A}x_\lambda$ in the normed space $(X,\|\cdot\|)$ converges to a point $x$ if and only if for each positive number $\epsilon$ there is a finite subset
$I_{0}$ of $A$ such that $\left\|\sum\limits_{\lambda \in I}x_\lambda - x\right\| < \epsilon$ holds whenever $I$ is a finite subset of $A$ satisfying $I_{0} \subseteq I$.
\end{prop}

\begin{proof}
The proof is straightforward.	
\end{proof}

In case we are dealing with unordered sums of nonnegative real numbers we have the following result.

\begin{cor}
Let $(x_{\lambda})_{\lambda \in A}$ be an arbitrary family of nonnegative real numbers. Then, the unordered sum $\sum\limits_{\lambda \in A} x_{\lambda}$ converges
if and only if $\sup\left\{\sum\limits_{\lambda \in I}x_{\lambda} : I \in \mathscr{S}_{A}\right\}$ is a finite number. In either case, we have 
\begin{equation}
\sum\limits_{\lambda \in A} x_{\lambda} = \sup\left\{\sum\limits_{\lambda \in I}x_{\lambda} : I \in \mathscr{S}_{A}\right\}.
\end{equation} 
\end{cor}

\begin{proof}
In the case where $\sup\left\{\sum\limits_{\lambda \in I}x_{\lambda} : I \in \mathscr{S}_{A}\right\} = + \infty$, for each positive integer $N$ there exists a finite subset $I_{0}$ of $A$ such that $N \leq \sum\limits_{\lambda \in I_{0}}x_{\lambda}$.
Then, for each finite subset $I$ of $A$ such that $I_{0} \subseteq I$, we have
\[N \leq \sum\limits_{\lambda \in I_{0}}x_{\lambda} \leq \sum\limits_{\lambda \in I}x_{\lambda}.\]
It follows that the unordered sum $\sum\limits_{\lambda \in A} x_{\lambda}$ does not converge. 

On the other hand, if  $\sup\left\{\sum\limits_{\lambda \in I}x_{\lambda} : I \in \mathscr{S}_{A}\right\} < + \infty$, then for all positive number $\epsilon$ there is a
finite subset $I_{0}$ of $A$ such that $x - \epsilon < \sum\limits_{\lambda \in I_{0}}x_{\lambda} \leq x$. Thus, 
\[x - \epsilon < \sum\limits_{\lambda \in I_{0}}x_{\lambda} \leq \sum\limits_{\lambda \in I}x_{\lambda} < x + \epsilon \]
holds whenever $I$ is a finite subset of $A$ such that $I_{0} \subseteq I$.
\end{proof}



The results presented in the following characterize the convergence of unordered sums of countable families of elements of a normed space.

\begin{teo}\label{teosum}
In the case where $A= \mathbb{N}$, the unordered sum $\sum\limits_{n \in \mathbb{N}}x_n$ in the normed space $X$ converges to a point $x$ if and only if
the series $\sum\limits_{n = 1}^{\infty} x_{\sigma(n)}$ converges to $x$ for every permutation $\sigma : \mathbb{N} \rightarrow \mathbb{N}$.
\end{teo}

\begin{proof}
Let us consider a permutation $\sigma: \mathbb{N} \rightarrow \mathbb{N}$. For each $\epsilon > 0$, we use the convergence of $\sum\limits_{n \in \mathbb{N}}x_n$ 
to choose a corresponding $I_{0}$ (note that it can be supposed to be nonempty).
If we let $N_0 = \max\{\sigma^{-1}(n) : n \in I_{0}\}$, then for every positive integer $N \geq N_0$ we have $I_{0} \subseteq \sigma(\{1,\dots,N\})$. Thus 
\[\left\|\sum\limits_{n=1}^N x_{\sigma(n)} - x\right\| = \left\|\sum\limits_{n \in \sigma(\{1,\dots,N\})}x_n - x\right\|<\epsilon\]
holds whenever $N$ is a positive integer such that $N \geq N_{0}$.

On the other hand, let us suppose that exists a positive number $\epsilon$ such that for each finite subset $I$ of $\mathbb{N}$ there is another finite subset $\widetilde{I}$ of $\mathbb{N}$ such that
$I \subseteq \widetilde{I}$ but $\left\|\sum\limits_{n \in \widetilde{I}}x_n - x\right\|\geq\epsilon$. Under this assumption, let us show that we can find a permutation $\sigma$ of $\mathbb{N}$ such that
the series $\sum\limits_{n = 1}^{\infty} x_{\sigma (n)}$ does not converges to $x$.
In order to do so, we need to consider the sequence $(s_n)_{n \in \mathbb{N}}$ of partial sums of the series $\sum\limits_{n = 1}^{\infty}x_{n}$. 

In the case where the sequence $(s_n)_{n \in \mathbb{N}}$ does not converges to $x$, if we let $\sigma$ be the identity mapping of $\mathbb{N}$, 
then the series $\sum\limits_{n = 1}^{\infty} x_{\sigma (n)}$ clearly does not converges to $x$.

Now, let us consider the case where the sequence $(s_n)_{n \in \mathbb{N}}$ converges to $x$. Let $n_1 = \min\{n \in \mathbb{N}: \|s_n - x\| < \epsilon\}$, let
$F_1 = \{n \in \mathbb{N}: n \leq n_1\}$, and let $\sigma_1 : F_1 \rightarrow F_1$ be a map given by $\sigma_1(n) = n$ for each $n$.
Using the hypothesis that $\sum\limits_{n \in \mathbb{N}}x_n$ does not converges to $x$, let us choose a finite subset $\widetilde{F_1}$ of positive integers corresponding to $F_1$. 
If $\widetilde{F_1}$ has $m_1$ elements, it follows that 
$n_1 < m_1$ (otherwise we would have $F_1 = \widetilde{F_1}$, which leads to a contradiction).

Suppose that we have already defined two finite sets $F_l$ and $\widetilde{F_l}$ of positive integers, where $F_l$ is properly contained in $\widetilde{F_l}$ and each of them contains $n_{l}$ and $m_{l}$ elements, respectively,  
and a bijection $\sigma_l : F_l \rightarrow F_l$. 
Then, let $n_{l+1} = \min\left\{n > \max \widetilde{F_{l}}: \|s_n - x\| < \epsilon\right\}$, let
$F_{l+1} = \{n \in \mathbb{N}: n \leq n_{l+1}\}$, and let
$\sigma_{l+1} : F_{l+1} \rightarrow F_{l+1}$ be a map defined by letting 
$\sigma_{l+1}(n) = \sigma_l(n)$ for each $n$ in $F_l$, $\sigma_{l+1}(n_l+1)< \cdots < \sigma_{l+1}(m_l)$ an increasing enumeration of $\widetilde{F_{l}} \backslash F_{l}$, and 
$\sigma_{l+1}(m_l+1) < \cdots < \sigma_{l+1}(n_{l+1})$ an increasing enumeration of $F_{l+1}\backslash \widetilde{F_{l}}$. It is easy to check that
$\sigma_{l+1}$ is a bijection.
Again, let us use the hypothesis that $\sum\limits_{n \in \mathbb{N}}x_n$ does not converges to $x$ to choose 
a finite subset $\widetilde{F_{l+1}}$ of positive integers corresponding to $F_{l+1}$. If $\widetilde{F_{l+1}}$ has $m_{l+1}$ elements, it follows that 
$n_{l+1} < m_{l+1}$ (otherwise we would have $F_{l+1} = \widetilde{F_{l+1}}$, which leads to a contradiction). 

At the end, we obtained positive integers $n_1< m_1< \cdots < n_l < m_l < n_{l+1}< m_{l+1} < \cdots$ and bijections $\sigma_l: F_l \rightarrow F_l$ for each $l$ in $\mathbb{N}$,
where $\sigma_{l+1}\restriction_{F_{l}} = \sigma_{l}$.
It is possible to define another bijection $\sigma: \mathbb{N} \rightarrow \mathbb{N}$ such that the identity ${\sigma\restriction}_{F_l} = \sigma_l$ holds for each $l$. Thus, we have
\[\left\|\sum\limits_{n=1}^{m_l}x_{\sigma(n)} - x\right\| = \left\|\sum\limits_{n=1}^{m_l}x_{\sigma_{l+1}(n)} - x\right\| = \left\|\sum\limits_{n \in \widetilde{F_l}}x_{n} - x\right\| \geq \epsilon\] 
for every positive integer $l$.
We conclude that there is a permutation $\sigma : \mathbb{N} \rightarrow \mathbb{N}$ such that the series $\sum\limits_{n = 1}^{\infty} x_{\sigma(n)}$ does not converges to $x$.
\end{proof}

\begin{cor}\label{cor}
In the case where $A$ is a countably infinite set, the unordered sum $\sum\limits_{\lambda \in A}x_\lambda$ in the normed space $X$ converges to a point $x$ if and only if
the series $\sum\limits_{n = 1}^{\infty} x_{\sigma(n)}$ converges to $x$ for any bijection $\sigma: \mathbb{N} \rightarrow A$.
\end{cor}

\begin{proof}
The first part of this proof is completely analogous to the first part of the proof of Theorem \ref{teosum}. Thus, we only need to prove the second part. Let us consider a bijection $\sigma : \mathbb{N} \rightarrow A$.
For every permutation $\pi : \mathbb{N} \rightarrow \mathbb{N}$ we have $\sum\limits_{n=1}^{\infty}x_{\sigma(\pi(n))} = \sum\limits_{n=1}^{\infty}x_{(\sigma \circ \pi)(n)} = x$,
then using Theorem \ref{teosum} we conclude that the unordered sum $\sum\limits_{n \in \mathbb{N}}x_{\sigma(n)}$ converges to $x$. It follows that for each positive number $\epsilon$ there is a finite subset $J_{0}$ of $\mathbb{N}$ such that
\[\left\| \sum\limits_{n \in J}x_{\sigma (n) }- x \right\|< \epsilon\] 
holds whenever $J$ is a finite subset of $\mathbb{N}$ satisfying $J_{0} \subseteq J$.
If we let $I_0 = \sigma (J_{0})$, then for every finite subset $I$ of $A$ such that $I_{0} \subseteq I$, we have
\[\left\| \sum\limits_{\lambda \in I}x_{\lambda}- x \right\| = \left\| \sum\limits_{n \in \sigma^{-1}(I)}x_{\sigma (n)} - x \right\| < \epsilon.\]
\end{proof}

The next result is of great importance for Chapter \ref{cap:gibbs} since it allow us to define the concept of a Gibbs measure.

\begin{cor}\label{svdap}
Let $(x_{k})_{k \in \zd}$ be a family of real numbers indexed by $\zd$.
The unordered sum $\sum\limits_{k \in \zd}x_k$ converges to a real number if the limit $\lim\limits_{N \to \infty}\sum\limits_{k \in \Lambda_N}|x_k|$ converges.
\end{cor}

\begin{proof}
Let us suppose that $\lim\limits_{N \to \infty}\sum\limits_{k \in \Lambda_N}|x_k|$ converges.	
Let $\sigma: \mathbb{N} \rightarrow \zd$ be a bijection. For each positive integer $N$, if we define $N_0 = \max\{\|\sigma (n)\| : 1 \leq n \leq N\} + 1$, then we have
\[\sum\limits_{n = 1}^{N}|x_{\sigma(n)}| \leq \sum\limits_{k \in \Lambda_{N_0}}|x_k|.\]
It follows that 
\[\sum\limits_{n = 1}^{N}|x_{\sigma(n)}| \leq \lim\limits_{N \to \infty}\sum\limits_{k \in \Lambda_N}|x_k| < +\infty\]
holds for every positive integer $N$, thus $\sum\limits_{n = 1}^{\infty}|x_{\sigma(n)}| < +\infty$.

Since the series $\sum\limits_{n = 1}^{\infty}x_{\sigma(n)}$ converges absolutely, it follows that the equality
$\sum\limits_{n = 1}^{\infty}x_{\sigma(n)} = \sum\limits_{n = 1}^{\infty}x_{\sigma(\pi(n))}$ holds for every permutation $\pi : \mathbb{N} \rightarrow \mathbb{N}$. 
Thus, if we let $\sigma': \mathbb{N} \rightarrow \zd$ be another bijection
and let $\pi = \sigma^{-1} \circ \sigma'$, we obtain 
\[\sum\limits_{n = 1}^{\infty}x_{\sigma(n)} = \sum\limits_{n = 1}^{\infty}x_{\sigma'(n)}.\]
Hence, Corolary \ref{cor} implies that $\sum\limits_{k \in \zd}x_k$ converges to a real number.
\end{proof}